\documentclass[letterpaper,10pt]{article}%
\usepackage{etex}
\usepackage[utf8]{inputenc}%
\usepackage[T1]{fontenc}%
\usepackage{charter}
\usepackage{textcomp}%
\usepackage{pifont}%
\usepackage[scaled=.92]{helvet}%
\usepackage{eulervm}
\usepackage{euscript}%
\usepackage[margin=20pt,font=small,labelfont=bf,labelsep=period,format=plain]{caption}%

\usepackage{cmll}%
\usepackage{xparse}%

\usepackage{marvosym}%

\DeclareFontFamily{T1}{qzc}{}
\DeclareFontShape{T1}{qzc}{m}{it}{<-> s * [1.2] ec-qzcmi}{}
\DeclareMathAlphabet{\mcal}{\encodingdefault}{qzc}{m}{it}

\usepackage[cmex10]{amsmath}
\usepackage{amsmath,amssymb}
\usepackage{theorem}
\usepackage{graphicx}%
\usepackage[scale={0.72,0.85},nohead]{geometry}
\usepackage{enumitem}%
\usepackage{booktabs}%
\usepackage[usenames,dvipsnames]{xcolor}
\usepackage{xspace}%
\RequirePackage[pdfencoding=auto,unicode=true,linktocpage]{hyperref}

\usepackage{rotating}%
\usepackage{floatpag}%
\rotfloatpagestyle{empty}%
\floatpagestyle{empty}%

\usepackage[usestackEOL]{stackengine}
\usepackage{scalerel}

\usepackage[pdf]{pstricks}
\usepackage{pst-node,pst-tree,pstricks-add}%
\psset{nodesep=3pt,linewidth=.5pt,
	arrowsize=4pt 1,
	arrowlength=.9,
	arrowinset=.4
}

\usepackage[page]{appendix}
\let\appendixpagenameorig\appendixpagename
\renewcommand{\appendixpagename}{\Large\appendixpagenameorig}

\usepackage{prftree}%
\newlength\axiomthickness
\newlength\nonaxiomthickness
\newcommand\linkthickness{.4pt}
\newcommand\cutthickness{\linkthickness}
\prflinethickness=\nonaxiomthickness%
\newcommand\thickaxiomlines{\prflinethickness=\axiomthickness}
\newcommand\nonthickaxiomlines{\prflinethickness=\nonaxiomthickness}
\newcommand\inlinerulelabel[1]{{\sf #1}\@\xspace}
\newcommand\axrulename{ax}
\newcommand\cutrulename{cut}

\newcommand\inlinecutrulelabel{\inlinerulelabel{cut}}
\newcommand\rulelabel[1]{\raisebox{1.5pt}{\small\inlinerulelabel{#1}}}
\newcommand\axrulelabel{\rulelabel{\axrulename}}
\newcommand\cutrulelabel{\rulelabel{\cutrulename}}
\newcommand\unaryrule[2]{\prftree{#1}{\pfseq{#2}}}
\newcommand\drawaxiomrule[1]{\unaryrule{\pfseq{}}{\pfseq{#1}}}
\newcommand\axiomrule[1]{\nonthickaxiomlines\drawaxiomrule{#1}}
\newcommand\axiomruleright[1]{\nonthickaxiomlines\unaryruleright{\axrulelabel}{\pfseq{}}{\pfseq{#1}}}
\newcommand\thickaxiomrule[1]{\thickaxiomlines\drawaxiomrule{#1}\nonthickaxiomlines}
\newcommand\axruleP[1]{\axiomrule{P_{#1}\com\pp_{#1}}}
\newcommand\coaxruleP[1]{\axiomrule{\pp_{#1}\com P_{#1}}}
\newcommand\unaryruleleft[3]{\prftree[l]{$#1$}{#2}{\pfseq{#3}}}
\newcommand\unaryruleright[3]{\prftree[r]{$#1$}{#2}{\pfseq{#3}}}
\newcommand\binaryrule[3]{\prftree{#1}{#2}{\pfseq{#3}}}
\newcommand\binaryruleleft[4]{\prftree[l]{$#1$}{#2}{#3}{\pfseq{#4}}}
\newcommand\binaryruleright[4]{\prftree[r]{$#1$}{#2}{#3}{\pfseq{#4}}}
\newcommand\tightbinaryrule[4]{{\setlength\prfinterspace{#4}\binaryrule{#1}{#2}{#3}}}

\newcommand\forallruleright[2]{\unaryruleright{\forall}{#1}{#2}}
\newcommand\forallruleleft[2]{\unaryruleleft{\forall}{#1}{#2}}
\newcommand\forallrule[2]{\forallruleright{#1}{#2}}
\newcommand\existsruleright[2]{\unaryruleright{\exists}{#1}{#2}}
\newcommand\existsruleleft[2]{\unaryruleleft{\exists}{#1}{#2}}
\newcommand\existsrule[2]{\existsruleright{#1}{#2}}
\newcommand\tensorruleleft[3]{\binaryruleleft{\tensor}{#1}{#2}{#3}}
\newcommand\tensorruleright[3]{\binaryruleright{\tensor}{#1}{#2}{#3}}
\newcommand\tensorrule[3]{\tensorruleright{#1}{#2}{#3}}
\newcommand\parruleright[2]{\unaryruleright{\parr}{#1}{#2}}
\newcommand\parruleleft[2]{\unaryruleleft{\parr}{#1}{#2}}
\newcommand\parrule[2]{\parruleright{#1}{#2}}
\newcommand\existslink[3]{\unaryruleright{{\exists_{#1}}}{#2}{#3}}%

\newsavebox{\cutlabelbox}\sbox{\cutlabelbox}{\cutrulelabel}
\newcommand\cutlabel{\usebox{\cutlabelbox}}
\newcommand\cutrule[3]{\binaryruleright{\cutlabel}{#1}{#2}{#3}}

\newcommand\displayexistsrule{\existsrule{\Gammacom\,\Asub}{\Gammacom\,\ex x A}}

\newcommand\openhyp[1]{\color{white}\prftree{\color{black}\pfseq{#1}}}
\newcommand\raiseproof[1]{{\color{white}\unaryrule{\color{black}#1}{}}}

\def\today{\number\day~%
 \ifcase \month \or January\or February\or March\or April\or May\or June\or
   July\or August\or September\or October\or November\or December\fi\space
 \number\year}

\newcommand\m[1]{\mkern-#1mu}
\newcommand\mm[1]{\mkern#1mu}

\newcommand\nth{^{\text{th}}}

\newcommand\fomll{MLL1\@\xspace}
\newcommand\fomllplus{MLL1$^{\sqcup}$\@\xspace}

\newcommand\foall{ALL1\@\xspace}
\newcommand\onorof{of\@\xspace}

\newcommand\eg{\emph{e.g.}\@\xspace}
\newcommand\ie{\emph{i.e.}\@\xspace}

\newcommand\cf{\emph{cf.}\@\xspace}

\makeatletter%
\newcommand*\etc{%
    \@ifnextchar{.}%
        {etc}%
        {etc.\@\xspace}%
}
\makeatother

\newcommand\defn[1]{{\textit{\textbf{#1}}}}
\newcommand\myitem[1]{\item[\textnormal{(#1)}]}

\renewcommand{\implies}{\Rightarrow}

\makeatletter
\newcommand{\shorteq}{\raisebox{.5pt}{$\mm3\psscalebox{.8 .8}{=}\mm3$}}

\newcounter{BoxedFigCounter}
\newlength{\figindent}
\newlength{\tw}
\newcommand{\boxedfigt}[4]%
{\stepcounter{BoxedFigCounter}\begin{figure}[#1]\vspace{-1.5ex}\begin{flushleft}\rule{5mm}{0.4pt}\hspace*{0.3mm}\raisebox{-3mm}{\rule{0mm}{6mm}}
\raisebox{-0.8mm}{\bf \,#3\,}{\rule{0mm}{6mm}}
\hrulefill{\rule{0mm}{6mm}}
\raisebox{-0.8mm}{\bf \,$\mm1$\emph{\normalsize Figure #2\,}$\m6$}%
\raisebox{-3mm}{\rule{0mm}{6mm}}\hspace*{1.3mm}\rule{5mm}{0.4pt}\raisebox{-3mm}{\rule{0mm}{6mm}}%
\vspace*{-3.4mm}
\begin{tabular}{@{}|@{\hspace{\figindent}}c@{\hspace{\figindent}}|@{}}\begin{minipage}[b]{\tw}\vspace*{5mm}\begingroup
#4\protect\rule{0mm}{1mm}\endgroup\end{minipage}\\
\hline\end{tabular}\end{flushleft}\vspace*{-3.5ex}\end{figure}}

\theoremheaderfont{\scshape}
\theorembodyfont{\normalfont\slshape}
\theoremstyle{plain}

\newtheorem{lemma}{Lemma}
\newtheorem{proposition}{Proposition}

\newtheorem{theorem}{Theorem}

\theorembodyfont{\normalfont}

\newenvironment{proof}{\begin{trivlist}\item{}\normalfont\textit{Proof.}}{\hfill$\square$\end{trivlist}}
\newenvironment{proofof}[1]{\begin{trivlist}\item{}\normalfont\textit{Proof of #1.}}{\hfill$\square$\end{trivlist}}

\DeclareSymbolFont{oldsymbols}{OMS}{cmsy}{m}{n}
\DeclareMathSymbol\wedge\mathbin{oldsymbols}{"5E}%
\DeclareMathSymbol\vee\mathbin{oldsymbols}{"5F}
\DeclareMathSymbol\neg\mathord{oldsymbols}{"3A}

\newcommand\bigleft{\left(\strut}
\newcommand\bigright{\strut\right)}

\newcommand{\tensor}{\otimes}

\newcommand{\gap}[1]{\hspace{#1ex}}

\newcommand{\nohang}[1]{\raisebox{0ex}[\height][0ex]{$#1$}}
\newcommand{\nohangs}[1]{\raisebox{0ex}[1.5ex][0ex]{$#1$}}
\newcommand{\alongbaseline}[1]{\raisebox{0ex}[0ex][0ex]{#1}}

\newcommand{\heightdepth}[3]{\raisebox{0ex}[#1ex][#2ex]{$#3$}}
\newcommand{\pfseq}[1]{\heightdepth{2}{.8}{#1}}

\newcommand\diredgestyle[1]{\psset{arrows=#1,nodesep=2pt,arrowsize=3pt 2,arrowinset=.4,arrowlength=.8,linewidth=.5pt}}
\newcommand\diredgesymb[1]{\raisebox{3pt}{\!\!\begin{psmatrix}[colsep=2.6ex]\rnode{l}{\rule{0pt}{1.2ex}}&\rnode{r}{\rule{0pt}{1.2ex}}%
\diredgestyle{#1}\ncline{l}{r}\end{psmatrix}\!\kern-.5pt}}
\newcommand\toedge{\mathrel{\diredgesymb{->}}}%

\newpsobject{sepline}{ncline}{nodesep=1pt,linestyle=dashed,dash=2pt 2pt,linewidth=1.2pt}

\newcommand\graph{\mcal{G}}
\newcommand\graphof[1]{\graph(#1)}
\newcommand\glambda{\graph(\lambda)}
\newcommand\mllencode[1]{{#1}_{_{^{\textsf{m}\m5}}}}
\newcommand\frameof[1]{\mllencode{#1}}

\newcommand{\dual}[1]{\overline{#1}}
\newcommand{\ddual}[1]{\raisebox{0pt}[0ex][0ex]{\(\dual{\dual{#1}}\)}}
\newcommand{\dualop}{(\overline{\rule{0ex}{1.5ex}\rule{1.2ex}{0ex}})}

\newcommand{\shortmapsto}{\psscalebox{.8 1}{\ensuremath\mapsto}}

\renewcommand{\ll}{l\m2'}

\newcommand\Aparr{A\m2\parr}

\renewcommand\v[1]{\vspace*{#1ex}}
\newcommand\vv[1]{\vspace*{-#1ex}}
\newcommand\h[1]{\hspace*{#1ex}}
\newcommand\hh[1]{\hspace*{-#1ex}}

\newcommand\assign{\mm1\shortmapsto\mm1}
\renewcommand\gets[2]{#1\assign #2}
\newcommand\openU{[\mm2}
\newcommand\closeU{\mm2]}
\newcommand\unifier[1]{\openU #1 \closeU}
\newcommand\assignment[1]{\unifier{#1}}
\newcommand\substitute[1]{\unifier{#1}}

\newcommand\substituteto[2]{\m1\unifier{\m2\gets{#1}{#2}\m2}}%

\newcommand\ueg{\openU \gets v x, \gets w {gu} , \gets y \hza \closeU}%

\newcommand\xx{\dot x}%

\newcommand\treestyle{\psset{nodesepA=3pt,nodesepB=3pt,arrows=-}}
\newcommand\dirtreestyle{\psset{nodesepA=3pt,nodesepB=3pt,arrows=->}}

\newcommand\treexy[5]{%
\rput(#1,0){\rput[B](0,0){\Rnode{#3}{\nohang{#4}}}\rput(0,#2){#5}}}
\newcommand\treey[4]{%
\treexy{0}{#1}{#2}{#3}{#4}}
\newcommand\leaf[3]{%
\rput[B](#1,0){\Rnode{#2}{\nohang{#3}}}}

\newcommand\linkarmseps[5]{\ncbar[angleA=90,arm=#1pt,nodesepA=#2pt,nodesepB=#3pt,linewidth=\linkthickness]{#4}{#5}}

\newcommand\linkarm[3]{\linkarmseps{#1}{2}{2}{#2}{#3}}
\newcommand\link[2]{\linkarm{5}{#1}{#2}}

\newcommand\cutlinkarm[3]{\ncbar[angleA=-90,arm=#3,nodesep=1.5pt,linewidth=\cutthickness]{#1}{#2}}
\newcommand\cutlink[2]{\cutlinkarm{#1}{#2}{6pt}}

\newcommand\lnode[2]{\psDefBoxNodes{#1}{#2}}
\newcommand\drawgirlink[3]{\ncbar[angleA=90,nodesep=3.5pt,arm=0,linewidth=#3]{#1:tl}{#2:tr}}
\newcommand\thingirlink[2]{\drawgirlink{#1}{#2}{\nonaxiomthickness}}
\newcommand\girlink[2]{\drawgirlink{#1}{#2}{\axiomthickness}}

\newcommand\chunkcut[2]{\ncbar[angleA=-90,nodesep=5.78pt,arm=0,linewidth=\nonaxiomthickness]{#1:tl}{#2:tr}}

\newcommand\thetaprime{\theta\m1'}

\newcommand\sigmaprime{\sigma'}

\newcommand\ex[1]{\exists\mm{0.2} #1\mm1}
\newcommand\all[1]{\forall\mm{0.2} #1\mm1}

\newcommand\aligntop[1]{\vtop{\vskip 0pt \vskip -\ht\strutbox #1 \vskip -\dp\strutbox}}%

\newcommand\column[1]{\begin{tabular}{@{}c@{}}#1\end{tabular}}

\newcommand\GnamedForallBit[4]{%
{
\unaryrule{\psDefBoxNodes{#1}{\nohangs{#3#4}}}{\psDefBoxNodes{#2}{\nohangs{\all #4#3#4}}}
}}

\newcommand\GnamedTensorBit[7]{%
\newcommand\leftLit{\nohangs{#3#7}}\newcommand\rightLit{\nohangs{#4#5#7#7}}\newcommand\tensoredLits{\nohangs{\leftLit\tensor\rightLit}}%
\unaryrule
{
  \tightbinaryrule
  {\psDefBoxNodes{#1}{\leftLit}}
  {\psDefBoxNodes{#2}{\rightLit}}
  {\tensoredLits}
  {1.9ex}
}
{\nohangs{\ex #6(#3#6\tensor #4#5#6#6)}}}

\newcommand\GnamedExistsBit[6]{%
{\unaryrule%
  {\psDefBoxNodes{#1}{\nohangs{#3#4#5#5}}}
  {\psDefBoxNodes{#2}{\nohangs{\ex #6#3#6}}}}}

\newcommand\gx{\dot x}
\newcommand\gxx{\ddot x}
\newcommand\gxxx{\hat x}
\newcommand\gxxxx{\tilde x}

\newcommand\GexponentialChunk[9]{%
\begin{array}[t]{@{}c@{\;\;}c@{\;\;}c@{}}
  \aligntop{\GnamedForallBit{forallLit}{#1}{#3}{#8}}
  &
  \aligntop{\GnamedTensorBit{tensorLeftLit}{tensorRightLit}{#4}{#5}{#7}{#8}{#8}}
  &
  \aligntop{\GnamedExistsBit{existsTop}{#2}{#6}{#7}{#8}{#9}}
  \thingirlink{forallLit}{tensorLeftLit}
  \thingirlink{tensorRightLit}{existsTop}
\end{array}}

\newcommand\colouredaxiomrule[2]{\color{#1}\thickaxiomrule{\color{black}#2}}

\newcommand\BigEgLeftColor{blue}
\newcommand\BigEgRightColor{red}
\newcommand\bone{u}
\newcommand\btwo{v}
\newcommand\bthree{w}
\newcommand\bfour{x}
\newcommand\bfive{y}
\newcommand\bsix{z}%
\newcommand{\tone}{(g\bone,f\bfour,a)}
\newcommand{\tonep}{(g\bone,f\btwo,a)}
\newcommand{\tonepp}{(g\bone,f\bfour,a)}
\newcommand{\tonexx}{(\bthree,f\bfour,a)}
\newcommand\hza{h(\bsix,\m2 a)}
\newcommand{\ttwo}{(\hza)}
\newcommand{\ttwop}{(\bfive)}
\newcommand\bigproofeg{
\existsrule
  {
      \tensorrule
      {
           \forallrule
           {
                \existsrule
                {
                     \forallrule
                     {
                          \existsrule
                          {
                             \colouredaxiomrule\BigEgLeftColor{\pp \tone\com\, \mm5 P \tone}
                          }
                          {\ex \btwo \pp \tonep\com \, \mm5 P \tone}
                     }
                     {\ex \btwo \pp \tonep\com \, \mm5 \all \bfour P \tonepp}
                }
                {\ex \btwo \pp \tonep\com \, \mm5 \ex \bthree \all \bfour P \tonexx}
           }
           {\all \bone \ex \btwo \pp \tonep\com \, \mm5 \ex \bthree \all \bfour P \tonexx}
      }
      {
         \colouredaxiomrule\BigEgRightColor{Q \ttwo\com \, \mm5 \qq\ttwo}
      }
      {\all \bone \ex \btwo \pp \tonep\com \, \mm5 \bigleft \ex \bthree \all \bfour P \tonexx \bigright \tensor Q \ttwo\com \, \mm5 \qq \ttwo}
  }
  {\all \bone \ex \btwo \pp \tonep\com \, \mm5 \bigleft \ex \bthree \all \bfour P \tonexx \bigright \tensor Q \ttwo\com \, \mm5 \ex \bfive \qq \ttwop}
}

\newcommand\quadraticSeqCalcCutElimEgConclusionTo[1]{\ovdash (\cdots((P_1\tensor P_2)\tensor P_3)\cdots )\tensor P_{#1}\com \pp_1\com \ldots \com \pp_{#1}}
\newcommand\quadraticSeqCalcCutElimEgConclusion{\quadraticSeqCalcCutElimEgConclusionTo{n}}
\newcommand\quadraticSeqCalcPiSubFour{
  \tensorrule{
    \tensorrule{
      \tensorrule{\axruleP1}{\axruleP2}{P_1\tensor P_2\com \pp_1\com\pp_2}
    }
    {\axruleP3}
    {(P_1\tensor P_2)\tensor P_3\com \pp_1\com \pp_2 \com \pp_3}
  }
  {\axruleP4}
  {((P_1\tensor P_2)\tensor P_3)\tensor P_4\com \pp_1\com \pp_2 \com \pp_3\com \pp_4}
}
\newcommand\quadraticSeqCalcPiSubn{
  \begin{array}{c}
    \hspace{-10ex}
    \quadraticSeqCalcPiSubFour
    \hspace{15ex}
    \\[.5ex]
    \rput(-1.1,0){\psline[linewidth=1.5pt,linestyle=dotted]{-}(-.1,.2)(.1,-.2)}
    \\[1ex]
    \hspace{10ex}
    \tensorrule{
      \tensorrule{\quadraticSeqCalcCutElimEgConclusionTo{n-2}}{\axruleP{n-1}}{\quadraticSeqCalcCutElimEgConclusionTo{n-1}}
      }
      {\axruleP n}
      {\quadraticSeqCalcCutElimEgConclusion}
    \hspace{-8ex}
  \end{array}
}
\newcommand\quadraticSeqCalcCutElimEg{\begin{center}\begin{math}
  \hspace{-15ex}
  \begin{array}{c}
  \cutrule{
    \cutrule{
        {\unaryrule{\Pi_n}{\quadraticSeqCalcCutElimEgConclusionTo{n}}}
    }
    {\coaxruleP1}
    {\quadraticSeqCalcCutElimEgConclusion}
  }
  {\coaxruleP2}
  {\quadraticSeqCalcCutElimEgConclusion}
  \hspace{5ex}
  \\[.5ex]
  \rput(.5,0){\psline[linewidth=1.5pt,linestyle=dotted]{-}(-.1,.2)(.1,-.2)}
  \\[1ex]
  \hspace{25ex}
  \cutrule{
    {\quadraticSeqCalcCutElimEgConclusion\hspace{6ex}}
  }
  {\coaxruleP n}
  {\quadraticSeqCalcCutElimEgConclusion}
  \end{array}\hspace{-6ex}\end{math}\end{center}}

\newcommand\biguneteg{
\rule{0ex}{4ex}
\forall \bone\,\exists \btwo\,
\Rnode{pp}{\pp}(g\bone,f\btwo,a)
\gap8
\bigleft\exists\bthree\,\forall\bfour\,
\Rnode{p}{P}(\bthree,f\bfour,a)\bigright
\mm3\tensor\mm4
\Rnode{q}{Q}(h(\bsix,a))
\gap8
\exists\bfive\, \Rnode{qq}{\qq}(\bfive)
\psset{nodesepA=2pt,nodesepB=2pt}
\ncbar[angle=90,arm=7pt,linecolor=\BigEgLeftColor]{pp}{p}
\ncbar[angle=90,arm=7pt,linecolor=\BigEgRightColor]{qq}{q}
}

\newcommand\idot\circ%
\newcommand\expsmallleft{\mm1(\cdots((c\idot x_0)\idot x_1)\cdots\m2)\idot x_n}
\newcommand\expsmallright{\mm1x_n\idot(x_{n-1}\idot(\cdots(x_0\idot c)\cdots\m2))}
\newcommand\expsmallegseq{\ex x_0 \ex x_1 \cdots \ex x_n\mm2
\bigleft
\Rnode{P}{P}(\expsmallleft)
\,\parr\,
\Rnode{pp}{\pp}(\expsmallright)
\bigright
}
\newcommand\expsmalleg{
\rule{0ex}{3ex}
\expsmallegseq
\psset{linecolor=darkgreen}
\link{P}{pp}
}

\newcommand\subbox[2]{%
\begin{center}\begin{math}#2\end{math}

\v5#1\v3\end{center}}

\newcommand\girhza{h(\optionalUnderline z,\m2 a)}

\newcommand\optionalUnderline[1]{#1}%
\newcommand\biggirardnet[3]{%
\renewcommand\ttwo{(\girhza)}
\raiseproof{
  \forallrule
  {
    \existslink{#1}
    {\openhyp{\lnode{pp}{\pp \tone}}}
    {\ex \btwo \pp \tonep}
  }
  {\all \bone \ex \btwo \pp \tonep}
}
\h5
      \tensorrule
      {
                \existslink{#2}
                {
                     \forallrule
                     {
                        \openhyp{\lnode{p}{P \tone}}
                     }
                     {\all \bfour P \tonepp}
                }
                {\ex \bthree \all \bfour P \tonexx}
      }
      {
         \openhyp{\lnode{q}{Q \ttwo}}
      }
      {\bigleft \ex \bthree \all \bfour P \tonexx \bigright \tensor Q \ttwo}
\h5
  \existslink{#3}
  {\openhyp{\lnode{qq}{\qq \ttwo}}}
  {\ex \bfive \qq \ttwop}
\psset{linecolor=\BigEgLeftColor}
\girlink{pp}{p}
\psset{linecolor=\BigEgRightColor}
\girlink{q}{qq}
}

\newcommand\bigoldgirardnet{\biggirardnet{\phantom{\dot y}}{\phantom{gy}}{\phantom{\girhza}}}

\newcommand\figcomparison{\begin{figure*}
\subbox{An \fomll proof}{\bigproofeg}
\v5
\subbox{Its Girard net}{\bigoldgirardnet}
\v5
\subbox{Its unification net}{\biguneteg}
\v4
\caption{An \fomll proof, its Girard net, and its unification net.\label{fig:comparison}\label{fig:translation-eg}}
\end{figure*}}

\newcommand\elim[2]{
\begin{array}{c}
\rnode{a}{} \\[#1ex]
\rnode{b}{}
\ncline[arrows=->]a b
\naput{#2}
\end{array}}
\newcommand{\lastconc}{\ex z\alongbaseline{$\big(\strut
    \Rnode p Pz\mm1\tensor(\m2\Rnode Q\qq z\parr\Rnode q Qz\m1)\m2\big)$}}
\newcommand\subouty{{\color{blue} y}}
\newcommand\subinfx{{\color{red}f x}}

\newcommand\girnonlocaleg{%
\begin{center}\psscalebox{.9}{\begin{math}
\begin{array}{c}%
{\raiseproof
  {
   \forallrule  {\pfseq{\lnode x{\pp fx}}}
                     {\all x\mm1\pp fx}
  }
  {}
}
\h6
\binaryrule
  {
   \existsrule%
                           {\pfseq{\lnode y {Pfx}}}
                           {\ex y P y}
  }
  {
   \forallrule     {\pfseq{\lnode Y {\pp\subouty}}}
                        {\all\subouty\pp\subouty}
  }
  {}
\h6
\existsrule%
  {
   \tensorrule
     {\pfseq{\lnode z{P\subouty}}}
     {\parrule
       {\pfseq{\lnode Q{\qq\subouty}\phantom{{}\parr{}}\lnode{QQ}{Q\subouty}}}
       {\qq\subouty\parr Q\subouty}
     }
     {P\subouty\tensor (\qq\subouty\parr Q\subouty)}
  }
  { \ex z \left(Pz \tensor (\qq z \parr Q z)\right) }
\thingirlink{x}{y}
\thingirlink{Q}{QQ}
\thingirlink{Y}{z}
\\[-1ex]
\elim{4}{\put(-92,6){global substitution \;\;\;$\subouty\mapsto\subinfx$}}
\h{10}
\\[3ex]
\h{.5}
{\raiseproof
  {
   \forallrule    {\pfseq{\lnode x{\pp fx}}}
                       {\all x\mm1\pp fx}
  }
  {}
}
\h7
{\raiseproof
  {
   \binaryrule
     {\pfseq{\lnode y {Pfx}\;\;}}
     {\pfseq{\;\;\;\lnode Y {\pp \subinfx}}}
     {}
  }
  {}
}
\h7
\existsrule%
  {
   \tensorrule
     {\pfseq{\lnode z{P \subinfx}}}
     {\parrule
       {\pfseq{\lnode Q{\qq \subinfx}\phantom{{}\parr{}}\lnode{QQ}{Q \subinfx}}}
       {\qq \subinfx\parr Q \subinfx}
     }
     {P \subinfx\tensor (\qq \subinfx\parr Q \subinfx)}
  }
  { \ex z \left(Pz \tensor (\qq z \parr Q z)\right) }
\thingirlink{x}{y}
\thingirlink{Q}{QQ}
\thingirlink{Y}{z}
\\[-3ex]
\elim{3}{}
\h{10}
\\[4ex]
\renewcommand\subinfx{fx}%
\h{.5}
{\raiseproof
  {
   \forallrule    {\pfseq{\lnode x{\pp fx}}}
                       {\all x\mm1\pp fx}
  }
  {}
}
\h7
\phantom{
{\raiseproof
  {
   \binaryrule
     {\pfseq{\lnode y {Pfx}\;\;}}
     {\pfseq{\;\;\;\lnode Y {\pp \subinfx}}}
     {}
  }
  {}
}
}%
\h7
\existsrule%
  {
   \tensorrule
     {\pfseq{\lnode z{P \subinfx}}}
     {\parrule
       {\pfseq{\lnode Q{\qq \subinfx}\phantom{{}\parr{}}\lnode{QQ}{Q \subinfx}}}
       {\qq \subinfx\parr Q \subinfx}
     }
     {P \subinfx\tensor (\qq \subinfx\parr Q \subinfx)}
  }
  { \ex z \left(Pz \tensor (\qq z \parr Q z)\right) }
\thingirlink{x}{z}
\thingirlink{Q}{QQ}
\end{array}
\end{math}}\end{center}}

\newcommand{\figcutelimcomp}{\begin{figure*}\begin{center}%
\v2{Girard's cut elimination is not local:}\v3
\girnonlocaleg
\v{7}{Unification net cut elimination is local:}\v{1}
\unetslocalegDisplayedSyntacticSequents
\v2
\caption{Cut elimination comparison.%
\label{fig:cutelimcomp}\label{fig:cut-elim-comparison}}\end{center}\end{figure*}}

\newcommand{\unetslocalegDisplayedSyntacticSequents}{\begin{center}\v4\begin{math}\begin{array}{c}
\all x\Rnode x{\pp}fx
\h5
{
\ex y}\Rnode y P y
\h5
{
\all y}\Rnode Y{\pp}y
\h5
\lastconc
\link x y
\link Y p
\link Q q
\cutlink y Y
\\[1ex]
\elim{2.5}{}%
\h{9.5}
\\[5ex]
\all x\Rnode x{\pp}fx
\h5
\phantom{\ex y{}}\Rnode y P y
\h5
\phantom{\all y{}}\Rnode Y{\pp}y
\h5
\lastconc
\link x y
\link Y p
\link Q q
\cutlink y Y
\\[1ex]
\elim{2.5}{}
\h{9.5}
\\[5ex]
\all x\Rnode x{\pp}fx
\h5
\phantom{\ex y{}\Rnode y P y
\h5
\all y{}\Rnode Y{\pp}y}
\h5
\lastconc
\link x p
\link Q q
\cutlink y Y
\\[1ex]
\end{array}
\end{math}\end{center}}%

\newcommand\uniformityseqvar[1]{\ex x \pp x\,\com\,\ex #1 P #1}

\newcommand\uniformityseqtwosidedvar[1]{\all x P x\mm4\vdash\mm4\ex #1 P#1}
\newcommand\uniformityseqtwosidedinlinevar[1]{\all x P x\vdash\ex #1 P#1}

\newcommand\uniformityseqtwosidedinline{\uniformityseqtwosidedinlinevar x}
\newcommand\additionalproved{\pp\com\,\ex x P}
\newcommand\splitredundancyproof[3]{%
\existsrule
  {\existsruleleft
    {\axiomrule{\ovdash \pp #1\,\com \, P #2}}
    {\ovdash \ex x \pp x\,\com \, P#2}
  }
{\uniformityseqvar{#3}}
}
\newcommand\othersplitredundancyproof[3]{%
\existsruleleft
  {\existsrule
    {\axiomrule{\ovdash \pp #1\,\com \, P #2}}
    {\ovdash \pp #1\,\com \, \ex #3 P #3}
  }
{\uniformityseqvar{#3}}
}
\newcommand\redundancyproof[2]{\splitredundancyproof{#1}{#1}{#2}}
\newcommand\otherredundancyproof[2]{\othersplitredundancyproof{#1}{#1}{#2}}

\newcommand\additionalredundancyproof{
\existsrule
{\axiomrule{\pp\com \, P}}
{\additionalproved}
}

\newcommand\redundancygirardnet[2]{%
\existsruleleft%
  {
    {\pfseq{\lnode x {\pp #1}}}
  }
  {\ex x \pp x}
\h4
\existsrule%
  {
    {\pfseq{\lnode y {P #1}}}
  }
  {\ex #2 P #2}
\thingirlink x y}

\newcommand\additionalredundancygirardnet[1]{%
\raiseproof
   {\pfseq{\lnode y \pp}}
\h4
\existslink{#1}
  {
    {\pfseq{\lnode x P}}
  }
  {\ex x P}
\thingirlink y x}

\newcommand\nonredundantunificationnet[1]{%
\pfseq{\ex x\Leaf* P x    \h4    \ex{#1} \Leaf P {#1}}
\axlink P {Pdual} 4}
\newcommand\additionalredundancyunet{%
\pfseq{\Leaf* P \h4 \ex x \Leaf P}
\axlink P {Pdual} 4}

\newcommand{\figadditionalredundancy}{\begin{figure*}%
\v3\begin{center}\begin{math}
\begin{array}{c@{\h9}c@{\h9}c}
\additionalredundancyproof
&
\additionalredundancygirardnet{t}
&
\additionalredundancyunet
\\[3ex]
\mbox{\begin{tabular}{@{}c@{}}
The unique cut-free\\proof of\\$\additionalproved$
\end{tabular}}
&
\mbox{\begin{tabular}{@{}c@{}}
Infinitely many cut-free\\
\cite{Gir96} nets $G_t^{96}$ of\\
$\additionalproved$\end{tabular}}
&
\mbox{\begin{tabular}{@{}c@{}}
The unique cut-free
\\
unification net
of\\
$\additionalproved$
\end{tabular}}
\\[2ex]
\end{array}\end{math}\end{center}%
\caption{Additional redundancy in the \cite{Gir96} variant of Girard \fomll nets.\label{fig:girard-additional-redundancy}}\hrulefill\end{figure*}}

\newcommand{\figredundancy}{
\begin{figure*}
\begin{center}\begin{math}
\def\onseq{\uniformityseqtwosidedvar x}
\begin{array}{c@{\h9}c@{\h9}c}\\[-2ex]
\redundancyproof t x
&
\redundancygirardnet t x
&
\nonredundantunificationnet x \\[1.5ex]
\mbox{\begin{tabular}{@{}c@{}}
Infinitely many (one-sided)\\cut-free proofs $\Pi_t$ of\\$\onseq$\\(one per witness $t$)
\end{tabular}}
&
\mbox{\begin{tabular}{@{}c@{}}Infinitely many cut-free\\Girard nets $G_t$
\onorof\\$\onseq$\\(one per witness $t$)\end{tabular}}
&
\mbox{\begin{tabular}{@{}c@{}}The unique cut-free
\\
unification net
\onorof\\
$\onseq$\\{}%
\end{tabular}}
\\[-1.5ex]
\end{array}\end{math}\end{center}
\caption{Illustrating unification net canonicity.\label{fig:redundancy}\label{fig:canonicity}}\hrulefill\end{figure*}}

\newcommand\leftall[2]{\forall #1\mm2 \Rnode{#2}{\pp} #1}
\newcommand\rightex[2]{\exists #1\mm2 \Rnode{#2}{P} #1}
\newcommand\extensorbit[3]{\exists {#1}(\Rnode{#2}{P}#1\tensor\Rnode{#3}{\pp} f{#1}{#1})}
\newcommand\chunk[4]{\leftall{#1}{#3}\h3
\extensorbit{#1}{left}{right}
\h3 \rightex{#2}{#4}\link{#3}{left}\link{right}{#4}}

\newcommand{\figexp}{\begin{sidewaysfigure*}\begin{center}%
\scalebox{.9}{%
\begin{math}\begin{array}{@{}c@{}}
\\[0ex]
\textbf{\,Cut elimination of Girard nets blows up exponentially:}
\\[7ex]
\GexponentialChunk{left1}{right1}{\pp}{P}{\pp}{P}{f}{x}{\gx}
\h2
\GexponentialChunk{left2}{right2}{\pp}{P}{\pp}{P}{f}{\gx}{\gxx}
\chunkcut{right1}{left2}
\h2
\GexponentialChunk{left3}{right3}{\pp}{P}{\pp}{P}{f}{\gxx}{\gxx}
\chunkcut{right2}{left3}
\h2
\GexponentialChunk{left4}{right4}{\pp}{P}{\pp}{P}{f}{\gxxx}{\gxxxx}
\chunkcut{right3}{left4}
\\[5ex]
\elim{8}{\raisebox{2pt}{\begin{math}\begin{array}{@{}l@{}}
\text{Non-local cut reductions}\\
\text{substituting }\gx\mm2\shortmapsto fxx\text{, }\gxx\mm2\shortmapsto f\gx\gx\text{, }\gxxx\mm2\shortmapsto f\gxx\gxx
\end{array}\end{math}}}
\\[8ex]
\begin{array}[t]{@{\;}c@{\;\;}c@{\;\;}c@{\;\;}c@{\;\;}c@{\;\;}c@{\;}}
  \hspace{-5ex}
  \aligntop{\GnamedForallBit{forallLit1}{forall1}{\pp}{x}}
  &
  \aligntop{\GnamedTensorBit{tensorLeftLit1}{tensorRightLit1}{P}{\pp}{f}{x}{x}}
  \thingirlink{forallLit1}{tensorLeftLit1}
  &
  \aligntop{\GnamedTensorBit{tensorLeftLit2}{tensorRightLit2}{P}{\pp}{f}{\gx}{fxx}}
  \thingirlink{tensorRightLit1}{tensorLeftLit2}
  &
  \aligntop{\GnamedTensorBit{tensorLeftLit3}{tensorRightLit3}{P}{\pp}{f}{\gxx}{ffxxfxx}}
  \thingirlink{tensorRightLit2}{tensorLeftLit3}
  &
  \aligntop{\GnamedTensorBit{tensorLeftLit4}{tensorRightLit4}{P}{\pp}{f}{\gxxx}{f ffxxfxx ffxxfxx}}
  \thingirlink{tensorRightLit3}{tensorLeftLit4}
  &
  \aligntop{\GnamedExistsBit{existsTop5}{existsBot5}{P}{f}{f ffxxfxx ffxxfxx}{\gxxxx}}
  \thingirlink{tensorRightLit4}{existsTop5}
\end{array}
\\[21ex]
\textbf{\,Cut elimination of unification nets is linear time:}
\\[8ex]
\chunk{x}{\gx}{a}{b}
\h3
\chunk{\gx}{\gxx}{c}{d}\cutlink{b}{c}
\h3
\chunk{\gxx}{\gxxx}{e}{f}\cutlink{d}{e}
\h3
\chunk{\gxxx}{\gxxxx}{g}{h}\cutlink{f}{g}
\\[2ex]
\elim{6}{\raisebox{2pt}{Local cut reductions}}%
\\[9ex]
\leftall{x}a
\h3
\extensorbit{x}bc\link a b
\h3
\extensorbit{\gx}de\link c d
\h3
\extensorbit{\gxx}fg\link e f
\h3
\extensorbit{\gxxx}hi\link g h
\h3
\rightex{\gxxxx}j\link i j
\\[5ex]
\end{array}\end{math}\hspace{-8ex}}%
\end{center}%
\caption{Cut elimination complexity comparison.\label{fig:cut-elim-complexity-comparison}}\end{sidewaysfigure*}}

\newcommand{\cutwithwidth}[3]{\Rnode{x}{#1}\hspace{#3}\Rnode{y}{#2}\ncbar[arm=2pt,nodesep=1pt,angle=-90]xy}
\newcommand{\cut}[2]{\cutwithwidth{#1}{#2}{.5ex}}

\newcommand{\cutwith}[5]{\Rnode{x}{#1}#2\;\Rnode{y}{#3}\ncbar[arm=#4pt,nodesep=#5pt,angle=-90]xy}

\newcommand{\unfold}[1]{\widehat{#1}}

\newcommand{\below}[3]{\ncbar[arm=#3,angle=-90,linestyle=none]{#2}{#2}\ncput{#1}}
\newcommand{\putright}[3]{\ncbar[arm=#3,angle=0,linestyle=none]{#2}{#2}\ncput{#1}}

\newcommand{\swc}{\raisebox{.3pt}{$\mcal C\m2'$}}
\newcommand{\swcc}{\raisebox{.3pt}{$\mcal C$}}

\newcommand\depsym{\toedge}
\newcommand\dep[2]{\,\ex{#1}\depsym\all{#2}\,}

\newrgbcolor{darkred}{.65 0 0}
\newrgbcolor{darkgreen}{0 .4 0}
\newrgbcolor{darkblue}{0 0 .7}

\newcommand\tenex[1]{#1^{\tensor}_{\exists}}
\newcommand\encode[1]{\tenex{#1}}

\def\myoverline#1{{\ThisStyle{{%
  \setbox0=\hbox{{$\SavedStyle#1$}}%
  \stackengine{{1\LMpt}}{{$\SavedStyle#1$}}{{\rule{\wd0}{.4\LMpt}}}{O}{c}{F}{F}{S}%
}}}}
\renewcommand{\dual}[1]{{\mkern1mu\myoverline{\mkern-1mu#1\mkern-1.5mu}\mkern1.5mu}}%
\newcommand\pp{{\dual P}}

\newcommand{\qq}{\mkern1mu\dual{\mkern-1mu Q\mkern-1mu}\mkern1mu}
\newcommand{\rr}{\dual R}

\renewcommand{\AA}{\mkern1mu\dual{\mkern-1muA\mkern-1mu}\mkern1mu}
\newcommand{\BB}{\dual{B\mkern-.5mu}\mkern.5mu}
\newcommand{\demorgan}[2]{\mkern1mu\dual{\mkern-1muA\ifthenelse{\equal{#1}{\parr}}{\mkern-2mu}{}#1 B\mkern-1mu}\mkern1mu=\BB#2\ifthenelse{\equal{#2}{\parr}}{\mkern-1mu}{} \AA}

\newcommand\linknodesep{2pt}
\DeclareDocumentCommand\axlink{s m m m o}%
  {\psset{arrows=-,nodesep=\linknodesep,linewidth=\linkthickness,arm=#4pt,angleA=\IfBooleanT{#1}{-}90}\IfNoValueTF{#5}{\ncbar}{\ncbar[#5]}{#2}{#3}}
\DeclareDocumentCommand{\caxlink}{m m o}%
  {\IfNoValueTF{#3}{\axlink{#1}{#2}{11}}{\axlink{#1}{#2}{#3}}[linearc=0.2,nodesep=3pt]}
\DeclareDocumentCommand{\weightpos}{O{5} m m O{WEIGHT}}%
  {\def\@WEIGHTNODE{#4}#3[framesep=0pt,npos=1.#1]%
  {\scalebox{0.85}{\rnode{\@WEIGHTNODE}{$\mkern-2mu{#2}\mkern-3mu$}}}}%
\DeclareDocumentCommand\weight{O{5} m O{WEIGHT}}{\weightpos[#1]{#2}{\ncput*}[#3]}%
\newcommand{\fullweight}{\ncput{\rnode{WEIGHT}{}}\def\@WEIGHTNODE{WEIGHT}}
\DeclareDocumentCommand{\legto}{o m m}%
  {\pnode(!\psGetNodeCenter{#3}\psGetNodeCenter{#2} #2.x #3.y){JOIN}\IfNoValueTF{#1}{\ncline[nodesepA=\linknodesep,nodesepB=0.03pt]}{\ncline[nodesepA=\linknodesep,nodesepB=0.03pt,#1]}{#2}{JOIN}}
\DeclareDocumentCommand{\leg}{o m}{\IfNoValueTF{#1}{\legto{#2}{\@WEIGHTNODE}}{\legto[#1]{#2}{\@WEIGHTNODE}}}

\DeclareDocumentCommand{\Nnode}{o m o}{\Rnode{#2\IfNoValueF{#1}{#1}}{\IfNoValueTF{#3}{#2}{#3}}}%
\DeclareDocumentCommand{\Leaf}{s o m}{\IfBooleanTF{#1}{\Rnode{#3dual\IfNoValueF{#2}{#2}}{\nohang{\dual #3}}}{\Rnode{#3\IfNoValueF{#2}{#2}}{\nohang{#3}}}}%
\DeclareDocumentCommand{\Tensor}{o}{\Nnode[#1]{Tensor}[\mkern1mu\tensor\mkern1mu]}
\DeclareDocumentCommand{\Par}{o}{\Nnode[#1]{Par}[\mkern1mu\parr\mkern1mu]}

\newcommand\ovdash{}%
\newcommand\comsymbol{\MVComma}
\newcommand\com{\mbox{\kern-1.5pt\comsymbol\kern-1.1pt}}

\newcommand\plinkcolour{red}%
\newcommand\qlinkcolour{blue}%
\newcommand\rlinkcolour{Green}%

\newcommand\drinkerarrow{\ex x(P x\!\implies\! \all y Py)}

\newcommand\sig[1]{\unifier{\gets{x_{#1}}{t_{#1}}}}

\newcommand\twoseqs{\begin{center}\v1\begin{math}
\cutrule
  {\axiomrule{P fx\com \pp fx}}
  {
   \existsrule
     {\axiomrule{P fx\com \pp fx}}
     {P fx\com \ex z P z}
  }
  {P fx\com \;\cutwith{\pp}{fx\;}{P}{5}{2}fx\com \;\ex z \pp z}
\h{20}
\cutrule
  {\axiomrule{P fx\com \pp fx}}
  {
   \existsrule
     {\axiomrule{P fx\com \pp fx}}
     {P fx\com \ex z P z}
  }
  {P fx\com \;\cutwith{\pp}{y\;}{P}{5}{2}y\com \;\ex z \pp z}
\end{math}\v3\end{center}}

\newcommand{\twoencodedseqs}{\begin{center}\v1\begin{math}
\existsrule{
 \tensorrule
  {\axiomrule{P fx\com \pp fx}}
  {
   \existsrule
     {\axiomrule{P fx\com \pp fx}}
     {P fx\com \ex z P z}
  }
  {P fx\com \;\pp fx\tensor Pfx\com \;\ex z \pp z}
}
{Pfx\com\;\ex x \pp fx\tensor Pfx\com\;\ex z \pp z}
\h{20}
\existsrule{
 \tensorrule
  {\axiomrule{P fx\com \pp fx}}
  {
   \existsrule
     {\axiomrule{P fx\com \pp fx}}
     {P fx\com \ex z P z}
  }
  {P fx\com \;\pp fx\tensor P fx\com \;\ex z \pp z}
}
{Pfx\com\;\ex y \pp y\tensor Py\com\;\ex z \pp z}
\end{math}\v1\end{center}}

\colorlet{dblue}{black!15!blue}
\colorlet{dred}{black!15!red}
\colorlet{dgreen}{black!45!green}
\colorlet{dorange}{black!60!orange}
\colorlet{dpurple}{blue!45!red}
\colorlet{dmix}{dblue!45!dgreen!85!}
\colorlet{dveryshared}{dpurple!50!dred}
\colorlet{dbrown}{black!35!brown}

\newcommand\identityEg[3]{%
\psset{nodesepA=#1pt,nodesepB=#2pt}
(\ex x\Rnode x{\pp} x)
\parr\mm2
(\all y\m1\Rnode y P y)
\ncbar[linecolor=dblue,angle=90,arm=#3pt]{x}{y}}%

\newcommand\identityEgForMultiplicativeEncoding[3]{
\psset{nodesepA=#1pt,nodesepB=#2pt}
\,({\color{blue}\ex x}\Rnode x{\pp} x)
\,\parr\,
({\color{red}\all y}\m1\Rnode y P y)
\ncbar[angle=90,arm=#3pt]{x}{y}}%

\newcommand\identityMultiplicativeEncoding{
({\color{blue}\Rnode 1 Q\tensor} \Rnode 2 \pp)\,\parr\,({\color{red}\Rnode 3 \qq\parr }\Rnode 4 P)
\ncbar[linecolor=dpurple,angle=90,arm=5pt] 1 3
\ncbar[angle=90,arm=10pt] 2 4
}

\newcommand\prenexvdashinline[2]{A\mm2 #1 \mm2 \ex x B\,\vdash\,\ex #2 (A\mm2 #1\mm2 B)}
\newcommand\goodprenexvdashinline[1]{\prenexvdashinline{\tensor}{#1}}
\newcommand\badprenexvdashinline[1]{\prenexvdashinline{\parr}{#1}}

\newcommand\prenexinline[3]{\Rnode{pp}\pp\mm2 #1\mm2 \all x \Rnode{qq}\qq x\mm1\com\mm3\ex #3(\Rnode P P #2 \Rnode Q Q#3)}
\newcommand\goodprenexinline[1]{\prenexinline\parr\tensor{#1}}

\newcommand\prenexgraph[3]{%
\prenexgraphcore{#1}{#2}{#3}
\dirtreestyle
\ncline[nodesepA=1pt]{pp}{in}
\ncline[nodesepA=3pt]{allx}{in}
\ncline[nodesepA=3pt]{qq}{allx}
\ncline[nodesepA=2pt,nodesepB=3pt]{P}{out}
\ncline[nodesepA=1pt,nodesepB=3pt]{Q}{out}
\ncline[nodesepA=4pt]{out}{ex}
}
\newcommand\prenexgraphcore[3]{%
\def\gap{1.3}
\rule{0ex}{16.5ex}
\treexy{-\gap}{.9}{in}{#1}{
  \leaf{-.5}{pp}{\pp}
  \treexy{.5}{.9}{allx}{\all x}{
    \leaf{0}{qqx}{\Rnode{qq}{\qq}\makebox[0pt][l]{\(x\)}}
  }
}
\treexy{\gap}{.9}{ex}{\ex #3}{
  \treey{.9}{out}{#2}{
    \leaf{-.5}{P}{P}
    \leaf{.5}{Qx}{\rnode{Q}{Q}#3}
  }
}
\axlink{pp}{P}{12}[linecolor=dblue]
\axlink{qq}{Q}{22}[linecolor=dred]
}

\newcommand\goodprenexgraph[1]{\prenexgraph{\parr}{\tensor}{#1}}
\newcommand\badprenexgraph[1]{\prenexgraph{\tensor}{\parr}{#1}}

\newcommand\goodprenexleapgraph[1]{\goodprenexgraph{#1}\ncline[arrows=->]{ex}{allx}}
\newcommand\badprenexleapgraph[1]{\badprenexgraph{#1}\ncline[arrows=->]{ex}{allx}}

\newcommand\goodprenexswitching[1]{\prenexgraphcore{\parr}{\tensor}{#1}
\treestyle
\ncline{ex}{allx}
\ncline[nodesepA=1pt]{pp}{in}
\ncline[nodesepA=2pt,nodesepB=3pt]{P}{out}
\ncline[nodesepA=1pt,nodesepB=3pt]{Q}{out}
\ncline[nodesepA=4pt]{out}{ex}
}
\newcommand\badprenexswitching[1]{\prenexgraphcore{\tensor}{\parr}{#1}
\treestyle
\ncline{ex}{allx}
\ncline[nodesepA=1pt]{pp}{in}
\ncline[nodesepA=3pt]{allx}{in}
\ncline[nodesepA=2pt,nodesepB=3pt]{P}{out}
\ncline[nodesepA=4pt]{out}{ex}
}

\newcommand\prenexdisplayed[3]{\Rnode{pp}\pp \mm3 #1 \mm3 \all x \Rnode{qq}\qq x\h4\ex #3(\Rnode P P \mm2 #2 \mm2 \Rnode Q Q#3)}
\newcommand\goodprenexdisplayed[1]{\prenexdisplayed\parr\tensor}
\newcommand\badprenexdisplayed[1]{\prenexdisplayed\tensor\parr}

\newcommand\prenexlinking[3]{
\prenexdisplayed{#1}{#2}{#3}
\rule{0ex}{5ex}
\axlink{pp}{P}{8}[linecolor=dblue]
\axlink{qq}{Q}{14}[linecolor=dred]
}
\newcommand\goodprenexlinking[1]{\prenexlinking{\parr}{\tensor}{#1}}
\newcommand\badprenexlinking[1]{\prenexlinking{\tensor}{\parr}{#1}}

\newcommand\prenexlinkingPreExists{
\Rnode{pp}\pp\mm1\com\, \Rnode{qq}\qq x\mm1\com\,\Rnode P P \tensor \Rnode Q Qx
\updownpadding{5}{2}
\axlink{pp}{P}{3.5}[linecolor=dblue,nodesep=1.5pt]
\axlink{qq}{Q}{8}[linecolor=dred,nodesep=1.5pt]
}

\newcommand\prenexlinkingPreForall[1]{
\Rnode{pp}\pp\mm1\com\, \Rnode{qq}\qq x\mm1\com\mm3\ex #1(\Rnode P P \tensor \Rnode Q Q#1)
\updownpadding{5}{2}
\axlink{pp}{P}{3.5}[linecolor=dblue,nodesep=1.5pt]
\axlink{qq}{Q}{8}[linecolor=dred,nodesep=1.5pt]
}

\newcommand\prenexlinkingPrePar[1]{
\Rnode{pp}\pp\mm1\com\, \all x \Rnode{qq}\qq x\mm1\com\mm3\ex #1(\Rnode P P \tensor \Rnode Q Q#1)
\updownpadding{5}{2}
\axlink{pp}{P}{3.5}[linecolor=dblue,nodesep=1.5pt]
\axlink{qq}{Q}{8}[linecolor=dred,nodesep=1.5pt]
}

\newcommand\prenexlinkingConclusion[1]{
\goodprenexinline{#1}
\updownpadding{5}{1}
\axlink{pp}{P}{3.5}[linecolor=dblue,nodesep=1.5pt]
\axlink{qq}{Q}{8}[linecolor=dred,nodesep=1.5pt]
}

\newcommand\mainegseqinline[5]{%
\ex #1(\pp #1\tensor Q#2)\com\,\qq#2\com\,Pf#4
}

\newcommand\updownpadding[2]{\raisebox{0ex}[#1ex][#2ex]{}}

\newcommand\mainexampleunetinline[5]{%
\updownpadding 5 3
\renewcommand\gap{\hspace{4ex}}
\ex #1 ( \Rnode{pp1}{\pp} #1 \tensor \Rnode{Q1}{Q} #2 )
\gap
\Rnode{qq2}{\qq} #3 \,\Rnode{par}{\parr}\, \all{#4}\mm2 \Rnode{P2}{P}f#4
\gap
\Rnode{Q3}{Q} #3 \,\Rnode{tensor}{\tensor}\, \ex{#5}\mm2 \Rnode{pp3}{\pp}f#5
\gap
\Rnode{qq4}{\qq} #3
\gap
\Rnode{P5}{P} fz
\axlink{P2}{pp1}{12}[linecolor=dpurple]
\axlink{Q1}{qq2}{6}[linecolor=dred]
\axlink*{par}{tensor}{8}[linecolor=dbrown]
\axlink{Q3}{qq4}{6}[linecolor=dgreen]
\axlink{P5}{pp3}{12}[linecolor=dblue]
}

\newcommand\abstracteg{\,\all x\mm2 Px \vdash \ex x\mm2 Px}

\newcommand\firstwitness{a}
\newcommand\secondwitness{fc}

\newcommand\Gammacom{\Gamma\mkern-1mu\com}%

\newcommand\trackingstyle{\psset{arrows=-,nodesep=1pt,linewidth=.2pt,linecolor=dblue,arrowsize=2.5pt,arrowinset=.6}}

\newcommand\Asub{A\substituteto x t}
\newcommand\Bsub{B\substituteto y u}

\newcommand\commgap{\hspace{4ex}}
\newcommand\commraise{3.7ex}
\newcommand\comm{\commgap\raisebox{\commraise}{$\leftrightarrow$}\commgap}
\newcommand\commsup[1]{\commgap\raisebox{\commraise}{$\overset{\raisebox{4pt}{\makebox[0ex]{\small\(#1\)}}}{\leftrightarrow}$}\commgap}

\newcommand{\ttcomm}{\newcommand{\lefthyp}{A}
\newcommand{\midhyp}{\pfseq{B\com C}}
\newcommand{\righthyp}{\pfseq D}
\newcommand{\conc}{A\tensor B\com C\tensor D}
\tensorruleright
  {\tensorruleleft{\lefthyp}{\midhyp}{A\tensor B\com C}}
  {\righthyp}
  {\conc}
\comm
\tensorruleleft
   {\lefthyp}
   {\tensorruleright{\midhyp}{\righthyp}{B\com C\tensor D}}
   {\conc}
}

\newcommand\ppcomm{\renewcommand\gap{}%
\def\hyp{\pfseq{A\com B\com C\com D}}
\def\conc{\Aparr B\com C\parr D}
\parruleright
  {\parruleleft{\hyp}{\Aparr B\com C\com D}}
  {\conc}
\comm
\parruleleft
  {\parruleright{\hyp}{A\com B\com C\parr D}}
  {\conc}
}

\newcommand\eecomm{
\def\hyp{\pfseq{\Asub\com \Bsub}}
\def\conc{\ex x A\com \ex y B}
\existsruleright
  {\existsruleleft{\hyp}{\ex x A\com \Bsub}}
  {\conc}
\comm
\existsruleleft
  {\existsruleright{\hyp}{\Asub\com \ex y B}}
  {\conc}
}

\newcommand\aacomm{
\def\hyp{\pfseq{A\com B}}
\def\conc{\all x A\com \all y B}
\forallruleright
  {\forallruleleft{\hyp}{\all x A\com B}}
  {\conc}
\comm
\forallruleleft
  {\forallruleright{\hyp}{A\com \all y B}}
  {\conc}
}

\newcommand\tpcomm{
\def\lefthyp{\pfseq{A\com B\com C}}
\def\righthyp{\pfseq{D}}
\def\conc{\Aparr B\com C\tensor D}
\tensorruleright
  {\parruleleft{\lefthyp}{\Aparr B\com C}}
  {\righthyp}
  {\conc}
\comm
\parruleleft
  {\tensorruleright{\lefthyp}{\righthyp}{A\com B\com C\tensor D}}
  {\conc}
}

\newcommand\tecomm{
\def\lefthyp{\pfseq{\Asub\com B}}
\def\righthyp{\pfseq{C}}
\def\conc{\ex x A\com B\tensor C}
\tensorruleright
  {\existsruleleft{\lefthyp}{\ex x A\com B}}
  {\righthyp\;\;\:}
  {\conc}
\comm
\existsruleleft
  {\tensorruleright{\lefthyp}{\righthyp}{\Asub\com B\tensor C}}
  {\conc}
}

\newcommand\tecommcontexts{
\def\lefthyp{\pfseq{\Gammacom\,\Asub\com\, B}}
\def\righthyp{\pfseq{C\com\,\Delta}}
\def\conc{\Gammacom\,\ex x A\com\, B\tensor C\com\,\Delta}
\tensorruleright
  {\!\!\!\existsruleleft{\lefthyp}{\!\!\!\!\Gammacom\,\ex x A\com\, B\!\!\!\!}}
  {\!\!\!\righthyp}
  {\;\;\conc\;\;}
\comm
\existsruleleft
  {\tensorruleright{\lefthyp}{\righthyp}{\Gammacom\,\Asub\com\, B\tensor C\com\,\Delta}}
  {\conc}
}

\newcommand\tacomm{
\def\lefthyp{\pfseq{A\com B}}
\def\righthyp{\pfseq{C}}
\def\conc{\all x A\com B\tensor C}
\tensorruleright
  {\forallruleleft{\lefthyp}{\all x A\com B}}
  {\righthyp}
  {\conc}
\commsup{x\!\not\in\! C}
\forallruleleft
  {\tensorruleright{\lefthyp}{\righthyp}{A\com B\tensor C}}
  {\conc}
}

\newcommand\eacomm{
\def\hyp{\pfseq{\Asub\com B}}
\def\conc{\ex x A\com \all y B}
\forallruleright
  {\existsruleleft{\hyp}{\ex x A\com B}}
  {\conc}
\commsup{y\!\not\in\! t}
\existsruleleft
  {\forallruleright{\hyp}{\Asub\com \all y B}}
  {\conc}
}

\newcommand\apcomm{
\def\hyp{\pfseq{A\com B\com C}}
\def\conc{\all x A\com B\parr C}
\parruleright
  {\forallruleleft{\hyp}{\all x A\com B\com C}}
  {\conc}
\comm
\forallruleleft
  {\parruleright{\hyp}{A\com B\parr C}}
  {\conc}
}

\newcommand\epcomm{
\def\hyp{\pfseq{\Asub\com B\com C}}
\def\conc{\ex x A\com B\parr C}
\hspace{-2ex}
\parruleright
  {\existsruleleft{\hyp}{\ex x A\com B\com C}}
  {\conc}
\comm
\existsruleleft
  {\parruleright{\hyp}{\Asub\com B\parr C}}
  {\conc}
\hspace{-2ex}
}

\newcommand\commfig{\newcommand\ygap{9ex}\begin{figure*}\begin{center}\begin{math}
\setlength{\prfinterspace}{1.85ex}
\begin{array}{@{}c@{\hspace{10.5ex}}c@{}}
\ttcomm
&
\ppcomm
\\[\ygap]
\eecomm
&
\aacomm
\\[\ygap]
\tpcomm
&
\eacomm
\\[\ygap]
\tacomm
&
\apcomm
\\[\ygap]
\tecomm
&
\epcomm
\\[3ex]
\end{array}
\end{math}\end{center}\caption{\label{fig:rule-commutations}\label{fig:commutations}\fomll rule commutations.
Each is a local rewrite in a proof, left-to-right or right-to-left, with arbitrary passive side-formulas in the context;
see main text for details.
In the $\exists/\forall$ commutation $y$ may not occur in $t$
and in the $\forall/\tensor$ commutation $x$ may not occur free in $C$
(precluding malformed $\forall$-rules on the right).
The lower three commutations involving a $\tensor$ rule have a symmetric variant exchanging the left and right hypotheses (omitted to save space).}\end{figure*}}

\definecolor{shadecolour}{gray}{.8}

\newcommand\fullshade[1]{\colorbox{shadecolour}{$\m5#1\m5$}}

\newcommand\expTermZero{c\idot c}
\newcommand\expTermOne{(\expTermZero)\idot(\expTermZero)}
\newcommand\expTermTwo{(\expTermOne)\idot(\expTermOne)}
\newcommand\expTermThree{(\expTermTwo)\idot(\expTermTwo)}

\newcommand\linkspark[1]{\;\Rnode x{\,} \hspace{1.8ex} \Rnode y {\,} \ncbar[angle=90,arm=4pt,nodesep=1pt,linecolor=#1] x y\;}
\newcommand\cutspark[1]{\;\Rnode x{\,} \hspace{1.8ex} \Rnode y {\,} \ncbar[angle=-90,arm=4pt,nodesep=-4pt,linecolor=#1] x y\;}

\newcommand\blowupArgsOne{c}
\newcommand\blowupArgsTwo{\blowupArgsOne,\expTermZero}
\newcommand\blowupArgsThree{\blowupArgsTwo,\expTermOne}
\newcommand\blowupArgsFour{\blowupArgsThree,\expTermTwo}
\newcommand\blowupAxiomRule[1]{\axiomrule{\pp(#1)\com\,P(#1)}}
\newcommand\blowupAxiomRuleBig[1]{\axiomrule{\pp\bigleft#1\bigright\com\,P\bigleft#1\bigright}}

\newcommand\blowupAxiomRuleOne{\blowupAxiomRule{\blowupArgsOne}}
\newcommand\blowupAxiomRuleTwo{\blowupAxiomRule{\blowupArgsTwo}}
\newcommand\blowupAxiomRuleThree{\blowupAxiomRuleBig{\blowupArgsThree}}
\newcommand\blowupAxiomRuleFour{\blowupAxiomRuleBig{\blowupArgsFour}}
\newcommand\blowupLeftAtom{\psDefBoxNodes{leftatom}{\pp\bigleft\blowupArgsFour\bigright}}
\newcommand\blowupRightAtom{\psDefBoxNodes{rightatom}{P\bigleft\blowupArgsFour\bigright}}
\newcommand\blowupExistsFourFormula[1]
{\ex{#1}\,P\bigleft\blowupArgsThree,#1\bigright}
\newcommand\blowupExistsThreeFormula[1]
{\ex{#1}\,\pp(\blowupArgsTwo,#1,#1\idot #1)}
\newcommand\blowupExistsTwoFormula[2]
{\ex{#1}\,\ex{#2}\,P(\blowupArgsOne,#1,#1\idot #1,#2)}
\newcommand\blowupExistsOneFormula[2]
{\ex{#1}\,\ex{#2}\,\pp(#1,#1\idot #1,#2,#2\idot #2)}
\newcommand\shortestBlowupProofFour{\shortestBlowupProofFourVars v x y z}
\newcommand\shortestBlowupProofFourVars[4]{
\hspace{-10ex}
 \existsruleleft{
  \existsruleright{
   \existsruleleft{
    \existsruleright{\blowupAxiomRuleFour}{\blowupLeftAtom\com\, \blowupExistsFourFormula{#4}}
   }{\blowupExistsThreeFormula{#3}\com\, \blowupExistsFourFormula{#4}}
  }{\blowupExistsThreeFormula{#3}\com\, \blowupExistsTwoFormula{#2}{#4}}
 }{\blowupExistsOneFormula{#1}{#3}\com\, \blowupExistsTwoFormula{#2}{#4}}
\hspace{-10ex}
}

\newcommand\blowupGirardNetLeftVars[2]{
 \existsruleleft{
  \existsruleleft{\blowupLeftAtom}
  {\blowupExistsThreeFormula{#2}}
 }
 {\blowupExistsOneFormula{#1}{#2}}
}
\newcommand\blowupGirardNetRightVars[2]{
 \existsruleright{
  \existsruleright{\blowupRightAtom}
  {\blowupExistsFourFormula{#2}}
 }
 {\blowupExistsTwoFormula{#1}{#2}}
}
\newcommand\blowupGirardNetVars[4]{
\hspace{-10ex}
\blowupGirardNetLeftVars{#1}{#3}
\hspace{1ex}
\blowupGirardNetRightVars{#2}{#4}
\thingirlink{leftatom}{rightatom}
\hspace{-10ex}
}

\newcommand\blowupGirardNet{\blowupGirardNetVars v x y z}

\newcommand\blowupUnetVars[4]{
\ex v\, \ex y\, \Rnode{pp}{\pp}(v,v\idot v,y,y\idot y)
\hspace{6ex}
\ex x\, \ex z\, \Rnode{P}{P}(c,x,x\idot x,z)
\axlink{pp}{P}{8}
}

\newcommand\blowupUnet{\blowupUnetVars v x y z}

\newcommand\blowupFormulaLeftVars[2]{\ex{#1}\,\ex{#2}\,\pp(#1,#1\idot#1,#2,#2\idot#2)}
\newcommand\blowupFormulaRightVars[2]{\ex{#1}\,\ex{#2}\,P(\blowupArgsOne,#1,#1\idot#1,#2)}

\newcommand\blowupSequentInlineVars[4]{
\blowupFormulaLeftVars{#1}{#3}\com\,
\blowupFormulaRightVars{#2}{#4}
}

\newcommand\blowupSequentInline{\blowupSequentInlineVars v x y z}

\newcommand\fibrestyle{\psset{linestyle=dotted,linewidth=.5pt}}

\newcommand\drinkerCps{\begin{math}
\newcommand\vx[1]{\cnode*(0,0){.09}{##1}}
\newcommand\ovx[1]{\cnode(0,0){.10}{##1}}
\newcommand\svx[1]{\fnode[framesize=.19]{##1}}
\newcommand\lvx[3]{\vx{##1}\rput[t](##1){\rule{0ex}{##2}##3}}
\newcommand\rvx[2]{\vx{##1}\rput[l](##1){\;\:\,##2}}
\newcommand\x{-1}
\renewcommand\xx{0}
\newcommand\xxx{1}
\newcommand\xxxx{2}
\rput(0,2.225){%
  \rput(\x,.45){\vx q}
  \rput(\x,0){\vx r}
  \rput(\xx,.73){\psset{linecolor=dblue}\ovx b}
  \rput(\xxx,-.28){\vx c}
  \rput(\xxxx,0){\psset{linecolor=dblue}\ovx d}
}
\rput(0,0){%
  \rput(\x,0){\lvx w {2.3ex} x}
  \rput(\xx,.3){\rvx v {\raisebox{.3ex}{\(\pp x\)}}}
  \rput(\xxx,-.33){\lvx x {2.3ex} y}
  \rput(\xxxx,0){\lvx y {2.6ex} {P y}}
}
{\fibrestyle\psset{nodesep=2pt}
\ncline b v
\ncline q r
\ncline r w
\ncline c x
\ncline d y}
\psset{nodesep=0pt}
\ncline b q
\ncline r c
\ncline r d
\ncline v w
\ncline w x
\ncline w y
\h{42}
\rput(0,2){%
  \Rnode{ex1Top}{\exists}
  x
  \,
  \Rnode{ppTop}{\pp}
  x
  \hspace{3ex}
  \Rnode{ex2Top}{\exists}
  x
  \,
  \Rnode{allTop}{\forall}
  y
  \,
  \Rnode{PTop}{P}
  y
}
\rput(0,0){%
  \Rnode{exBot}{\exists}
  x
  \bigleft
  \,
  \Rnode{ppBot}{\pp}
  x
  \;\,
  \Rnode{vee}{\vee}
  \;\,
  \Rnode{allBot}{\forall}
  y
  \,
  \Rnode{PBot}{P}
  y
  \bigright
}
\psset{arrows=->,nodesep=3pt,linewidth=.3pt}
\ncline{ex1Top}{exBot}
\ncline{ppTop}{ppBot}
\ncline{ex2Top}{exBot}
\ncline{allTop}{allBot}
\ncline{PTop}{PBot}
\axlink{ppTop}{PTop}{8}[linecolor=dblue]
\end{math}}

\newcommand\translate[1]{\lfloor #1\rfloor}

\newcommand\linkingon[2]{#1\mm8\rhd\mm8 #2}
\newcommand\tightlinkingon[2]{#1\mm3\rhd\mm3 #2}

\newcommand\figtranslation{\begin{figure}\begin{center}\vspace{4ex}\v{2.5}\(\setlength{\arraycolsep}{3ex}\hh{17}\begin{array}{ccc}
\axiomruleright{\linkingon{\{\mm2\{\pp,\m1 P\mm2\}\}}{\pp\com P}}
\hspace{-1.6ex}
&
\parrule
  {\linkingon{\theta}{\pfseq{\Gamma\com \, A\com \, B}}}
  {\linkingon{\theta}{\Gamma\com \, A\parr B}}
&
\existsrule
  {\linkingon{\theta}{\pfseq{\Gamma\com \, \Asub}}}
  {\linkingon{\theta}{\Gamma\com \, \ex x A}}
\\[6ex]
&
\tensorrule
  {\linkingon{\theta}{\pfseq{\Gamma\com \, A}}}
  {\linkingon{\phi}{\pfseq{B\com \, \Delta}}}
  {\linkingon{\theta\,\cup\,\phi}{\Gamma\com \, A\tensor B\com \,\Delta}}
&
\forallrule
  {\linkingon{\theta}{\pfseq{\Gamma\com \, A}}}
  {\linkingon{\theta}{\Gamma\com \, \all x A}}
\rput[lb](.3,.3){\text{($x$ not free in $\Gamma$)}}
\end{array}\)\v{2.5}\end{center}\caption{\label{fig:translation}Inductive translation of a cut-free \fomll proof $\Pi$ of $\Gamma$ to a linking $\protect\translate\Pi$ on $\Gamma$. We make two simplifying assumptions (without loss of generality): in the $\tensor$ case $\theta$ and $\phi$ are disjoint, and in the $\exists$ case the leaf vertices of $\ex x A$ and $\protect\Asub$ are the same (only their predicate labels vary, where $x$ becomes $t$).}\hrulefill\end{figure}}

\newcommand{\acut}{\cut{A}{\AA}}

\newcommand\id{\iota}

\newcommand\figprenextranslation{\begin{figure*}\begin{center}\begin{math}
\prftree[r]{\(\parr\)}{
 \prftree[r]{\(\forall\)}{
  \prftree[r]{\(\exists\)}{
   \prftree[r]{\(\tensor\)}
   {
    \prftree{\updownpadding{4}{2}
     \;\Rnode{pp}\pp\com\,\Rnode P P\; \axlink{pp}{P}{3.5}[linecolor=dblue,nodesep=1.5pt]
    }
   }
   {
    \prftree{\updownpadding{4}{2}
     \;\Rnode{qq}\qq x\com\,\Rnode Q Qx\; \axlink{qq}{Q}{3.5}[linecolor=dred,nodesep=1.5pt]
    }
   }
   {
    \prenexlinkingPreExists
   }
  }
  {
   \;\prenexlinkingPreForall y\;
  }
 }
 {
  \;\prenexlinkingPrePar y\;
 }
}
{
\;\prenexlinkingConclusion y\;
}
\hspace{20ex}\rule{0ex}{45ex}
\begin{array}{@{}c@{\hspace{8ex}}c@{}}
\\[-33ex]
\goodprenexlinking y
&
\raisebox{1.5ex}{\column{Unifier (mgu)\\\(\unifier{\gets y x}\)}}
\\[8ex]
\goodprenexleapgraph y
&
\raisebox{8ex}{\column{Graph\\(leap \(\ex y\toedge \all x\))}}
\\[8ex]
\goodprenexswitching y
&
\raisebox{8ex}{\text{A switching}}
\end{array}
\end{math}\end{center}%
\caption{\label{fig:prenex-translation}Illustrating unification net translation (left) and correctness (right). The left side shows the translation of an \fomll proof into a unification net $\theta$. The right side shows the three steps in verifying correctness for $\theta$: constructring the mgu $\sigma=\protect\unifier{\protect\gets y x}$; the graph (with a leap from $\ex y$ to $\all x$ since $\sigma(y)$ contains $x$); one of four switchings, each required to be a tree. See main text for details.}\hrulefill\end{figure*}}

\newcommand\figlinkingunifier{\begin{figure}\begin{center}\v1\begin{math}
\biguneteg
\end{math}\end{center}\caption{\label{fig:linking-unifier}\label{fig:biguneteg}A linking with mgu $\protect\ueg$, hence precedences $\protect\dep v x$ and $\protect\dep w u$.}\hrulefill\end{figure}}

\newcommand\twolinkingsvar[1]{%
\begin{center}\v2\begin{math}
\renewcommand{\gap}{\hspace{4ex}}\psset{labelsep=4ex}
\newcommand{\arm}{5}
\Rnode{p}{P} fx\,\gap\cutwith{\Rnode{pp}{\pp}}{f#1\gap}{\Rnode{p1}{P}}{\arm}{2}f#1\,\gap\ex z \Rnode{pp1}{\pp} z
\ncbar[angle=90,arm=\arm pt,nodesepA=1pt,nodesepB=2pt]{p}{pp}%
\ncbar[angle=90,arm=\arm pt,nodesepA=1pt,nodesepB=2pt]{p1}{pp1}
\h{16}
\Rnode{p}{P} fx\,\gap\cutwith{\Rnode{pp}{\pp}}{y\gap}{\Rnode{p1}{P}}{\arm}{2}y\,\gap\ex z \Rnode{pp1}{\pp} z
\ncbar[angle=90,arm=\arm pt,nodesepA=1pt,nodesepB=2pt]{p}{pp}%
\ncbar[angle=90,arm=\arm pt,nodesepA=1pt,nodesepB=2pt]{p1}{pp1}
\end{math}\v2\end{center}}%
\newcommand\twolinkings{\twolinkingsvar x}%

\newcommand\expunifierfootnote{For example, for an infix binary function symbol $\idot$ the unifier of the system of equations $E_n$ comprising $x_1=c$, $x_2=x_1\idot x_1$,
\ldots, $x_n=x_{n-1}\idot x_{n-1}$ grows exponentially with $n$, since $x_n$ is assigned a term containing $2^n$ copies of the constant $c$.}

\title{\vspace*{-3.5ex}\Large\textbf{%
Unification nets: canonical proof net quantifiers}
\author{\\[-2.5ex]
\large Dominic J.\ D.\ Hughes
\\[1ex]
\small Stanford University \& U.C.\ Berkeley\thanks{%
  I pursued this research as a Visiting Scholar at Stanford then
  Berkeley. I'm grateful to my hosts, Vaughan Pratt (Stanford Computer
  Science), Sol Feferman (Stanford Mathematics) and Wes Holliday
  (Berkeley Logic Group). Thanks to Marc Bagnol, Willem
  Heijltjes and Lutz Stra\ss burger for valuable feedback, and to Dale Miller for
  inviting me to present this work at the LIX Colloquium 2013.  In
memoriam Sol Feferman (1928--2016).}
\date{\vspace{-1ex}\small 21 January 2018}}}

\begin{document}

\maketitle
\protect\thispagestyle{empty}%

\begin{center}\vspace{-5ex}

\begin{minipage}{5.5in}\setlength{\parindent}{3ex}\small%
Proof nets for MLL (unit-free Multiplicative Linear Logic) are concise graphical representations of proofs which are \emph{canonical} in the sense that they abstract away syntactic redundancy such as the order of non-interacting rules.
We argue that Girard’s extension to \fomll (first-order MLL) fails to be canonical because of redundant existential witnesses,
and present canonical \fomll proof nets called \emph{unification nets}
without them.
For example, while there are infinitely many cut-free Girard nets $\abstracteg$, one per arbitrary choice of witness for $\ex x$, there is a unique cut-free unification net,
with no specified witness.

Redundant existential witnesses cause Girard's \fomll nets to suffer from severe complexity issues:
(1) cut elimination is non-local and exponential-time (and -space), and (2) some sequents require exponentially large cut-free Girard nets.
Unification nets solve both problems:
(1) cut elimination is local and linear-time, and
(2) cut-free unification nets grow linearly with the size of the sequent.
Since some unification nets are exponentially smaller than corresponding Girard nets and sequent proofs,
technical delicacy is required to ensure correctness is polynomial-time (quadratic).

These results extend beyond \fomll via a broader methodological insight:
for canonical quantifiers, the standard \emph{parallel/sequential} dichotomy of proof nets
fails; an
\emph{implicit/explicit witness}
dichotomy is also needed.
Work in progress
extends unification nets to additives
and uses them to extend \emph{combinatorial proofs} [\textsl{Proofs without syntax}, Annals of Mathematics, 2006] to classical first-order logic.
\end{minipage}\end{center}

\tableofcontents

\section{Introduction}\label{sec-intro}\vv{-.2}

Girard's
elegant
proof nets \cite{Gir87,DR89} are
concise
graphical representations of proofs
in MLL (unit-free multiplicative linear logic).
For example, the two MLL proofs
\begin{center}\vspace{1ex}\begin{math}
\newcommand\parcomma{\;\com,\;\:\mkern1mu}
\newcommand\preparencomma{\mkern1mu\;\com\,}
\hspace{-8ex}
\tensorrule
  {\tensorruleleft
     {\color{\plinkcolour}\thickaxiomrule{\color{black}\ovdash P\com \pp}}
     {\color{\qlinkcolour}\thickaxiomrule{\color{black}\ovdash Q\com \qq}}
     {\ovdash P\com \,\pp\tensor Q\com \,\qq}
  }
  {\color{\rlinkcolour}\thickaxiomrule{\color{black} \ovdash R\com \rr}}
  {\ovdash P\com \,\pp\tensor Q\com \,\qq\tensor R\com \,\rr}
\hspace{18ex}
\tensorruleleft
  {\color{\plinkcolour}\thickaxiomrule{\color{black}\ovdash P\com \pp}}
  {\tensorrule
     {\color{\qlinkcolour}\thickaxiomrule{\color{black}\ovdash Q\com \qq}}
     {\color{\rlinkcolour}\thickaxiomrule{\color{black}\ovdash R\com \rr}}
     {\ovdash Q\com \,\qq\tensor R\com \,\rr}
  }
  {\ovdash P\com \,\pp\tensor Q\com \,\qq\tensor R\com \,\rr}
\hspace{-8ex}
\end{math}\vspace{1ex}\end{center}
translate to the same MLL proof net:
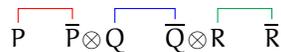
\begin{center}\vspace{2ex}
  \begin{math}\newcommand\armheight{4}\renewcommand\gap{\hspace{3ex}}
    \Leaf P   \gap   \Leaf*{P}\tensor \Leaf Q   \gap   \Leaf* Q\tensor \Leaf R   \gap   \Leaf* R
    \axlink{P}{Pdual}{\armheight}[linecolor=\plinkcolour]
    \axlink{Q}{Qdual}{\armheight}[linecolor=\qlinkcolour]
    \axlink{R}{Rdual}{\armheight}[linecolor=\rlinkcolour]
  \end{math}
\end{center}
MLL proof nets are \emph{canonical} in the sense that they
abstract
away
syntactic redundancy such as the
order of
non-interacting
rules.
{\figredundancy}%
For example,
the two proofs above differ only in the order they introduce
non-interacting tensors $\pp\tensor Q$ and $\qq\tensor R$; the proof net abstracts away this arbitrary choice.
Such
syntactic redundancies are not merely subjective aesthetic failures:
as noted by Girard \cite{Gir96}, they burden sequent calculus cut
elimination with endless mechanical rule commutations.
By purging these commutations, cut elimination for MLL proof nets is local (each reduction being a local graph rewrite) and linear time (eliminating all cuts is linear time in the size of the net). In contrast, cut elimination for MLL sequent calculus is non-local and at best quadratic (Appendix~\ref{sec:mll-seq-calc-cut-elim}).

Girard
extended MLL proof nets with
quantifiers, to \fomll (first-order MLL),
over
a series of
three papers \cite{Gir87,Gir88,Gir91q}.
He reiterated
them
in
\textsl{The Blind Spot} \cite{Gir11},
choosing one
for the cover picture,
and characterizing them as ``\textsl{The only really satisfactory extension of proof-nets}''
(Chapter 11).

At first glance,
they do indeed appear satisfactory:
like the MLL nets they extend,
they
abstract
away the redundant order of non-interacting rules
using \emph{parallelism} \cite{Gir96}.
However,
we argue that they fail to be canonical (hence fail to be satisfactory)
due to redundant existential witnesses, inherited from
sequent calculus.
For example, consider
$\,\uniformityseqtwosidedinlinevar x\,$, whose one-sided form is $\,\uniformityseqvar x\:$.
Figure~\ref{fig:redundancy} (left side) shows an infinite family of cut-free \fomll proofs $\Pi_t$,
one per existential witness term $t\,=\,z$, $a$, $f(z,a)$, $f(g(z),h(a,b))$, \etc.
The choice of $t$ is arbitrary, hence redundant.
Correspondingly, there is an infinite family of cut-free Girard nets $G_t$
(Figure~\ref{fig:redundancy} centre), one per witness term $t$, since Girard nets inherit redundant existential witness directly from sequent calculus.\footnote{The \cite{Gir96} variant introduces additional redundancy not even present in sequent calculus, due to explicit witness annotations $\exists_t$ even for vacuous quantifiers (see App.\,\ref{sec:additional-redundancy}), so by \emph{Girard net} we shall always mean the \cite{Gir91q} and \cite{Gir11} variant.}

Redundant witnesses cause Girard's nets to suffer from two exponential blow-ups (see Section~\ref{sec:complexity}):
(1) cut elimination is non-local and exponential-time (and -space), and (2) some sequents require exponentially large cut-free Girard nets, \ie, cut-free Girard nets are not \emph{polynomially bounded} \cite{CR79},
a serious flaw given that \fomll possesses a polynomially bounded proof system \cite{LS94}.

\subsection{Unification nets}

We present canonical \fomll proof nets called \emph{unification nets}, or \emph{unets} for short, free of redundant existential witnesses.
Figure~\ref{fig:redundancy} (right side) illustrates canonicity: in contrast to the infinite families of cut-free sequent proofs and Girard nets, there is a unique cut-free unification net \onorof $\uniformityseqtwosidedinline$.
By leaving witnesses implicit,
unification nets solve the exponential blow-ups of Girard nets (Section~\ref{sec:complexity}):
(1) cut elimination is local and linear-time, and
(2) cut-free unification nets are only
linearly larger than their underlying sequents, hence they are polynomially bounded (indeed \emph{linearly bounded}).
A quick informal overview of unification nets is provided in Section~\ref{sec:overview} below.

\subsection{Beyond sequentialization}\label{sec:beyond-seq}\label{sec:beyond-sequentialization}

\begin{figure*}
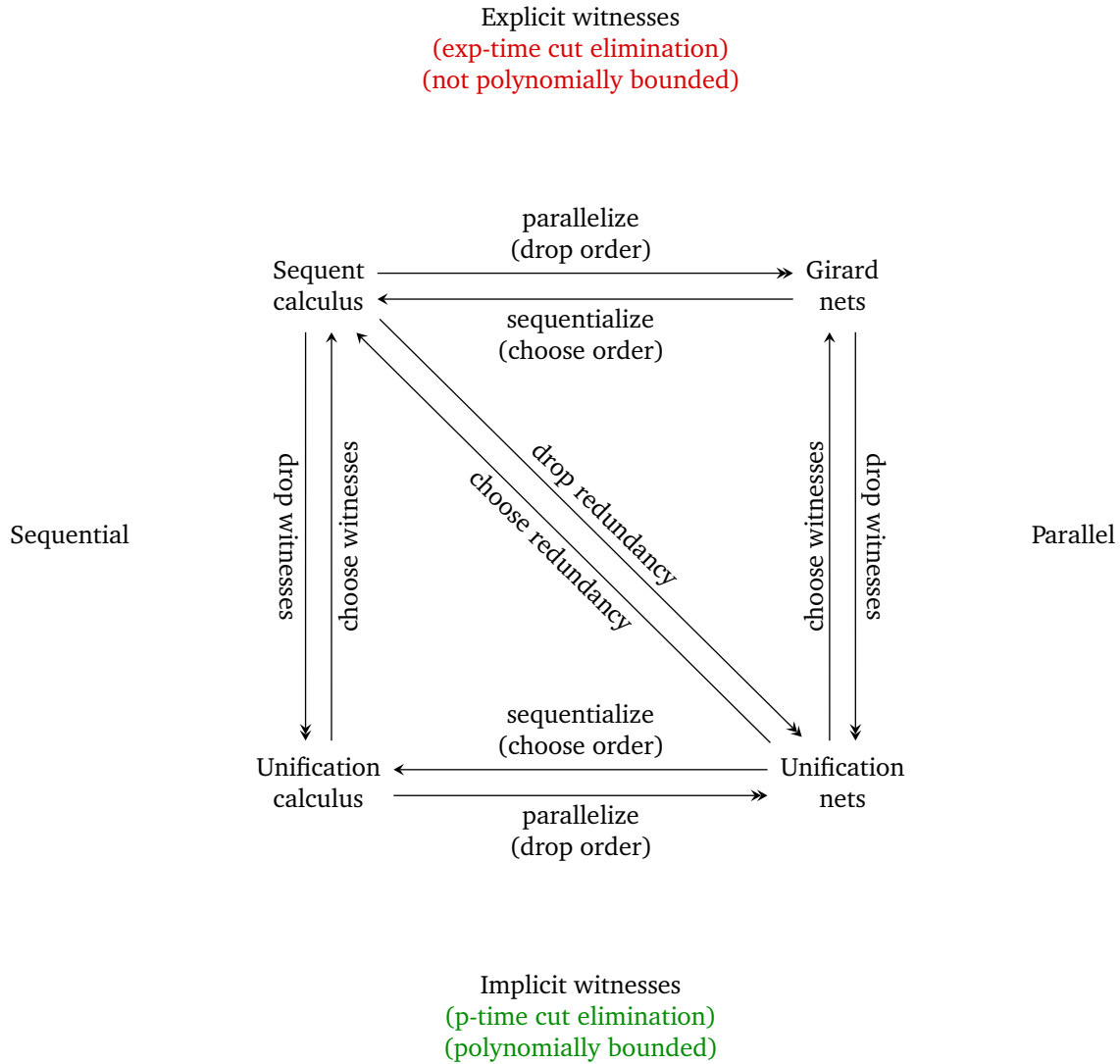

\[
\newcommand\seqlabel{\column{sequentialize\\(choose order)}}
\newcommand\witnessinglabel{\text{choose witnesses}}
\newcommand\deseqlabel{\column{parallelize\\(drop order)}}
\newcommand\dewitnesslabel{\text{drop witnesses}}
\newcommand\parallellabel{\text{Parallel}}
\newcommand\sequentiallabel{\text{Sequential}}
\newcommand\dropredundancy{\text{drop redundancy}}
\newcommand\chooseredundancy{\text{choose redundancy}}
\newcommand\explicitwitnesslabel{\column{Explicit witnesses\\\color{dred}(exp-time cut elimination)\\\color{dred}(not polynomially bounded)}}
\newcommand\implicitwitnesslabel{\column{Implicit witnesses\\\color{dgreen}(p-time cut elimination)\\\color{dgreen}(polynomially bounded)}}
\newcommand\parseqgap{14ex}
\newcommand\impexgap{14ex}
\newcommand\edgenudge{5pt}
\begin{array}{c@{\hspace{32ex}}c}
  \\[20ex]
  \Rnode{sc}{\column{Sequent\\calculus}}     & \Rnode{gn}{\column{Girard\\nets}}
  \\[36ex]
  \Rnode{uc}{\column{Unification\\calculus}} & \Rnode{un}{\column{Unification\\nets}}
  \\[28ex]
\end{array}
\psset{nodesep=5pt,arrows=->>,offset=\edgenudge,labelsep=2pt}
\ncline{sc}{gn}\taput{\deseqlabel}\taput[labelsep=\impexgap]{\explicitwitnesslabel}
\ncline{gn}{un}\aput{:U}{\dewitnesslabel}\trput[labelsep=\parseqgap]{\parallellabel}
\psset{offset=-\edgenudge}
\ncline{sc}{uc}\tlput[labelsep=\parseqgap]{\sequentiallabel}\bput{:U}{\dewitnesslabel}
\ncline{uc}{un}\tbput[labelsep=\impexgap]{\implicitwitnesslabel}\tbput{\deseqlabel}
\psset{arrows=->,offset=\edgenudge}
\ncline{gn}{sc}\tbput{\seqlabel}
\ncline{un}{gn}\aput{:U}{\witnessinglabel}
\psset{offset=-\edgenudge}
\ncline{uc}{sc}\bput{:U}{\witnessinglabel}
\ncline{un}{uc}\taput{\seqlabel}
\psset{offset=\edgenudge}
\ncline[arrows=->>,offsetB=2pt,offsetA=7pt]{sc}{un}\aput{:U}{\dropredundancy}
\ncline[arrows=->,offsetA=8pt,offsetB=3pt,nodesepA=10pt,nodesepB=3pt]{un}{sc}\aput{:D}{\chooseredundancy}
\]
\caption{Relationship between \fomll sequent calculus, Girard nets, and unification nets.
Each double-headed arrow is a surjection between cut-free systems.
The diagram commutes from top-left to bottom-right: the surjection from \fomll sequent calculus proofs to unification nets factorizes in two different ways, either through Girard nets (eliminating order redundancy first), or through unification calculus (eliminating witness redundancy first).
(The reader may look ahead to Figure~\ref{fig:comparison} for an example of translation via an intermediate Girard net.)
\\
\hspace*{3ex}%
The lower two systems are polynomially bounded, while the upper two are not, due to exponential size blow-ups (see Section~\ref{sec:complexity-issues}).
Because of redundant existential witnesses, cut elimination for the upper two systems is exponential-time (and -space); for the lower two systems it is polynomial-time
(also see Section~\ref{sec:complexity-issues})). Cut elimination for unification nets is linear time.\label{fig:beyond-sequentialization}}\end{figure*}Figure~\ref{fig:beyond-sequentialization} summarizes the relationship between \fomll sequent calculus, Girard nets, and unification nets.
The lower-left corner, \emph{unification calculus} (defined in Section~\ref{sec:unification-calculus}) is a variant of \fomll sequent calculus in which, like unification nets, existential witnesses remain implicit;
it was conceived to fill out a commuting square.

Along the east-west axis we have the standard \emph{parallel/sequential} dichotomy of proof nets \cite{Gir96}: sequent calculus and unification calculus are sequential (\emph{west}), including redundant order between non-interacting rules; Girard nets and unification nets are parallel (\emph{east}), abstracting away this redundancy.
Along the north-south axis we have an \emph{implicit/explicit witness} dichotomy: sequent calculus and Girard nets have redundant explicit existential witnesses (\emph{north}); unification calculus and unification nets abstract away this redundancy by leaving witnesses implicit (\emph{south}).

\subsection{Towards combinatorial proofs for classical first-order logic}\label{sec:fopws}

\textsl{Proof without syntax} \cite{Hug06} reformulated classical propositional logic in terms of \emph{combinatorial proofs} rather than syntactic proofs.
A key motivation for the present paper on unification nets was as a stepping stone towards extending combinatorial proofs to classical first-order logic, the subject of a paper in preparation.

A first-order combinatorial proof of
Smullyan's drinker paradox $\drinkerarrow$ is shown below-left.
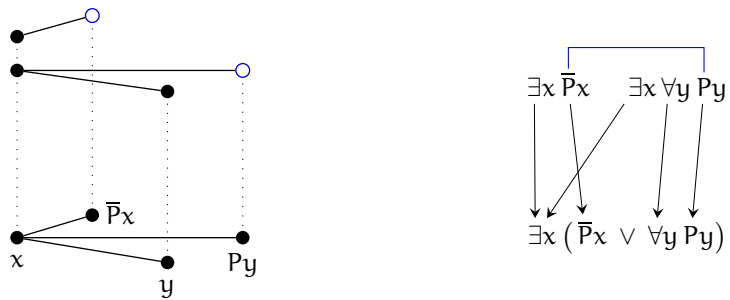
\begin{center}\v{17}\hh{18}\drinkerCps\hspace*{-12ex}\vspace*{5ex}\end{center}
The lower labelled graph abstracts the proved formula $\drinkerarrow$, the upper partially-coloured graph abstracts a unification net, and the dotted lines indicate a \emph{skew fibration}, a lax notion of graph fibration.
As in the original propositional case \cite{Hug06i} (see also \cite{Car10,Str17}), a skew fibration is a simultaneous parallelization of all contraction and weakening in a proof.
By using a \emph{semi-combinatorial} presentation style \cite[\S2.1]{Hug06i}, as top-right, the unification net becomes more apparent.

\subsection{Extending unification nets to additives}\label{sec:towards-additives}%

The exponential size blow-up from explicit witnesses applies not only to \fomll, but far more generally, to quantifier-only sequent calculus (see Section~\ref{sec:complexity}). Thus the blow-up applies to first-order additives.
Current work in collaboration with Willem Heijltjes and Lutz Stra\ss burger extends unification nets to \foall (first-order Additive Linear Logic without units).
Since the examples in Figure~\ref{fig:canonicity} involve no multiplicative connective, they are simultaneously additive. Thus the unification net in Figure~\ref{fig:canonicity} is both an additive and a multiplicative unification net.

\subsection{Canonicity Theorem}\label{sec:intro:canonicity}

The cut-free \fomll proofs
\[
\redundancyproof\firstwitness{x}
\h{20}
\otherredundancyproof\secondwitness{x}
\]
are \emph{equivalent} in the sense that the left yields the right by commuting the order of the $\exists$ rules and replacing one arbitrary choice of existential witness, $\firstwitness$, by another, $\secondwitness$.
While they have distinct Girard nets (because Girard nets inherit redundant explicit exponential witnesses),
they have the same unification net (from Figure~\ref{fig:canonicity}):
\[
\nonredundantunificationnet x\rule{0ex}{5ex}
\]
In Section~\ref{sec:canonicity} we formalize this notion of proof equivalence
and prove a \textsl{Canonicity Theorem} (Theorem~\ref{thm:canonicity}, page~\pageref{thm:canonicity}): two cut-free \fomll proofs are equivalent (modulo rule commutations and re-witnessing) if and only if they have the same unification net.

\subsection{Quick informal overview of unification nets}\label{sec:overview}

An example unification net of the sequent $\mainegseqinline x y y z z$ is shown below, with four \emph{links} $\linkspark{black}$ and one \emph{cut} $\cutspark{black}$ between dual formulas:
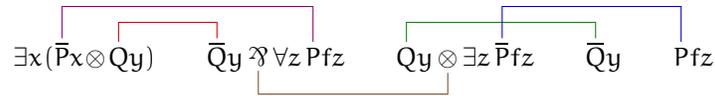
\[ \updownpadding{6}{4} \mainexampleunetinline x y y z z \]
Unlike MLL nets and Girard's \fomll nets, linked predicates need not be strictly dual, \eg\ $\pp x$ and $Pfz$ of the left-most link above.

An \fomll proof translates to a unification net by tracking pairs of dual predicate symbols down from axiom rules. {\figprenextranslation}Figure~\ref{fig:prenex-translation} shows a simple example.\footnote{The sequent is an instance of prenex extrusion $\goodprenexvdashinline x$, which is provable in \fomll\ ($x$ not free in $A$), with the right bound variable $x$ renamed to $y$ to avoid ambiguity.}
An example involving function symbols and ternary predicates is shown in {\figcomparison}Figure~\ref{fig:comparison}, together with the corresponding Girard net for comparison.
Some unification nets are exponentially smaller than corresponding Girards nets and sequent proofs (see Sections~\ref{sec:exp-compression} and \ref{sec:complexity}).

{\figcutelimcomp}Figure~\ref{fig:cut-elim-comparison} contrasts the local cut elimination of unification nets with the non-local cut elimination of Girard nets. Cut elimination for unification nets is linear time (Theorem~\ref{thm:cut-elim-linear}, page~\pageref{thm:cut-elim-linear}), versus exponential time and space for Girard nets (Section~\ref{sec:girard-cut-elim-exponential}).

The correctness criterion for unification nets has three parts, sketched below.
\begin{itemize}
\myitem 1 \emph{Unification.}
The linking must have a \emph{unifier}: an assigment of terms to non-vacuous existential variables making every link dual. For example, $\unifier{\gets y x}$ is a unifier for the linking $\theta$ in Figure~\ref{fig:prenex-translation} (top-right),
since $Q y\unifier{\gets y x}=Q x$, dual to $\qq x$.  As another example, $\ueg$ unifies the linking in Figure~\ref{fig:comparison}.
\myitem 2 \emph{Leaps.} Construct the \emph{graph} with an edge $\ex y\toedge\all x$ called a \emph{leap}\footnote{Leaps perform a similar function to Girard's \emph{jumps} \cite{Gir96}.} whenever the most general unifier (mgu) assigns to $y$ a term containing $x$. Figure~\ref{fig:prenex-translation} shows the graph of $\theta$.
\myitem 3 \emph{Switchings.} Verify every \emph{switching} is a tree, each obtained by deleting all but one edge into every  $\parr$ and $\forall$ and undirecting remaining edges. One of the four switchings of $\theta$ is shown in Figure~\ref{fig:prenex-translation}.
\end{itemize}
See Section~\ref{sec:cut-free-unets} (cut-free unification nets) and Section~\ref{sec:cut} (unification nets with cuts).

\subsection{Technical delicacy required for polynomial-time (quadratic) correctness}

Since some unification nets are exponentially smaller than corresponding Girards nets and sequent proofs (see Sections~\ref{sec:exp-compression} and \ref{sec:complexity}), technical delicacy is required to ensure that correctness is polynomial time. In Theorem~\ref{thm:quadratic} (page~\pageref{thm:quadratic}) we show that it is at worst quadratic time.

The complexity problem is that although \emph{unifiability} can be checked in linear time \cite{MM76}, constructing an actual \emph{unifier} --- in particular, the mgu required for the leaps in the graph --- in general takes exponential time, because the unifier can be exponential in size.\footnote{\expunifierfootnote}
In Section~\ref{sec:cut-free-quadratic-correctness} we show that, with an appropriate choice of linear-time algorithm for checking unifiability, a biproduct of the algorithm provides enough information to construct all leaps in quadratic time, without the need for constructing the mgu explicitly. Thus the correctness of a unification net can be verified in quadratic time (Theorem~\ref{thm:quadratic}, page~\pageref{thm:quadratic}).\footnote{Checking that every switching of a graph is a tree takes only linear time \cite{Gue99}.}

\subsection{Related work}\label{sec:related}

Unification in the context of first-order logic goes back to Herbrand's theorem \cite{Her30}. Robinson's resolution \cite{Rob65} is a seminal work.
Our axiom links between predicates which are not strictly dual (\eg $\pp x$ and $Pfy$) are akin to the first-order connections/matings employed in automated theorem proving \cite{Bib81,And81}. In fact, Bibel in \cite[p.\,4]{Bib81} coined \emph{link} as an alternative term for a connection; we have adopted this terminology.
The roots of first-order connections/matings
with unification can be traced back further to Prawitz \cite{Pra70} and \cite{Qui55}.

Our \emph{leaps} from $\ex x$ vertices to $\all y$ vertices play a similar role to Girard's \emph{jumps} between $\all y$ vertices and occurrences of witnesses containing $y$, but in a more rarefied context without explicit witnesses. Both leaps and jumps capture dependencies between $\forall$ rules and $\exists$ rules in a proof, and the interaction between tensors and quantifiers.

Bellin and van de Wiele \cite{BW95} add a condition on eigenvariables to Girard's \fomll net definition \cite{Gir91q} to streamline kingdoms and empires. Since we leave witnesses implicit, and have no need for eigenvariables, we do not need an anologous condition.

Abstract representations of first-order quantifiers with explicit witnesses for classical logic have been
presented by Heijltjes \cite{Hei10} (extending expansion
trees \cite{Mil84}) and McKinley \cite{McK10}.
Stra\ss burger presents proof nets for second-order MLL in \cite{Str09}.

First-order proof nets with explicit witnesses are employed in linguistic analysis, for example, \cite{Moo02}. It would be interesting to see if any simplication could result from using unification nets instead.

\section{\fomll{} (first-order multiplicative linear logic, without units)}\label{sec:mll1}

As in \cite{Gir91q}, we work with \fomll (first-order multiplicative linear
logic, without units).
We adopt the following conventions:
\begin{center}\renewcommand{\arraystretch}{1.2}\begin{tabular}{c@{\;\;}c@{\;\;}cl@{\h8}cl}
$x$ & $y$ & $z$ & (term) variables
&
$P\;\;Q\;\;R$ & $n$-ary predicate symbols ($n\m3\ge\m3 0$)
\\
$f$ & $g$ & $h$ & $n$-ary function symbols ($n\m3\ge\m3 1$)
&
$A\;\;B\;\;C$ & formulas
\\
$a$ & $b$ & $c$ & constants ($0$-ary function symbols)
&
$\Gamma\;\;\mm1\Delta\mm1\;\;\Sigma$ & sequents
\\
$s$ & $t$ & $u$   & terms
\\
\end{tabular}\end{center}
Fix an arity-preserving \defn{negation} or \defn{duality} function $\dualop$ on
predicate symbols such that
\(\ddual{P}=P\) and $\pp\neq P$ for all
$P$.
A \defn{predicate} or \defn{atom} is an expression
\(Pt_1\ldots t_n\) for any $n$-ary predicate symbol $P$ and
terms $t_i$.
We may insert parentheses to increase readability,
\eg,
\(
Pffy
=
P(ffy)
=
P\left(\strut f(f(y))\right)
\) if
$f$ is a unary ($1$-ary) function symbol.
\defn{Formulas} are generated from atoms by binary connectives tensor $\tensor$ and par
$\parr$ and unary quantifiers $\forall x$ and $\exists x$ for each variable $x$.
Negation extends to formulas by
$\dual{\pi \,t_1\ldots t_n}=\dual\pi\, t_1\ldots t_n$,
$\dual{A\tensor B}\m2=\m2\AA\parr\BB$,
$\dual{\Aparr B}\m2=\m2\AA\tensor\BB$,
$\dual{\ex x A}\m2=\m2\all x{\AA}$, $\dual{\all x A}\m2=\m2\ex x{\AA}$.

\subsection{Sequents as labelled directed forests}

We identify a formula with its parse tree, a directed tree with leaves labelled by atoms and internal vertices by connectives and quantifiers.
A \defn{sequent} is a disjoint union of formulas.
We write comma for disjoint union.
For \mbox{example}, the two-formula sequent \:$\all x \pp
fx\com\,\lastconc$\, is
the labelled directed forest below.
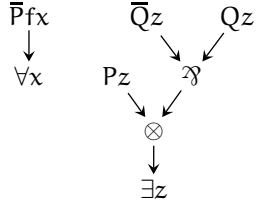
\begin{center}
\(
\newcommand\treecore{
\hh2
\rule{0ex}{16.5ex}
\rput(-.3,1.5){\treexy{0}{.8}{allx}{\forall x}{
  \leaf{0}{Pfx}{\pp fx}
}}
\h{8}
\treexy{0}{.8}{exz}{\exists z}{
  \treexy{0}{.7}{tensor}{\tensor}{
    \leaf{-.5}{Pz}{Pz}
    \treexy{.5}{.8}{par}{\parr}{
      \leaf{-.6}{qz}{\qq z}
      \leaf{.6}{Qz}{Qz}
    }
  }
}
\psset{nodesepB=2pt}
\ncline[nodesepA=2pt]{Pfx}{allx}
\ncline{Pz}{tensor}
\ncline[nodesepA=4pt]{qz}{par}
\ncline{Qz}{par}
\ncline[nodesepA=2pt]{par}{tensor}
\ncline{tensor}{exz}}
\dirtreestyle
\treecore
\)
\vspace*{1ex}
\end{center}

\subsection{Clean sequents}

A sequent or formula is \defn{clean} if all quantifed variables are distinct from each other and from all free variables.
For example, $\,\ex xPx\com\, \all yQzy\,$ is clean but $\,\ex xPx\com\, \all xQ zx\,$ and $\,\ex xPx\com\, Qx\,$ are not.
In a clean sequent, an \defn{existential} (resp.\ \defn{universal}) variable is one bound
by an existential (resp.\ universal) quantifier.
For example, in the sequent \:$\all x\pp fx\com \ex y Qyz$\: the variables
$x$, $y$, and $z$ are universal, existential and free (respectively).
A quantifier is \defn{vacuous}\label{sec:vacuous-quantifier-formula-defn} if it binds no variable. For example,
in $\all x\ex y\all z Pzc$ both $\all x$ and $\ex y$ are vacuous, but $\all z$ is not.

\subsection{\fomll{} rules}

Sequents are proved using the following rules, where $\Asub$ denotes the result of simultaneously substituting the term $t$ for all free occurrences of $x$ in $A$.
\begin{center}\v{2.5}\(\setlength{\arraycolsep}{6ex}\hh{12}\begin{array}{ccc}
\axiomruleright{P\com \pp}\hspace{1.6ex}
&
\parrule
  {\pfseq{\Gamma\com \, A\com \, B}}
  {\Gamma\com \, A\parr B}
&
\existsrule
  {\pfseq{\Gamma\com \, \Asub}}
  {\Gamma\com \, \ex x A}
\\[5ex]
\cutrule
  {\pfseq{\Gamma\com \, A}}
  {\pfseq{\AA\com \, \Delta}}
  {\Gamma\com \,\Delta}\;\;
&
\tensorrule
  {\pfseq{\Gamma\com \, A}}
  {\pfseq{B\com \, \Delta}}
  {\Gamma\com \, A\tensor B\com \,\Delta}
&
\forallrule
  {\pfseq{\Gamma\com \, A}}
  {\Gamma\com \, \all x A}
\rput[lb](.3,.3){\text{($x$ not free in $\Gamma$)}}
\end{array}\)\v{2.5}\end{center}
These are the standard rules for first-order multiplicative linear
logic \cite{Gir87,Gir88,Gir91q}, omitting turnstile $\vdash$
(redundant in a right-sided calculus) and the exchange rule (redundant
since we treat sequents as labelled graphs).
A sequent immediately above a rule is a \defn{hypothesis} of a rule, and the sequent immediately below the rule is its \defn{conclusion}.
The conclusion of a proof is its final sequent (the conclusion of its final rule).

The sub-system without the two quantifier rules is MLL (multiplicative linear logic, without units).

\subsection{Tracking symbols, subterms and leaves through rules}\label{sec:tracking}

Every rule instance
induces a \defn{tracking} function on symbol occurrences, from above to below (a partial function in the case of a \inlinecutrulelabel rule), for example,
\[
\renewcommand\gap{\hspace{6ex}}
\trackingstyle
\newcommand\traceseq[1]{\raisebox{-.1ex}[0ex][1.1ex]{\(#1\)}}
\newcommand\partraceseq[2]{%
\traceseq{
\Rnode{P#1}P
\,\Rnode{t#1}\tensor\,\mm2
\Rnode{all#1}\forall
\Rnode{xx#1}x\mm2
\Rnode{Q#1}Q
\Rnode{f#1}f
\Rnode{x#1}x
\com\;
\Rnode{qq#1}\qq
\Rnode{a#1}a
#2
\Rnode{R#1}R
\Rnode{y#1}y
\Rnode{c#1}c
}}
\parrule{\partraceseq 1 {\com\;}}{\partraceseq 2 \parr}
\ncline{P1}{P2}
\ncline{t1}{t2}
\ncline{all1}{all2}
\ncline{xx1}{xx2}
\ncline{Q1}{Q2}
\ncline{f1}{f2}
\ncline{x1}{x2}
\ncline{qq1}{qq2}
\ncline{a1}{a2}
\ncline{R1}{R2}
\ncline{y1}{y2}
\ncline{c1}{c2}
\gap
\newcommand\tensortraceseq[2]{%
\traceseq{
\Rnode{qq#1}\qq
\Rnode{y1#1}y
\mm1\com\mm2
\Rnode{Q#1}Q
\Rnode{y2#1}y
#2
\Rnode{e#1}\exists
\Rnode{xx#1}x
\,
\Rnode{P#1}P
\Rnode{x#1}x
\Rnode{f#1}f
\Rnode{y3#1}y
\,\com\;
\Rnode{pp#1}\pp\m{1.5}
\Rnode{a#1}a
\Rnode{yy#1}y
}}
\unaryruleright{\tensor}{\tensortraceseq 1 {\hspace{5ex}}}{\tensortraceseq 2 {\,\tensor\,}}
\ncline{Q1}{Q2}
\ncline{y11}{y12}
\ncline{qq1}{qq2}
\ncline{y21}{y22}
\ncline{e1}{e2}
\ncline{xx1}{xx2}
\ncline{P1}{P2}
\ncline{x1}{x2}
\ncline{f1}{f2}
\ncline{y31}{y32}
\ncline{pp1}{pp2}
\ncline{a1}{a2}
\ncline{yy1}{yy2}
\gap
\newcommand\fza[1]{%
\Rnode{f#1}f
\Rnode{z#1}z
\Rnode{a#1}a
}
\newcommand\fzalines[2]{
\ncline{f#1}{f#2}
\ncline{z#1}{z#2}
\ncline{a#1}{a#2}
}
\newcommand\fzaxlines[2]{
\ncline[offsetB=-.5pt]{f#1}{x#2}
\ncline[offsetB=0pt]{z#1}{x#2}
\ncline[offsetB=.5pt]{a#1}{x#2}
}
\newcommand\existstraceseq[2]{%
\traceseq{
\Rnode{pp#1}\pp
\fza{#1}
\mm2\com\,
#2
}}
\existsrule
{\existstraceseq 1
{
\Rnode{P1}P
\fza3
\mm2
\Rnode{t1}\tensor
\mm2
(
\Rnode{Q1}Q
\fza4
\Rnode{par1}\parr
\Rnode{qq1}\qq
\Rnode{h1}h
\fza5
)
}
}
{\existstraceseq 2
{
\Rnode{exists}\exists
\Rnode{boundx}x
\m2\bigleft
\Rnode{P2}P
\Rnode{x1}x
\mm5
\Rnode{t2}\tensor
\mm5
(
\Rnode{Q2}Q
\mm2
\fza6
\,
\Rnode{par2}\parr
\,
\Rnode{qq2}\qq
\Rnode{h2}h
\Rnode{x2}x
)\bigright
\h1
}
}
\ncline{pp1}{pp2}
\fzalines 1 2
\ncline{P1}{P2}
\fzaxlines 3 1
\ncline{t1}{t2}
\ncline{qq1}{qq2}
\fzalines 4 6
\ncline{par1}{par2}
\ncline{Q1}{Q2}
\ncline{h1}{h2}
\fzaxlines 5 2
\]
Tracking is injective except into occurrences of the variable $x$ bound by an $\exists$ rule (see example above-right, where some $f$, $z$ and $a$ occurrences track to the same $x$ occurrence).

Tracking extends to subterms, for example, above-right
the first occurrence of $fza$ above the rule tracks to the first occurrence of $fza$ below, the second occurrence of $fza$ above tracks to the first bound occurrence of $x$ below, the last occurrence of $fza$ above tracks to the last bound occurrence of $x$ below, and $hfza$ above tracks to $hx$ below.

The tracking of propositional variable occurrences doubles as a tracking of sequent leaves, since leaves are in bijection with propositional variable occurrences.
Leaf tracking is a partial injective function for the \inlinecutrulelabel\ rule, but is otherwise a bijection between leaves above the rule and leaves below.

\subsubsection{Ascent and descent in a proof}\label{sec:ascent-descent}

The \defn{descent} of a symbol/subterm occurrence in a proof is the sequence of symbols/subterms traversed from it by exhaustively iterating tracking functions down the proof (until reaching the conclusion of the proof, or a cut formula); the \defn{ascent} is the converse, exhaustively applying (inverse) tracking functions upwards (until reaching an axiom or a logical rule introducting the symbol).
For example,
consider the proof below-left. Below-right we have shaded the ascent of the bound occurrence of $y$ in the conclusion, which is also the descent of the rightmost occurrence of the subterm $ha$ in the axiom.
\[
\newcommand\shadeproof[2]{%
\parrule{
\existsrule
  {\existsruleleft
    {\axiomruleright{\pp fha\,\com \, Pf#1}}
    {\ex x \pp x\,\com \, Pf#1}
  }
{\ex x \pp x\,\com\,\ex y P f #2}
}
{(\ex x \pp x)\parr(\ex y P f #2)}
}
\shadeproof{ha}{y}
\hspace{20ex}
\shadeproof{\fullshade{ha}}{\fullshade{y}}
\]

\subsection{Vacuous versus witnessed \texorpdfstring{$\exists$}{∃} rules}\label{sec:proof-witness-definition}

Let $\rho$ be an instance
\[\displayexistsrule\]
of an $\exists$ rule in a proof $\Pi$.
The rule is \defn{vacuous} if $x$ does not occur free in $A$; otherwise its \defn{witness} is $t$ (recoverable from $\Asub$ and $A$ because $A$ contains at least one occurrence of $x$), and we say $\rho$ is \defn{witnessed}.
For example, in the proof above-left, the witness of the first $\exists$ rule is $fha$ and the witness of the second
is $ha$.
If every $\exists$ rule in $\Pi$ introduces a distinct bound variable, then
we can unambiguously say that $x$ is vacuous or has witness $t$ (since
$x$ unambiguously determines the $\exists$ rule instance $\rho$).
For example, in the proof above-left (end of Section~\ref{sec:ascent-descent}), the witnesses of $x$ and $y$ are $fha$ and $ha$, respectively.

\section{Cut-free unification nets}\label{sec:cut-free-unets}

\subsection{Linkings}

A \defn{link} is a pair $\{l,\m2\ll\}$ of leaves whose predicate symbols are dual.
A \defn{linking} on a sequent $\Gamma$ is a set of disjoint links whose
union contains every leaf of $\Gamma$.
We draw a link $\{l,\m2\ll\}$ as an undirected edge between the
predicate symbols of $l$ and $\ll$.
For example, a linking on the sequent $\goodprenexinline y$ is shown below, with two links:
\[
\goodprenexlinking y
\]

\subsubsection{Translating a cut-free proof to a linking}\label{sec:translation}

Every cut-free proof $\Pi$ of a sequent $\Gamma$ translates to a linking $\translate\Pi$ on $\Gamma$ in the obvious way, by tracking dual pairs of predicate symbols from each axiom down the proof to form links on $\Gamma$. (\emph{Tracking} was defined in Section~\ref{sec:tracking}.)
For example, Figure~\ref{fig:prenex-translation} (left side) on page~\pageref{fig:prenex-translation} shows the translation of a cut-free proof to the linking displayed above.
A corresponding inductive definition of $\translate\Pi$, implementing the same tracking one rule at a time, is shown in {\figtranslation}Figure~\ref{fig:translation}, where
$\,\linkingon{\theta\!}{\!\Gamma}\,$ asserts that $\theta$ is a linking on $\Gamma$.

\subsection{Unifiable linkings and mgus}

Let $\lambda$ be a linking on $\Gamma$.
Without loss of generality, assume $\Gamma$ is clean
(renaming bound variables if necessary, \eg\ $\,\ex xPx\com\, Qx\,$ becomes
$\:\ex y Py\com\, Qx$\:).
A \defn{unifier} for $\lambda$ is an assignment of terms to
non-vacuous\footnote{Recall (Section~\ref{sec:vacuous-quantifier-formula-defn}) that a quantifier is vacuous if it binds no variable, \eg\ in $\all x\ex y\all z Pzc$ both $\all x$ and $\ex y$ are vacuous, but $\all z$ is not.}
existential variables which equalizes the term sequences at either
end of every link.
For example,
$\sigma=\ueg$ is a unifier for the linking in {\figlinkingunifier}Figure~\ref{fig:linking-unifier}
since upon substituting by $\sigma$ the first link has the three-term sequence $(gu,fx,a)$
at either end, and the second has the one-term sequence $(\hza)$.

The formal unification problem is as follows.  An axiom link between
$P(s_1,\ldots, s_n)$ and $\pp (t_1,\ldots, t_n)$ determines $n$ equations
$s_i \shorteq t_i$.  Taking the union across all links, we obtain a set
of equations $E$.
Solve $E$ for the existential variables (treating free and universal variables as constants).
For example, the link
\psscalebox{.8 .8}{$\;\Rnode P \pp(gu,fv,a)\;\;\Rnode p P(w,fx,a)\ncbar[angle=90,arm=2pt,nodesepA=2.5pt,nodesepB=1pt]{P}{p}\;$}
of the linking in Figure~\ref{fig:linking-unifier}
determines three equations $gu\shorteq w$, $fv\shorteq fx$ and $a\shorteq a$, and the link
\psscalebox{.8 .8}{$\;\Rnode Q Q(\hza)\;\;\Rnode q \qq(y) \ncbar[angle=90,arm=1.5pt,nodesepA=2pt,nodesepB=2.2pt]{Q}{q}\;$}
yields $\hza\shorteq y$, so
$E=\{gu\shorteq w,fv\shorteq fx,a\shorteq a,\hza\shorteq y\}$.
Solve $E$ for the existential variables $v$, $w$ and $y$
(treating the universal $u$ and $x$ as constants): $\ueg$.

A linking is \defn{unifiable} if it has a
unifier.
Unifiability can be determined in linear time \cite{MM76}.
The \defn{most general unifier} or \defn{mgu}
yields every other unifier by substitution. For example, the
mgu of
\newcommand{\mgulambda}[3]{\psset{nodesepA=#1pt,nodesepB=#2pt}
\:\ex x\Rnode x{\pp} x
\;\;
\ex y\m1\Rnode y P y
\ncbar[angle=90,arm=#3pt]{x}{y}\,}%
\psscalebox{.9 .85}{\raisebox{-.5pt}{$\mgulambda{1}{.5}{1}$}}
is $\sigma\!=\!\unifier{\gets x \alpha, \gets y \alpha}$ for $\alpha$ a free variable: every unifier is
$\sigma_t\!=\!\unifier{\gets x t, \gets y t}$ for some term $t$, and
$\sigma$ yields $\sigma_t$ by substituting $t$ for $\alpha$, \ie, $\sigma_t=\sigma\unifier{\gets \alpha t}$.
The mgu is defined up to free variable renaming \cite{LMM88}:
$\unifier{\gets x \beta, \gets y \beta}$ also represents the
mgu, for any other free variable $\beta$.

\subsection{Leaps and switchings}\label{sec:leaps}

Let $\lambda$ be a unifiable linking on a sequent $\Gamma$.
Without loss of generality, assume $\Gamma$ is clean.
A \defn{precedence} $\dep x y$ is an existential quantifier
$\ex x$ and a universal quantifier $\all y$ such that the mgu of $\lambda$
assigns to $x$ a term containing $y$.\footnote{Equivalently, \emph{every}
  unifier of $\lambda$ assigns to $x$ a term containing $y$.  The
  mgu-based definition is well-defined modulo renamining of free
  variables in the mgu, since free variables are distinct from
  bound variables, hence from universal variables.}
For example, the precedences of the linking in Figure~\ref{fig:linking-unifier} are $\dep v x$ and $\dep w u$.
The \defn{graph} $\glambda$ of $\lambda$ is the labelled directed forest
$\Gamma$ together with an undirected edge
between leaves $l$ and $l'$ for every link $\{l,\m2\ll\}$ in $\lambda$, and
a directed edge
from $\ex x$ to $\all y$, called a \defn{leap}, for every precedence $\dep x y$.
For example, the graph of the linking
\[
\rule{0ex}{4ex}\identityEg{2}{2}{5}
\]
whose unique unifier (hence mgu) is $\openU \gets x y\closeU$ is shown below.
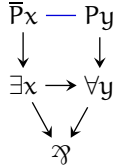
\newcommand{\core}{%
\psmatrix[colsep=1ex,rowsep=3ex]
\Rnode x {\dual P x} && \Rnode y {P y} \\
\Rnode e {\exists x} && \Rnode a {\forall y} \\
& \Rnode v \parr &
\ncline[arrows=-,linecolor=dblue] x y
\ncline x e
\endpsmatrix}%
\[
\rule{0ex}{11ex}
\dirtreestyle
\core
\ncline e v
\ncline y a
\ncline a v
\ncline e a
\]
A \defn{switching} of $\lambda$ is any
derivative of $\glambda$
obtained
by deleting all but one edge into each \,$\parr$\, and $\forall$
and undirecting
remaining edges.  For example, the four switchings of
the previous example are below.
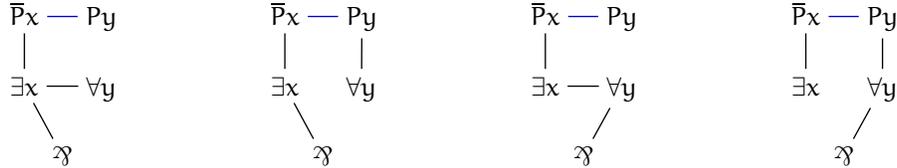
\begin{center}\begin{math}\label{four-switchings}
\renewcommand{\gap}{\hspace{12ex}}
\treestyle
\psset{arrows=-}
\core
\ncline e v
\ncline e a
\gap
\core
\ncline e v
\ncline y a
\gap
\core
\ncline a v
\ncline e a
\gap
\core
\ncline y a
\ncline a v
\end{math}\v2\end{center}
\subsection{Correctness criterion}\label{sec:correctness}
A linking is \defn{correct} if it is unifiable and all of its switchings are trees (acyclic
and connected).
For example, the linking above
is correct: all four of its switchings, depicted just above, are trees.
In Section~\ref{sec:cut-free-quadratic-correctness} we prove that correctness can be verified in quadratic time, despite the fact that constructing an explicit mgu, used to extract leaps, may take exponential time and space.

A \defn{cut-free unification net} (or \defn{cut-free unet} for short) on a sequent $\Gamma$ is a correct
linking on $\Gamma$.

\subsubsection{Correctness requires \texorpdfstring{$\tensor$-$\forall$}{⊗-∀} interaction}\label{sec:tensor-forall-interaction}

The following two linkings show that the interaction of tensor $\tensor$ and universal quantification $\forall$ is a necessary part of correctness, via leaps and switchings. Although the linkings differ only by exchanging $\tensor$ for $\parr$, the left is correct, while the right is not.
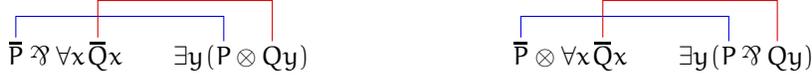
\[
\goodprenexlinking y
\hspace{16ex}
\badprenexlinking y
\]
The left linking $\theta$ was the subject of Figure~\ref{fig:prenex-translation}, and its sequent $\Gamma$ is an instance of prenex extrusion \mbox{$\goodprenexvdashinline x$}, which is provable in \fomll\ ($x$ not free in $A$), with the right bound variable $x$ renamed to $y$ to avoid ambiguity.
The right linking $\theta'$ differs from $\theta$ in that its sequent $\Gamma'$ has $\tensor$ and $\parr$ interchanged, and $\Gamma'$ is an unprovable instance of prenex extrusion $\badprenexvdashinline x$.
We shall verify that $\theta$ is correct, while $\theta'$ is not.

Both linkings have the same unifier, $\unifier{\gets y x}$.
Their respective graphs are:
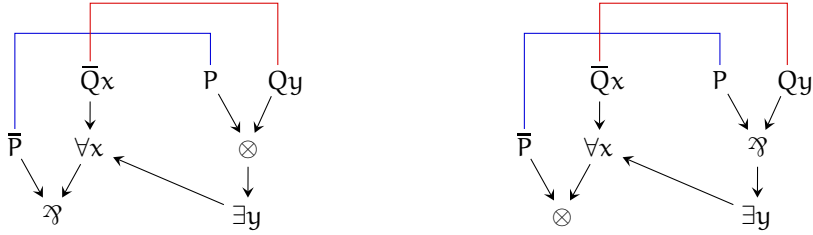
\[
\updownpadding{19} 3
\goodprenexleapgraph y
\hspace{40ex}
\badprenexleapgraph y
\]
Here is a switching of each:
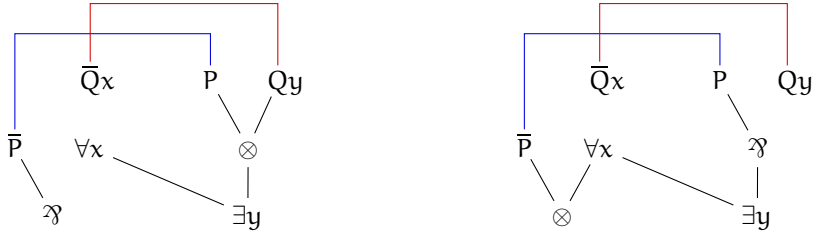
\[
\updownpadding{19} 3
\goodprenexswitching y
\hspace{40ex}
\badprenexswitching y
\]
All fours switchings of $\theta$ are trees, including the one above-left. However, the switching of $\theta'$ above-right is not a tree, so $\theta'$ fails to be a unification net.

This example shows that one cannot hope for a factorized correctness criterion which treats the propositional and first-order parts independently, for example, verifying separately that the underlying propositional MLL linking is correct (true for both linkings above), and that quantifier precedence $\depsym$ (together with the subformula relation on quantifiers) is acyclic (also true for both linkings above).

\subsection{Correctness is at worst quadratic time}\label{sec:cut-free-quadratic-correctness}

Unifiability can be verified in linear time \cite{MM76}. However, a standard mgu
of the form $$\unifier{\gets {x_1}{t_1},\ldots,\gets{x_n}{t_n}}$$
may take exponential time and space to construct, and be exponential in size.\footnote{\expunifierfootnote}
Therefore, since such an mgu was used to construct the leaps in the correctness criterion, via
its precedences, correctness is naively exponential time and space.
The following theorem shows that we can check correctness in quadratic time by extracting all mgu precedences without actually having to build the mgu explicitly.
\begin{theorem}[Cut-free quadratic-time correctness]\mbox{}\label{thm:quadratic}\\
The correctness of a cut-free unification net can be verified in quadratic time.
\end{theorem}
\begin{proof}
  Using the main linear-time unification algorithm of \cite{MM76} we
  construct a sequence of substitutions
  $\unifier{\gets{x_1}{t_1}}$,\ldots,$\unifier{\gets{x_n}{t_n}}$ with $x_i$ not in $t_j$ for $i>j$ whose sequential composition
  is the mgu. In other words, writing $s\sigma_k$ for
$s\sig 1\sig 2\cdots\sig k$
for any term $s$, the mgu is
\[\unifier{\:\gets{x_1}{t_1},\; \gets{x_2}{t_2\sigma_1},\;\ldots,\;\gets{x_n}{t_n\sigma_{n-1}}\:} \]
We now extract all precedences from the mgu using transitive
closure, without the mgu itself. Let $\{y_{i1},\ldots,y_{im_i}\}$ be the set
of universal variables occuring in $t_i$, and define $t'_i$ as the term
$f_iy_{i1}\ldots y_{im_i}$ for a fresh $m_i$-ary function symbol
$f_i$. By construction, the sequential composition of
substitutions $\unifier{\gets{x_1}{t'_1}}$,\ldots,$\unifier{\gets{x_n}{t'_n}}$ has the same precedences
as the mgu, but can be constructed in quadratic time (since each
universal variable appears at most once in $t'_i$).
Thus we can construct the graph in quadratic time.

The graph determines a contractibility graph
  \cite{Dan90} with $\parr\mm2$s and $\forall$s as switched nodes, and
  leaves, $\tensor\mm2$s and $\exists$s as unswitched nodes, checkable
  in linear time \cite{Gue99}.
Hence the overall complexity of correctness is at worst quadratic in the size of the unification net.
\end{proof}
Later we extend this result to unification nets with cuts, in Theorem~\ref{thm:cut-quadratic}, page~\pageref{thm:cut-quadratic}.

\subsection{The translation of a cut-free proof is a cut-free unification net}\label{sec:translation-is-unet}

Recall the translation of a cut-free \fomll proof $\Pi$ to a linking $\translate\Pi$ defined in Section~\ref{sec:translation}.
\begin{theorem}\label{thm:translation}
  The translation $\translate\Pi$ of a cut-free \fomll proof\/ $\mm2\Pi$ is a unification net.
\end{theorem}
\begin{proof}
By structural induction on the proof, with respect to Figure~\ref{fig:translation}. We assume (without loss of generality) that every sequent is clean (bound variables distinct from one another and from free variables).

The base case of an axiom rule is trivial.

$\parr$ case. The mgu for $\tightlinkingon\theta{\Gammacom A\com B}$ is also an mgu for $\tightlinkingon\theta{\Gammacom\, A\parr B}$ since all quantifers and leaves stay the same. Any non-tree switching of $\tightlinkingon\theta{\Gammacom A\parr B}$ will induce a non-tree switching of $\tightlinkingon\theta{\Gammacom\, A\com B}$, since the only addition to the switching graph is an outermost $\parr$ vertex.

$\tensor$ case. The mgu for $\theta\,\cup\,\phi$ is the union of the mgus for $\theta$ and $\phi$, since their bound variables are independent: each comes from either the left hypothesis of the $\tensor$ rule or the right, but not both. Every switching of $\theta\,\cup\,\phi$ is the disjoint union of a switching of $\theta$ and a switching of $\phi$ joined at the new outermost $\tensor$ vertex: due to bound variable independence, there can be no leap between the two.

$\forall$ case. Any mgu for $\tightlinkingon\theta{\Gammacom A}$ is also an mgu of $\tightlinkingon\theta{\Gammacom\all x A}$\,, since $x$ has merely transitioned from free to bound.  Every switching of $\tightlinkingon\theta{\Gammacom\all x A}$ is a switching of $\tightlinkingon\theta{\Gammacom A}$ plus a leap into the new $\forall x$ vertex, which cannot break the property of being a tree since it has no outgoing edge.

$\exists$ case.
Suppose $\sigma$ unifies $\tightlinkingon\theta{\Gammacom\Asub}$.
Assume $x$ occurs free in $A$, otherwise the result is immediate.
Then $\sigma\,\cup\,\unifier{\gets x t}$ unifies $\tightlinkingon\theta{\Gammacom\ex x A}$, since any occurrence of $x$ in a leaf of $\ex x A$ is replaced with an occurrence of $t$ in the same leaf of $\Asub$. Suppose $\tightlinkingon\theta{\Gammacom\ex x A}$ has a switching cycle. It must involve a leap from the new $\ex x$ vertex to some $\all y$, otherwise $\tightlinkingon\theta{\Gammacom\Asub}$ immediately has a switching cycle. Thus $\dep x y$ in $\tightlinkingon\theta{\Gammacom\ex x A}$, so the mgu $\sigma$ assigned to $x$ a term $t$ containing $y$, and $x$ occurs in $A$ (otherwise the mgu would assign a fresh free variable to $x$ instead). Thus $t$ occurs in $\Asub$, hence $y$ occurs free in $\Asub$, contradicting the fact that (without loss of generality) all free and bound variables are distinct.
\end{proof}

\subsubsection{Exponential compression of some proofs}\label{sec:exp-compression}\label{sec:exponential-compression}

On certain cut-free proofs, the translation to a cut-free unification net provides an exponential compression.
Consider the progression of formulas $A_i$ beginning
\[
\newcounter{n}
\newcommand\assocLeft[1]{%
\setcounter{n}{0}\whiledo{\not{\value{n}>#1}}{(\stepcounter{n}}
\blowupArgsOne
\setcounter{n}{0}\whiledo{\not{\value{n}>#1}}{\idot x_{\then})\stepcounter{n}}}
\newcommand\assocRight[1]{%
\setcounter{n}{#1+1}\whiledo{\value{n} > 0}{\addtocounter{n}{-1}(x_{\then}\idot}
\blowupArgsOne
\setcounter{n}{0}\whiledo{\not{\value{n}>#1}}{)\stepcounter{n}}}
\newcommand\existsvec[1]{%
\setcounter{n}{#1+1}\whiledo{\value{n}>0}{\addtocounter{n}{-1}\ex {x_{\then}}}}
\newcommand\expBody[1]{%
\bigleft
\pp
\assocLeft{#1}
\parr
P
\assocRight{#1}
\bigright
}
\newcommand\expRow[1]{
A_{#1} &=& \existsvec{#1} & \expBody{#1}
}
\begin{array}{lcr@{}l}
\expRow 0 \\
\expRow 1 \\
\expRow 2 \\
\expRow 3
\end{array}
\]
While the size of $A_i$ grows linearly in $i$, the unique cut-free proof $\Pi_i$ of $A_i$ grows exponentially in $i$, since its axiom rule contains the predicate $\alpha_i$ with $2^i$ occurrences of the constant $c$, and its dual $\dual\alpha_i$:
\[
\begin{array}{l@{\;\;\;=\;\;\;}l}
\alpha_0 & P\expTermZero  \\
\alpha_1 & P\expTermOne   \\
\alpha_2 & P\expTermTwo   \\
\alpha_3 & P\expTermThree
\end{array}
\]
The translation $\translate{\Pi_i}$ of $\Pi_i$ is
\[ \expsmalleg \]
which is the unique cut-free unification net on $A_i$. This grows linearly with $i$, hence $\translate{\Pi_i}$ is exponentially smaller than $\Pi_i$.

The unique cut-free Girard net $G_i$ on $A_i$ also grows exponentially in $i$, since its axiom link is between $\alpha_i$ and $\dual\alpha_i$.
In Section~\ref{sec:complexity} we discuss the complexity issues of Girard nets in depth, and explain how unification nets resolve them.

\section{Cut-free surjectivity theorem}\label{sec:cut-free-surjectivity}

In this section we show that every cut-free unification net derives from a cut-free proof.
In standard proof net theory, a surjectivity theorem of the following form would typically be called a \emph{sequentialization} theorem. However, as remarked in the Introduction (Section~\ref{sec:beyond-sequentialization}), and emphasized in the commuting diagram in Figure~\ref{fig:beyond-sequentialization}, in the context of unification nets the inverse of the surjection expresses both sequentialization (choice of rule orderings) and explicit witness assignment (choice of witnesses). Thus we simply label the theorem as \emph{surjectivity}.
\begin{theorem}[Cut-free surjectivity]\mbox{}\label{thm:cut-free-surj}\label{thm:surjectivity}\label{thm:cut-free-surjectivity}\\
The translation from cut-free proofs to cut-free unification nets is surjective.
\end{theorem}
We prove this theorem via an MLL encoding of a unification net, called the \defn{frame}, defined in Section~\ref{sec:frame}, via which we can appeal to the standard MLL splitting tensor theorem \cite{DR89}. The proof of Theorem~\ref{thm:surjectivity} is Section~\ref{sec:surjectivity-proof}.

\subsection{The MLL frame of a unification net}\label{sec:frame}

Let $\theta$ be a unification net on an \fomll sequent $\Gamma$.
Define the \defn{frame}
of $\theta$ by exhaustively applying the following subformula rewrites, in order, to obtain a linking $\mllencode\theta$ on an MLL sequent $\mllencode\Gamma$:
\begin{itemize}
\myitem 1 \emph{Encode every precedence $\dep x y$ as a new link.} Iterate through the precedences $\dep x y$ one by one. For each such precedence $\dep x y$, with corresponding subformulas $\ex x A$ and $\all y B$, add a link as follows. Let $Q$ be a fresh predicate symbol (distinct for each precedence). Replace $\ex x A$ by $Q\tensor \ex x A$ and $\all y B$ by $\qq\parr\all y B$, and add a link between $Q$ and $\qq$.
\myitem 2 \emph{Delete quantifiers.} After step 1, replace every subformula of the form $\all y A$ or $\ex x\mm2 A$ by $A$.
(We no longer need their leaps, because we encoded leaps as links in step 1.)
\myitem 3 \emph{Delete terms.} After step 2, replace every predicate $Pt_1\ldots t_n$ by a nullary predicate symbol $P$.
\end{itemize}
For example, the frame of the unification net $\theta$
\[ \identityEgForMultiplicativeEncoding 3 3 6\rule{0ex}{3ex} \]
(already considered in Section~\ref{sec:leaps}) is the following MLL linking $\mllencode\theta$:
\[ \identityMultiplicativeEncoding\rule{0ex}{4ex} \]
Note that this is a correct MLL proof net.
We generalize this in the following proposition.
\begin{lemma}\label{lem:frame-is-mll-net}
Let $\theta$ be a unification net on $\Gamma$. The frame $\mllencode\theta$ on $\mllencode\Gamma$ is an MLL proof net.
\end{lemma}
\begin{proof}
  Each step (1)--(3) in the frame construction preserves the property that every switching is a tree. Steps (1) and (2) together replace every leap with a link, and since the new $\tensor$ represents the outgoing end of the leap and the new $\parr$ represents the incoming end, switchings correspond before and after.
Step (3) has no effect on switchings (since it just re-labels leaves).
\end{proof}
For example, here is the graph of $\identityEgForMultiplicativeEncoding 2 2 4$, followed by its four switchings:
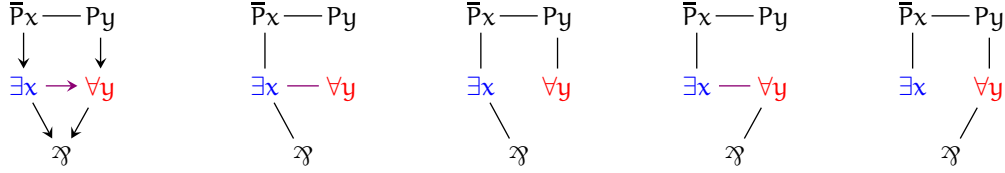
\[
\dirtreestyle
\renewcommand{\core}{%
\psmatrix[colsep=1ex,rowsep=3ex]
\rule{0ex}{2ex}\Rnode x {\pp x}
&& \Rnode y {P y}
\\
{\color{blue}\Rnode e {\exists x}} && {\color{red}\Rnode a {\forall y}} \\
& \Rnode v \parr &
\ncline[arrows=-,nodesep=1.5pt] x y
\ncline x e
\endpsmatrix}
\renewcommand{\gap}{\hspace{8.5ex}}
\core
\ncline e v
\ncline y a
\ncline a v
\ncline[linecolor=dpurple]  e a
\hspace{2ex}
\psset{arrows=-}
\gap
\core
\ncline e v
\ncline[linecolor=dpurple]  e a
\gap
\core
\ncline e v
\ncline y a
\gap
\core
\ncline a v
\ncline[linecolor=dpurple]  e a
\gap
\core
\ncline y a
\ncline a v
\]
Correspondingly, its frame
\( \identityMultiplicativeEncoding\rule{0ex}{6ex} \)
has the following graph and four switchings:
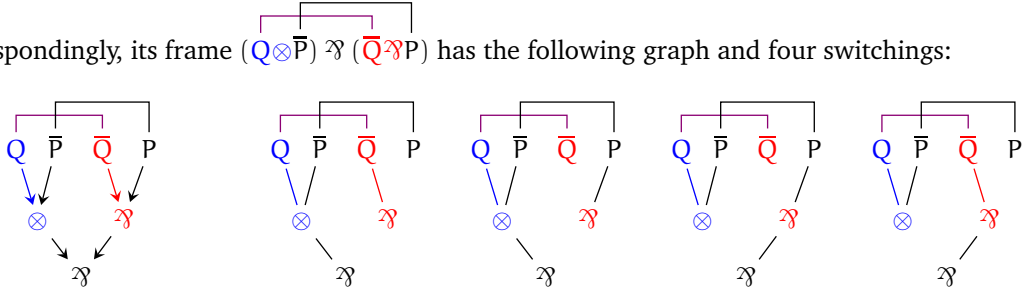
\[\label{four-mll-switchings}
\dirtreestyle
\renewcommand{\core}{%
\psmatrix[colsep=0ex,rowsep=3ex]
\rule{0ex}{5ex}
{\color{blue}\Rnode 1 Q}
&&
\Rnode x \pp\mm5
&&
{\color{red}\Rnode 2 \qq}
&&
\mm5\Rnode y P
\\
&
{\color{blue}\Rnode e {\tensor}}
&&&&
{\color{red}\Rnode a {\parr}}
\\[-1ex]
&&&
\Rnode v \parr
\\[-4ex]
\ncline[nodesepA=2pt,linecolor=blue] 1 e
\ncline x e
\psset{arrows=-}
\ncbar[angle=90,arm=5pt,linecolor=dpurple]{1}{2}
\ncbar[angle=90,arm=10pt]{x}{y}
\endpsmatrix}
\renewcommand{\gap}{\hspace{3.8ex}}
\core
\ncline e v
\ncline a v
\ncline y a
\ncline[nodesepA=2pt,linecolor=red] 2 a
\h5
\psset{arrows=-}
\gap
\core
\ncline e v
\ncline[nodesepA=2pt,linecolor=red] 2 a
\gap
\core
\ncline e v
\ncline y a
\gap
\core
\ncline a v
\ncline y a
\gap
\core
\ncline a v
\ncline[nodesepA=2pt,linecolor=red] 2 a
\]
Observe the direct correspondence, switching for switching.

We shall require the following frame-related lemma in the proof of Theorem~\ref{thm:surjectivity} (Cut-free surjectivity).
Let $\theta$ be a unification net on $\Gamma$.
A $\tensor$ root vertex $v$ \defn{splits} if deleting $v$ (and its two incoming edges) from the graph $\graphof\theta$ disconnects the it into two connected components.
\begin{lemma}\label{lemma:frame-split}
  No tensor added during the frame construction splits.
\end{lemma}
\begin{proof}
Let the MLL proof net $\frameof\theta$ on $\frameof\Gamma$ be the result of applying the frame construction to the unification net $\theta$ on $\Gamma$.
Every tensor added during the construction has the form $Q\tensor C$ for a fresh predicate symbol $Q$. Let $\qq\parr D$ be the subformula of the dual predicate symbol $\qq$, also added during the construction.
If the $\tensor$ splits then the unique path from the $\tensor$ to the $\parr$ in the graph $\graphof{\frameof\theta}$ traverses the link from $Q$ to $\qq$. Thus every switching which deletes the edge from $\qq$ into the $\parr$ is disconnected, contradicting the fact that (by Lemma~\ref{lem:frame-is-mll-net}) $\frameof\theta$ is an MLL proof net, every one of whose switchings is a tree.
\end{proof}

\subsection{Proof of cut-free surjectivity theorem}\label{sec:surjectivity-proof}

\begin{proofof}{Theorem~\ref{thm:surjectivity} (Cut-free surjectivity)}
  Let $\theta$ be a cut-free unification net on $\Gamma$. We proceed by induction on the number of connectives in $\Gamma$. In the base case $\Gamma$ is $Pt_1\ldots t_n\com\,\pp t_1\ldots t_n$ for some $n$-ary predicate symbol $P$ and terms $t_i$, hence the corresponding axiom translates to $\theta$, a single link.

For the induction step, let $\graph$ be the graph of $\theta$.
\begin{itemize}
\myitem{$\parr$}
Suppose $\Gamma$ is $\Delta\com A\m2\parr B$. Let $\Gamma'$ be $\Delta\com A\com B$ and define $\theta'$ on $\Gamma'$ by the same links as $\theta$ (identifying the leaves of $\Gamma'$ with those of $\Gamma$). The linking $\theta'$ is a unification net because (a) the mgu of $\theta$ is also the mgu of $\theta'$ (since all quantifiers and terms remain untouched, so the unification problem is identical) and (b) every switching of $\theta'$ is a tree, since were some switching of
$\theta'$ not a tree, it would induce a non-tree switching of
$\theta$ by adding an edge to the deleted $\parr$ down from the root of $A$ (or of $B$).
Appealing to induction with $\theta'$ yields a cut-free proof $\Pi'$ whose translation is $\theta'$. Appending the par rule $\frac{\Delta\com A\com B}{\Delta\com A\m2\parr B}$ yields a cut-free proof $\Pi$, whose translation is $\theta$ because all links pass through the $\parr$ rule.
\myitem{$\forall$}
Suppose $\Gamma$ is $\Delta\com\all x A$. Let $\Gamma'$ be $\Delta\com A$ and define $\theta'$ on $\Gamma'$ by the same links as $\theta$ (identifying the leaves of $\Gamma'$ with those of $\Gamma$). The mgu of $\theta$ is also the mgu of $\theta'$ since $x$ has only transitioned from a universal variable to a free variable (hence the unification problem is identical).
Every switching of $\theta'$ is a tree, since were some switching of
$\theta'$ not a tree, it would induce a non-tree switching of
$\theta$ by adding an edge down from the root of $A$ to the deleted $\all x$.
Appealing to induction with $\theta'$ yields a cut-free proof $\Pi'$ whose translation is $\theta'$. Appending the $\forall$ rule $\frac{\Delta\com A}{\Delta\com \all x A}$ yields a cut-free proof $\Pi$, whose translation is $\theta$ because all links pass through the $\forall$ rule.

\myitem{$\exists$}
If $\graph$ has a root $\exists$ with no outgoing leap,
say $\ex x$, we write down a final $\exists$ rule as follows.
Let $\sigma$ be the mgu of $\theta$, assigning the term $t$ to $x$.
Delete $\ex x$ by replacing the corresponding formula $\ex x A$ in $\Gamma$ by $A\substitute{\gets x t}$ (substituting $t$ for $x$ throughout $A$) to form $\Gamma'$, write down a final $\exists$ rule inferring $\Gamma$ from $\Gamma'$, and appeal to induction with $\theta'$ on $\Gamma'$.
We obtain the mgu of $\theta'$ on $\Gamma'$ by deleting the assignment $\gets x t$ from $\sigma$ and replacing every other assignment $\gets y u$ with $\gets y u'$ where $u'=u\substitute{\gets x t}$ (substituting $t$ for $x$ throughout $u$).
Every switching of the graph $\graph'$ of $\theta'$ on $\Gamma'$ is a tree because each switching induces one in $\graph$ (since the deleted $\exists x$ was a root of $\Gamma$ and every leap of $\graph'$ is also a leap in $\graph$).\footnote{Intuitively, the fact that $\ex x$ has no leap means that no earlier $\forall$ rule requires this $\exists$ rule to proceed it, in order to hide free variables occurrences in the $\forall$ rule context. Thus it is safe to write down this $\exists$ rule as the final rule of a proof (since no $\forall$ rule is forced to be below the $\exists$ rule).}

\myitem{$\exists\tensor$}
Otherwise every root of $\graph$ is either an $\exists$ with an outward leap or a $\tensor$.
Let $\frameof\theta$ on $\frameof\Gamma$ be the frame of $\theta$ on $\Gamma$ (defined in Section~\ref{sec:frame}).
By the MLL splitting tensor theorem \cite{DR89} some $\tensor$ root vertex $v$ of $\frameof\theta$ on $\frameof\Gamma$ splits.
By Lemma~\ref{lemma:frame-split} $v$ is a $\tensor$ vertex in $\Gamma$, and since every root $\exists$ has an outward leap, $v$ is a root (since no root $\tensor$ of $\frameof\Gamma$ can result from step 2 in the frame construction deleting an $\exists$ vertex below it).
Thus $v$ splits in $\graph$: deleting $v$ (and its two incoming edges) disconnects $\graph$ into $\graph_1$ and $\graph_2$.
Let
$\Gamma_1$ and $\Gamma_2$ be the underlying sequents of $\graph_1$ and $\graph_2$,
and let $\theta_1$ and $\theta_2$ be the respective restrictions of $\theta$.
Since $v$ splits, each $\theta_i$ is a unification net: its mgu is by restriction from $\theta$, and any non-tree switching of $\theta_i$ would induce a non-tree switching of $\theta$.
Write down a $\tensor$ rule inferring $\Gamma$ from $\Gamma_1$ and $\Gamma_2$, and appeal to induction with $\theta_1$ on $\Gamma_1$ and $\theta_2$ on $\Gamma_2$.
\end{itemize}
\vspace{-4.2ex}\end{proofof}

\section{Canonicity Theorem}\label{sec:canonicity}

The cut-free \fomll proofs
\[
\redundancyproof\firstwitness{x}
\h{20}
\otherredundancyproof\secondwitness{x}
\]
are \emph{equivalent} in the sense that the left yields the right by commuting the order of the $\exists$ rules and replacing one arbitrary choice of existential witness, $\firstwitness$, by another, $\secondwitness$.
While they have distinct Girard nets (because Girard nets inherit redundant explicit exponential witnesses),
they have the same unification net (the one in Figure~\ref{fig:canonicity} of the Introduction):
\[
\nonredundantunificationnet x\rule{0ex}{4ex}
\]
In Section~\ref{sec:equivalence} we formalize two proofs as equivalent if one can be obtained from the other by rule commutations
and witness replacement,
and in Section~\ref{sec:canonicity-proof} prove:
\begin{theorem}[Canonicity]\label{thm:canonicity}
Two cut-free \fomll proofs are equivalent (modulo rule commutations and witness replacement) if and only if they have the same unification net.
\end{theorem}

\subsection{Proof equivalence via commutations and witness replacement}\label{sec:equivalence}

\subsubsection{Witness replacement}\label{sec:witness-replacement}

Let $\Pi$ be a proof of $\Gamma$. Without loss of generality, assume $\Gamma$ is clean. Thus every $\exists$ rule introduces a distinct existential variable.
Let $x$ be an existential variable in $\Pi$ and let $\rho$ be the $\exists$ rule
\[ \displayexistsrule \]
introducing $x$.
The \defn{scope} of $x$ in $\Pi$ is every occurrence of $t$ above $\rho$ which descends to an occurrence of $x$ in $\ex x A$ in the conclusion of $\rho$. (See Section~\ref{sec:ascent-descent} for the definition of descent, via the symbol tracking functions through rules.)
Given a term $u$, define the \defn{witness replacement} $\Pi\substituteto x u$
by replacing every occurrence of $t$ in the scope of $x$ by $u$.
For example, if $\Pi$ is below-left then $\Pi\substituteto x {hzb}$ is below-centre and $\Pi\substituteto x {hzb}\substituteto y {hzb}$ is below-right:
\[
\newcommand\showsub[2]{\hspace{5ex}\raisebox{4ex}{\begin{math}\overset{#1\shortmapsto #2}{\longrightarrow}\end{math}}\hspace{5ex}}
\otherredundancyproof{fc}y
\showsub x {hzb}
\othersplitredundancyproof{hzb}{fc}y
\showsub y {hzb}
\otherredundancyproof{hzb}y
\]
In general, a witness replacement may not be a well-formed proof: in the center example $\Pi\substituteto x {hzb}$ the axiom rule is ill-formed, since it is between non-dual predicates $\pp hzb$ and $Pfc$.
The right example $\Pi\substituteto x {hzb}\substituteto y {hzb}$ is, however, a well-defined cut-free proof, since the axiom is between dual predicates $\pp hzb$ and $P hzb$.

A \defn{re-witnessing}\label{def:re-witnessing} of a cut-free proof $\Pi$ is any proof obtained from $\Pi$ by a sequence of witness replacements. For example, $\Pi\substituteto x {hzb}\substituteto y {hzb}$ above-right is a re-witnessing of $\Pi$ above-left.
If $\sigma=\assignment{\gets{x_1}{t_1},\ldots,\gets{x_n}{t_n}}$ is an assigment of terms to existential variables in $\Pi$, write $\Pi\sigma$ for
re-witnessing by $\sigma$, \ie, $\Pi\sigma$ = $\Pi\substituteto{x_1}{t_1}\ldots\substituteto{x_n}{t_n}$. This is well-defined with respect to the choice of ordering of the $x_i$ because scopes of distinct existential variables do not overlap.

\subsubsection{Proof equivalence definition}

A \defn{rule commutation} is any of the subproof rewrites in {\commfig}Figure~\ref{fig:rule-commutations}, where, for clarity and brevity, passive side formulas are omitted. For example, the $\exists/\tensor$ commutation at the bottom-left of Figure~\ref{fig:rule-commutations} abbreviates
\[ \tecommcontexts \]
where the omitted contexts $\Gamma$ and $\Delta$ flow passively through the rules.

Two cut-free \fomll proofs
are \defn{commutation-equivalent} if one yields the other by a sequence of (zero or more) rule commutations,
and \defn{equivalent} if one yields the other by a sequence of rule commutations and/or re-witnessings.
For example, the two proofs shown at the beginning of Section~\ref{sec:canonicity} are equivalent, but not commutation-equivalent (since re-witnessing is required).

\subsection{Proof of the Canonicity Theorem}\label{sec:canonicity-proof}

We prove Theorem~\ref{thm:canonicity} (page~\pageref{thm:canonicity}), the Canonicity Theorem. The proof follows from a number of auxiliary results below.

Let $\theta$ be a cut-free unification net on $\Gamma$, with graph $\graphof\theta$. A root $v$ of $\Gamma$ is \defn{ready} if any of the following cases hold, which correspond to our ability to write down a final rule introducing $v$ in the proof of Theorem~\ref{thm:surjectivity} (Cut-free surjectivity):
\begin{itemize}
\item $v$ is a $\parr$ or $\forall$;
\item $v$ is an $\exists$ with no outgoing leap in $\graphof\theta$;
\item $v$ is a $\tensor$ which splits $\graphof\theta$.
\end{itemize}
A rule $\rho$ \defn{commutes downwards} if a commutation rewrite (Figure~\ref{fig:commutations}) applies with $\rho$ as the upper rule.
\begin{lemma}\label{lem:ready-commute-once}
  Let $\rho$ be a penultimate logical rule in a cut-free proof\/\, $\Pi$ introducing a vertex $v$. If\/ $\mm2v$ is ready in the unification net of\/\, $\Pi$, then $\rho$ commutes downwards.
\end{lemma}
\begin{proof}
  Let $w$ be the root introduced by the final rule. Since $v$ is ready, it is also root. Thus $\rho$ will only fail to commute downwards if one of the side conditions of a commutation is not satisfied, of which there are two cases (see Figure~\ref{fig:commutations}): (a) $\forall$ commuting down through $\tensor$, with the $x\not\in C$ side condition, and (b) $\exists$ commuting down through $\forall$, with the $y\not\in t$ side condition. The former case is ruled out by assuming (without loss of generality) that $\Gamma$ is clean, and the latter case is ruled out by observing that if $y\in t$ then $v=\ex x$ would have
a leap $\ex x\toedge \all y$, contradicting the readiness of $v$.
\end{proof}
Let $\Pi$ be a cut-free proof of $\Gamma$ and $v$ a vertex of $\Gamma$. Since \fomll has no contraction or weakening, a unique rule $\rho(v)$ in $\Pi$ introduces $v$.
\begin{lemma}\label{lem:ready-final}
Let $\theta$ be the unification net of a cut-free proof\/ $\Pi$.
If\/ $v$ is a ready vertex in $\theta$, then $\Pi$ is commutation-equivalent to a cut-free proof\/\, $\Pi'$ whose final rule introduces $v$.
\end{lemma}
\begin{proof}
  Let $\rho$ be the rule in $\Pi$ which introduces $v$.
  Proceed by induction on the number of rules between $\rho$ and the final rule of $\Pi$,
iterating Lemma~\ref{lem:ready-commute-once}.
\end{proof}
\begin{lemma}\label{lem:mgu-rewitness-well-defined}
Let $\sigma$ be the mgu of the unification net of a cut-free proof\/ $\Pi$\,. The re-witnessing $\Pi\sigma$ is a well-defined cut-free proof.
\end{lemma}
\begin{proof}
The unification net correctness criterion ensures that the mgu $\sigma$ equalizes the term sequences in every link. Thus every axiom of $\Pi\sigma$ is well-formed, so $\Pi$ can only fail to be a well-formed proof if one of its $\forall$ rules introducing $\all y$ fails the side condition precluding free occurrences of $y$ in the context.
Since $\sigma$ is an mgu, for every existential variable $x$ the witness of $x$ in $\Pi$ contains more universal variables than the witness assigned by $\sigma$. Thus a $\forall$ rule side condition fails in $\Pi\sigma$ only if it also fails in $\Pi$.
\end{proof}
Let $\Pi$ be cut-free proof. Without loss of generality, its conclusion $\Gamma$ is clean, hence every quantifer rule introduces a distinct bound variable.
Define the \defn{witness assignment}\label{witness-assignment} $\sigma_\Pi$ of $\Pi$ by
setting $\sigma_\Pi(x)$ to be the witness of $x$, for every non-vacuous existential variable of $\Pi$.
\begin{lemma}\label{lem:same-wit-comm-equiv}
  Suppose $\Pi$ and $\Pi'$ are cut-free proofs with the same witness assignment and the same unification net.
  Then $\Pi$ and $\Pi'$ are commutation-equivalent.
\end{lemma}
\begin{proof}
  Let $\theta$ be the unification net of $\Pi$ and $\Pi'$.
  Let $\rho$ be the last rule of $\Pi$, introducing the vertex $v$. Since $\rho$ is the last rule of $\Pi$, $v$ is ready in $\theta$.
By Lemma~\ref{lem:ready-final}, $\Pi'$ is commutation-equivalent to $\Pi''$ whose final rule $\rho''$ introduces $v$. Since commutations do not change witnesses, $\Pi$ and $\Pi''$ have the same witness assignment. Thus $\rho$ and $\rho''$ are the same rule instance: if they are not $\exists$ rules, this is immediate; otherwise the equality of witness assignment ensures that as $\exists$ rules $\rho$ and $\rho''$ introduce the same witness. Appeal to induction with the subproofs above $\rho$ and $\rho''$.
\end{proof}
\begin{proofof}{Theorem~\ref{thm:canonicity} (Canonicity)}
  Let $\Pi$ and $\Pi'$ be cut-free proofs with the same unification net, whose mgu is $\sigma$.
  By Lemma~\ref{lem:mgu-rewitness-well-defined} the re-witnessings $\Pi\sigma$ and $\Pi'\sigma$ are well-defined cut-free proofs, which are commutation-equivalent by Lemma~\ref{lem:same-wit-comm-equiv} because they have same witness assigment, $\sigma$. Thus $\Pi$ and $\Pi'$ are equivalent modulo rule commutations and re-witnessings.
\end{proofof}

\section{Unification nets with cuts}\label{sec:cut}\label{sec:unets-with-cut}

Extending unification nets with cuts comes essentially
for free, as in the propositional case \cite{Gir87} where one
treats a cut as a tensor (see \eg\ \cite{HG03}):\footnote{While the
  definition comes for free, proving that cut elimination is
  well-defined requires work, as in the propositional case.}
\begin{equation*}
  \cut{\,A\,}{\:\AA\:} \h4 \approx \h4 A\tensor \AA
\end{equation*}
For quantifiers one must generalize slightly, to an existentially closed tensor:
\begin{equation}\label{cutex}
  \h4\cut{\,A\,}{\:\AA\:} \h4 \approx \h4 \ex {\vec x} (A\tensor \AA)
\end{equation}
where $\ex{\vec x}=\ex{x_1}\ldots\ex{x_n}$ for $x_1\ldots x_n$
the free variables in $A$.
Appendix~\ref{sec:cutex-intuition} provides motivation and
intuition for (\ref{cutex}) from a proof-theoretic perspective.
The following definitions derive automatically from the cut-free
definitions (Section~\ref{sec:cut-free-unets}) by thinking of a cut as
an existentially closed tensor.

A \defn{cut} $\cut{A}{\AA}$ is a disjoint union of dual formulas $A$
and $\AA$, the \defn{cut formulas},
with an undirected edge between their roots, a \defn{cut edge}.
A \defn{cut sequent} is a disjoint union of a sequent and zero or more cuts.
Let $\Delta$ be a cut sequent.
A \defn{link} on $\Delta$ is a pair $\{l,\m2\ll\}$ of dual leaves in
$\Delta$.
A \defn{linking} on $\Delta$ is a set of disjoint links whose union
contains every leaf of $\Delta$.
The lower half of Figure~\ref{fig:cut-elim-comparison} (page~\pageref{fig:cut-elim-comparison}) shows two linkings on cut
sequents: the first linking on
\[\all x\pp fx\,\com\; \cut{\ex y Py\,}{\,\all y\pp y}\,\com\; \ex z(Pz\tensor\m1(\qq z\m1\parr Qz))\]
and the second on
\[\all x\pp fx\,\com\; \cut{Py\,}{\,\pp y}\,\com\; \ex z(Pz\tensor\m1(\qq z\m1\parr Qz))\]
(The sequent of the third linking is trivially a cut sequent, with zero cuts.)

\vspace{2pt}%
We consider every free variable of $A$ (hence also $\AA$) to be bound
in the cut $\cut{A}{\AA}$.
Such bound variables are the \defn{cut variables} of $\cut{A}{\AA}$.
Their renaming is analogous to renaming of existential or universal
variables.
For example, the following two cut sequents are equivalent up to
renaming of bound variables:
\begin{center}\begin{math}
\all x Px,\, \;\cut{Qx}{\qq x},\, \;\cut{Rx}{\rr x}
\hspace{26ex}
\all x Px,\, \;\cut{Qy}{\qq y},\, \;\cut{Rz}{\rr z}
\end{math}\end{center}
This is akin to the renaming of bound variables in the cut-free
sequent below-left to yield the cut-free sequent below-right:
\begin{center}\begin{math}
\all x Px,\, \;\ex x (Qx\tensor \qq x),\, \;\ex x(Rx\tensor \rr x)
\hspace{14ex}
\all x Px,\, \;\ex y (Qy\tensor \qq y),\, \;\ex z(Rz\tensor \rr z)
\end{math}\end{center}
The (cut-free) \defn{encoding} of a cut $\cut{A}{\AA}$ is the
existentially closed tensor $\ex{\vec x}(A\tensor\AA)$ where $\ex{\vec
  x}$ denotes $\ex{x_1}\ldots\ex{x_n}$ for $x_1\ldots x_n$ the free
variables in $A$.\footnote{For definiteness, we assume a fixed order
  of the $x_i$.  The choice of this order is arbitrary.}
For technical convenience, and without loss of generality, we assume
the leaves of the encoding are identical to the leaves of the cut.
(For example, if $\cut{Px}{\pp x}$ is the cut whose leaves are $l$ and
$l'$, labelled $Px$ and $\pp x$, respectively, then the encoding is
$\ex x(Px\tensor \pp x)$ with the same leaves $l$ and $l'$, still
labelled $Px$ and $\pp x$, respectively.)
The \defn{encoding} of a cut sequent $\Delta$ is the
sequent $\encode{\Delta}$ obtained by replacing each cut by its
encoding.

Let $\theta$ be a linking on a cut sequent $\Delta$.
By our assumption that the leaves remain unchanged by encoding,
$\theta$ also constitutes a (cut-free) linking on
$\encode{\Delta}$.
The linking $\theta$ on $\Delta$ is \defn{correct} if $\theta$ is
correct (in the cut-free sense of Section~\ref{sec:correctness}) on $\encode{\Delta}$.
A \defn{unification net} (or \defn{unet} for short) on a cut sequent $\Delta$ is a correct
linking on $\Delta$.

\begin{theorem}[Quadratic-time correctness]\mbox{}\label{thm:cut-quadratic}\\
The correctness of a unification net can be verified in quadratic time.
\end{theorem}
\begin{proof}
  The cut-free case (Theorem~\ref{thm:quadratic}) carries over, since cut-free
  encoding is linear time.
\end{proof}

\subsection{Cuts beyond Girard's}

Our definition of cut is more general than Girard's. For example, consider the two unification nets below:
\twolinkings
The former has an analogue in Girard's setting, with four conclusions ($Pfx$, $\pp fx$, $Pfx$ and $\ex z\pp z$).
The latter unification net, slightly more compact with $y$ in place of $fx$ in the cut, has no analogue.

\subsection{Cut elimination}

A \defn{cut reduction} on a unification net is a subgraph rewrite of
any of the following forms:
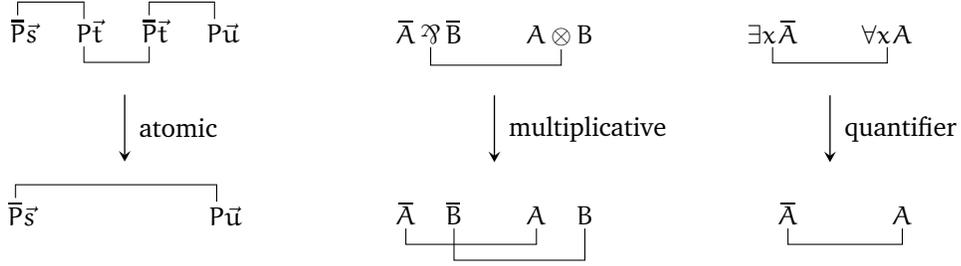
\begin{center}
\(
\begin{array}{c@{\h{10}}c@{\h{10}}c}
\renewcommand{\gap}{\h3}\begin{array}{c}
\Rnode a {\pp} \vec s \rule{0ex}{4ex}
\gap
\Rnode b P \vec t
\gap
\Rnode c {\pp} \vec t
\gap
\Rnode d P \vec u
\link a b
\link c d
\cutlink b c
\\[1ex]
\elim{4}{\text{atomic}}
\\[5ex]
\Rnode a {\pp} \vec s
\gap
\phantom{\Rnode b P \vec t
\gap
\Rnode c {\pp} y}
\gap
\Rnode d P \vec u
\link a d
\end{array}
&
\begin{array}{c}
\AA\mm2\Rnode x {{}\parr{}}\mm2 \BB
\rule{0ex}{4ex}
\h5
A\mm2\Rnode y {{}\tensor{}}\mm2 B
\cutlink x y
\\[1ex]
\elim{4}{\text{multiplicative}}
\\[5ex]
\Rnode a \AA
\phantom{{}\mm2\tensor\mm2{}}
\Rnode b \BB
\h5
\Rnode x A
\phantom{{}\mm2\parr\mm2{}}
\Rnode y B
\cutlink a x
\cutlinkarm b y {12pt}
\end{array}
&
\begin{array}{c}
\Rnode x {\ex x \AA}
\rule{0ex}{4ex}
\h5
\Rnode y {\all x A}
\cutlink x y
\\[1ex]
\elim{4}{\text{quantifier}}
\\[5ex]
\phantom{\ex x{}}\Rnode x \AA
\h5
\phantom{\all x{}}\Rnode y A
\cutlink x y
\end{array}
\end{array}\)\v3\end{center}
Here $\vec t$ denotes any sequence of terms.
We refer to the upper subgraphs as \defn{redexes}.
\begin{theorem}\label{thm:cut-elim}
  Reducing a cut from a unification net yields a unification net.
\end{theorem}
To prove this theorem we shall require auxiliary definitions and a key lemma concerning the reduction of a quantifier cut.

A \defn{cycle} in the graph of a linking is a subgraph $\swcc$ with vertex set $\{v_1,\ldots,v_n\}$ for $n\m3\ge\m3 2$, all $v_i$ distinct, and an edge (directed or undirected) between $v_i$ and $v_{i+1}$ for all $i$ (mod $n$),
such that if $n\m4=\m4 2$ then $\swcc$ contains two distinct edges\footnote{This can arise if there is a leap $\ex x\toedge \all y$ with $\all y$ the argument of $\ex x$.} between $v_1$ and $v_2$; $\swcc$ is a \defn{switching cycle} (\cf\ \cite{HG03}) if it contains at most one directed edge into any $\mm2\parr\mm2$ or $\forall$ vertex.
\begin{lemma}\label{lem:quantifier-reduction}
  Let $\theta$ be a unification net on
\:$\Gamma,\,\Rnode a {\ex x \AA\,}\:\Rnode e{\,\all x A}\ncbar[arm=2pt,nodesep=1pt,angle=-90] a e$\:
  and let the linking $\thetaprime$ on
\:$\Gamma,\,\Rnode a \AA\;\Rnode A A\ncbar[arm=2pt,nodesep=1pt,angle=-90] a A$\:
  be the result of reducing the distinguished quantifier cut. Then the graph of $\thetaprime$ has no switching cycle.
\end{lemma}
\begin{proof}
  The respective cut-free encodings are
  $$\encode{\Gamma},\,\h1\exists{x_1}\ldots\ex{x_n}(\ex x \AA\tensor \all x A)
  \h{12}
  \encode{\Gamma},\,\h1\exists{x_1}\ldots\ex{x_n}\ex x(\AA\tensor A)$$
  where the additional $\exists x$ in the latter is because,
  without loss of generality, $x$ is free
  in $A$: when $x$ is not free in $A$, the result is trivial since
  the quantifiers $\exists x$ and $\forall x$ in the redex cut \,$\Rnode a
  {\ex x \AA\,}\:\Rnode e{\,\all x A}\ncbar[arm=2pt,nodesep=1pt,angle=-90]
  a e$\, are vacuous, hence topologicaly inert.
  \newcommand{\xo}{\dot x}%
  \newcommand{\Ao}{\dot A}%
  To avoid bound variable conflict, rename $\forall x$ to $\forall\xo$:
  $$\encode{\Gamma},\,\h1\exists{x_1}\ldots\ex{x_n}(\ex x\AA\tensor \all\xo\Ao)$$
  where $\Ao$ is the result of substituting $\xo$ for $x$ in $A$.

  Let $\sigma$ be a unifier for $\theta$. Thus $\sigma=\openU \gets{z_1}{t_1},\ldots,\gets{z_k}{t_k},\gets x t\closeU$, where the $z_i$
  include the $x_j$.
  The term $t$ assigned to $x$ cannot contain $\xo$, or there
  would be a switching cycle due to the resulting precedence $\dep x
  \xo$, via the $\tensor$ of the encoding of $\cut{\ex x \AA}{\all x A}$\,:
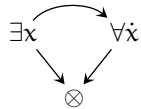
  $$
    \psset{arrows=->,nodesepA=3pt,nodesepB=2pt}
    \rule{0ex}{5ex}
    \begin{array}{ccc}
      \Rnode x{\exists x} & & \Rnode X{\forall\xo} \\[2.5ex]
      & \Rnode t \tensor
    \end{array}
    \nccurve[angleA=50,angleB=140,nodesepA=1pt] x X
    \ncline x t
    \ncline X t
  $$
  Let $t'_i$ be the result of substituting $t$ for $\xo$ in
  $t_i$.
  Define $\sigmaprime=\openU \gets{z_1}{t'_1},\ldots,\gets{z_k}{t'_k},\gets x t\closeU$.  This is a well-defined unifier for $\thetaprime$ since none
  of the $t'_i$ contains $\xo$ (because $t$ did not contain
  $\xo$).
  Without loss of generality, $\sigmaprime$ is an mgu.

  We must prove that $\graph(\thetaprime)$ has no switching cycle.
  Suppose $\swc$ was such.  We consider three subcases,
  according to whether there are zero, one or two (or more) leaps in
  $\swc$ which are not in $\graph(\theta)$.  Let $r$ and $\dual r$ be
  the root vertices of $A$ and $\AA$, respectively, and assume that
  $\Ao$ has the same vertices as $A$.  We assume $\graph(\theta)$
  and $\graph(\thetaprime)$ have the same vertices, except for the necessary difference
  around the tensors of the encodings of the two cuts:
  $$
    \psset{arrows=->,nodesepA=3pt,nodesepB=2pt}
    \newcommand{\rowdelta}{2.5ex}
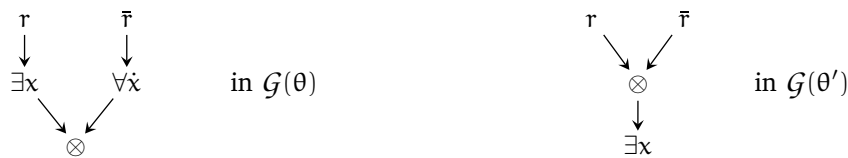
    \begin{array}{ccc}
      \Rnode r r    & &    \Rnode R{\dual r}  \\[\rowdelta]
      \Rnode x{\exists x}
         & &
     \Rnode X{\forall\xo}
     \\[\rowdelta]
      & \Rnode t \tensor
    \end{array}
    \ncline r x
    \ncline R X
    \ncline X t
    \ncline x t
    \h6
    \text{in }\graph(\theta)
    \h{20}
    \begin{array}{ccc}
      \Rnode r r    & &    \Rnode R{\dual r}  \\[\rowdelta]
      & \Rnode t \tensor \\[\rowdelta]
      & \Rnode x {\exists x}
    \end{array}
    \ncline r t
    \ncline R t
    \ncline t x
    \h4
    \text{in }\graph(\thetaprime)
    \h2
  $$
  For technical convenience we shall assume the vertex of $\exists x$ is
  the same in each case.

  Case: every leap of $\swc$ is in $\graph(\theta)$. Define a
  switching cycle in $\graph(\theta)$ from $\swc$ by, if necessary, re-routing a traversal
  of the tensor of the encoding of $\cut A\AA$ to the tensor of the encoding
  of $\cut{\all\xo\Ao}{\ex x A}$.

  Case: $\swc$ contains a single leap $\exists z_i\toedge\forall y$ which
  does not occur in $\graph(\theta)$. (The leap must be from
  an $\exists{z_i}$ since both $\sigma$ and $\sigmaprime$ assign $\gets x t$.)
  This is depicted below-left, where the dashed line represents one or
  more edges in $\swc$.
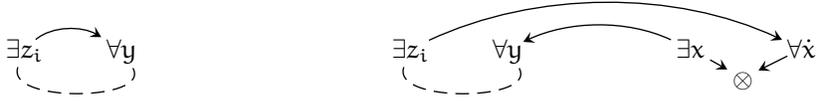
  \begin{center}\v3
    \(\psset{nodesep=2pt}
      \Rnode z {\exists z_i}
        \h5
      \Rnode y {\forall y}
      \nccurve[linestyle=dashed,angleA=-110,angleB=-60] z y
      \nccurve[angleA=40,angleB=150,arrows=->,nodesepA=0pt,nodesepB=2pt] z y
      \h{20}
      \Rnode z {\exists z_i}
        \h5
      \Rnode y {\forall y}
      \nccurve[linestyle=dashed,angleA=-110,angleB=-60] z y
      \h{12}
      \Rnode a {\exists x}
        \h3
      \rput(0,-.3){\Rnode t \tensor}
        \h{3.5}
      \Rnode b {\forall\xo}
      \nccurve[angleA=25,angleB=155,arrows=->,nodesepB=2pt,nodesepA=0pt] z b
      \ncline[arrows=->,nodesepA=0pt] b t
      \ncline[arrows=->,nodesepA=2pt] a t
      \nccurve[angleB=25,angleA=155,arrows=->,nodesepB=0pt] a y
      \raisebox{-3ex}{}
    \)
  \end{center}
  The leap $\exists{z_i}\toedge\forall y$ came from a precedence
  present in $\thetaprime$ but not in $\theta$.
  Such an additional precedence can arise only from the construction
  of $t_i'$ by substituting $t$ for $\xo$ in $t_i$, hence $y$
  must be in $t$, so $\dep x y$ is a precedence of
  $\theta$ (since $\gets x t$ in $\sigma$), with a corresponding leap
  $\exists x\toedge\forall y$ in
  $\graph(\theta)$.
  Since $t_i$ contains $\xo$, there is a precedence $\dep
  {x_i}{\xo}$, hence a leap $\exists z_i\toedge \forall\xo$ in
  $\graph(\theta)$.
  Thus we can construct a switching cycle in $\graph(\theta)$ as
  above-right.

  Case: there are two or more leaps  in
  $\swc\subseteq\graph(\thetaprime)$
  which are not present in $\graph(\theta)$, say (without loss of
  generality) $\exists z_1\toedge\forall y_1$ and
  $\exists z_2\toedge\forall y_2$.
  Either (a) the leaps are in the same direction around $\swc$, as
  shown below-left, or (b) they are in opposite directions, as
  below-right.
  \begin{center}\v1
    \(\psset{nodesep=2pt}
      \Rnode{v1} \exists z_1\h5 \Rnode {y1} {\forall y_1}
      \h9
      \Rnode{v2} \exists z_2\h5 \Rnode {y2} {\forall y_2}
      \nccurve[linestyle=dashed,angleA=-130,angleB=-40] {v1} {y2}
      \nccurve[linestyle=dashed,angleA=-20,angleB=-160] {y1} {v2}
      \nccurve[angleA=40,angleB=150,arrows=->,nodesepB=2pt]{v1}{y1}
      \nccurve[angleA=40,angleB=150,arrows=->,nodesepB=2pt]{v2}{y2}
      \h{12}
      \Rnode{v1} \exists z_1\h5 \Rnode {y1} {\forall y_1}
      \h9
      \Rnode {y2} {\forall y_2}
      \h5
      \Rnode{v2} \exists z_2
      \nccurve[linestyle=dashed,angleA=-140,angleB=-40,nodesepB=6pt] {v1} {v2}
      \nccurve[linestyle=dashed,angleA=-20,angleB=-160] {y1} {y2}
      \nccurve[angleA=40,angleB=150,arrows=->,nodesepB=2pt]{v1}{y1}
      \nccurve[angleA=140,angleB=30,arrows=->,nodesepB=0pt]{v2}{y2}
      \raisebox{-6ex}{}
    \)
  \end{center}
  Reasoning for each $\exists z_i$ and $\forall y_i$ as in the previous
  subcase for $\exists z_i$ and $\forall y$, we have leaps $\exists x\toedge
  \forall y_i$ and $\exists z_i\toedge\forall x$.  Thus, in
  $\graph(\theta)$, if (a), we can construct the switching cycle
  below-left,
  and if (b), the switching cycle below-right.
  \begin{center}\vv{.5}
    \(\psset{nodesep=2pt}
      \phantom{\Rnode{v1} \exists z_1\h5 }\Rnode {y1} {\forall y_1}
      \h9
      \Rnode{v2} \exists z_2\phantom{\h5 \Rnode {y2} {\forall y_2}}
      \below{\h4\Rnode{ex}{\exists x}}{y1}{5ex}
      \putright{\Rnode{fx}{\forall\xo}}{ex}{6ex}
      \nccurve[linestyle=dashed,angleA=-20,angleB=-160] {y1} {v2}
      \ncline[linestyle=none]{ex}{fx}\nbput[labelsep=3ex]{\Rnode t \tensor}
      \ncline[arrows=->]{v2}{fx}
      \ncline[arrows=->]{ex}{y1}
      \ncline[arrows=->,nodesepB=.5pt]{fx}{t}
      \ncline[arrows=->,nodesepB=.5pt]{ex}{t}
      \h{12}
      \phantom{\Rnode{v1} \exists z_1\h5} \Rnode {y1} {\forall y_1}
      \h9
      \Rnode {y2} {\forall y_2}
      \phantom{\h5 \Rnode{v2} \exists z_2}
      \nccurve[linestyle=dashed,angleA=-20,angleB=-160] {y1} {y2}
      \nbput[labelsep=4.5ex]{\Rnode{x}{\exists x}}
      \ncline[arrows=->]{x}{y1}
      \ncline[arrows=->]{x}{y2}
      \raisebox{-7ex}{}
    \)
  \end{center}
\end{proof}
With Lemma~\ref{lem:quantifier-reduction} in hand, we can now prove that reducing a cut from a unification net yields a unification net (Theorem~\ref{thm:cut-elim}).
\begin{proofof}{Theorem~\ref{thm:cut-elim}}
  Each of the three reductions preserves the difference between the
  number of links and the number of $\tensor\mm2$s and cuts,
  thus (see \eg\ \cite[\S4.7.1]{HG05}) to confirm a switching is a
  tree we need only check that it is acyclic.  Acyclicity of all
  switchings is equivalent (see \eg \cite[\S4.7.2]{HG05}) to there
  being no switching cycle in the graph of the linking.

  Atomic case: an atomic cut reduction takes $\theta$ on $\;\Gamma,\,
  \Rnode x {P\vec t}\;\,\Rnode y {\pp\vec
    t}\ncbar[arm=2pt,nodesep=1pt,angle=-90] x y\;$ to $\thetaprime$ on
  $\Gamma$.  Let $\sigma$ be mgu for $\theta$, which by definition
  equalizes the term sequences $\vec s$ and $\vec t$ (due to the left
  link in the redex) and $\vec t$ and $\vec u$ (due to the right link).  By
  transitivity $\sigma$ equalizes $\vec s$ and $\vec u$, thus the
  restriction $\sigmaprime$ of $\sigma$ to existential variables in $\thetaprime$ is
  an mgu for $\thetaprime$.  A switching cycle $\swc$ of
  $\graph(\thetaprime)$ induces a corresponding switching cycle
  $\swcc$ of $\graph(\theta)$: since $\sigmaprime$ is a restriction of $\sigma$,
  every leap in
  $\swc$ determines a corresponding leap in $\swcc$;
  if $\swc$ passes through the new link \psscalebox{.9 .9}{$\Rnode x
    \pp\vec s\h1\Rnode y P\vec u
    \ncbar[arm=1pt,nodesepB=1pt,nodesepA=2pt,angle=90] x y$}, in
  $\swcc$ go instead between $\pp\vec s$ and $P\vec u$ via the cut
  \psscalebox{0.9 0.9}{$\Rnode x {P\vec t}\h1\Rnode y {\pp\vec
      t}\ncbar[arm=2pt,nodesep=1pt,angle=-90] x y$} (\ie, via the
  $\tensor$ of its encoding $\ex{\vec x}(P\vec t\tensor\pp\vec t)$).

  Multiplicative case: a multiplicative cut reduction takes
  $\theta$ on
  \:$\Gamma,\,
  \AA\mm2\Rnode{p}{\parr}\mm2\BB
  \;
  A\mm2\Rnode{t}{\tensor}\mm2B
  \ncbar[arm=2pt,nodesep=1pt,angle=-90]
  t p$\:
  to $\thetaprime$ on
  \:$\Gamma,\,
  \Rnode{a}{\AA}\;\Rnode{aa}{A}\ncbar[arm=2pt,nodesep=1pt,angle=-90]
  a {aa},\,
  \Rnode{b}{\BB}\;\Rnode{bb}{B}\ncbar[arm=2pt,nodesep=1pt,angle=-90]
  b {bb}$\:.
  There is no change in mgu, precedences or leaps, so the reasoning
  of the usual multiplicative case \cite{Gir96} goes through directly.

  Quantifier case: Lemma~\ref{lem:quantifier-reduction}.
\end{proofof}

\begin{theorem}[Strong normalization]
Every sequence of cut reductions terminates.
\end{theorem}
\begin{proof}
  Each reduction reduces the size of the cut sequent.
\end{proof}

\begin{theorem}[confluence]
  Cut reduction is confluent.
\end{theorem}
\begin{proof}
  Reduction is local.
\end{proof}
\begin{theorem}[Linear time cut elimination]\mbox{}\label{thm:cut-elim-linear-time}\label{thm:cut-elim-linear}\\
Eliminating all cuts from a unification net $\theta$ takes time linear in the size of $\theta$.
\end{theorem}
\begin{proof}
  Cut elimination is strongly normalizing, confluent and local.
\end{proof}

\section{Surjectivity Theorem with cut}

The principle that a cut is akin to an existentially closed tensor
\begin{equation*}
  \h4\cut{\,A\,}{\:\AA\:} \h4 \approx \h4 \ex {\vec x} (A\tensor \AA)
\end{equation*}
yields surjectivity essentially for free: view each cut rule as a tensor rule
followed by zero or more existential rules (one per free variable $x$
in $A$), appeal to the cut-free surjectivity theorem (Theorem~\ref{thm:surjectivity}), then observe that $\exists$-rules can always be commuted upwards (so that encoded $\exists$-rules can be brought up to immediately below their encoded tensor rule, ready for conversion into a cut rule). This argument is detailed and formalized below.

\subsection{Extended cut rule}\label{sec:extended-cut-rule}

To streamline the formalization, corresponding to the existential closure we extend the cut rule by retaining the cut formulas in the conclusion and allowing a substitution $\sigma$ of the cut formulas $A$ and $\AA$ in the hypotheses:
\begin{center}\begin{math}
\cutrule
  {\pfseq{\Gamma\m1,\mm2A\sigma}}
  {\pfseq{\AA\sigma,\mm2\Delta}}
  {\Gamma\m1,\cut{A}{\AA},\mm2\Delta}
\v1\end{math}\end{center}
Here $\sigma$ is any substitution of terms for free variables in $A$
and $\AA$.
Two examples are below:
\twoseqs
In the left example $\sigma$ is trivial, $\sigma=\unifier{\gets x x}$,
so that $A\sigma=A$,
and in the right example $\sigma=\unifier{\gets y {fx}}$.
We write \fomllplus for this extended sequent calculus.

The translation of a proof to a linking is unchanged from the cut-free
case: trace the atoms down from the axioms onto the conclusion.
For example, the two proofs above translate to the linkings below.
\twolinkings

\subsection{Surjectivity Theorem}

\begin{theorem}[surjectivity]\mbox{}\label{thm:cut-surjectivity}\\
The translation from \textnormal{\fomllplus} proofs to linkings is a surjection onto unification nets.
\end{theorem}
\begin{proof}
  Let $\theta$ be a unification net.  Replace each cut $\cut{A}{\AA}$ by the corresponding
  existentially closed tensor $\ex{\vec x}(A\tensor\AA)$, then apply
  the cut-free surjectivity theorem (Theorem~\ref{thm:cut-free-surjectivity}) to obtain a cut-free proof
  $\tenex\Pi$. For each cut $\acut$ between formulas with $n$ free
  variables, commute (if necessary) the $n$ $\exists$-rules associated
  with $\acut$ upwards in the proof to be adjacent to the tensor rule
  associated with $\acut$.  Form $\Pi$ by replacing the tensor rule
  and $n$ $\exists$-rules of $\acut$ by a single cut rule, for each
  cut $\acut$.  By induction, $\Pi$ translates to
  $\theta$.

  Conversely, suppose $\Pi$ translates to $\theta$, with
  concluding cut sequent $\Delta$.  Form $\tenex\Pi$ with concluding
  cut-free sequent $\tenex\Delta$ by replacing each cut rule in $\Pi$
  by a tensor rule followed by $\exists$-rules.  Thus $\tenex\Delta$
  is the result of replacing each cut $\cut{A}{\AA}$ in $\Delta$ by
  the corresponding existentially closed tensor $\ex{\vec
    x}(A\tensor\AA)$.  Since the transformation of $\Pi$ to
  $\tenex\Pi$ does not change the tracking of atoms down the proof
  onto the conclusion, $\theta$ on $\tenex\Delta$ can be viewed as a
  translation of $\tenex\Pi$.  Via the cut-free surjectivity
  theorem, $\theta$ on $\tenex\Delta$ is a cut-free unification net.
  Since the definition of a unification net with cuts was formulated
  by encoding cuts as closed existential tensors, $\theta$ on
  $\Delta$ is a unification net.
\end{proof}

\subsection{Examples illustrating surjectivity with cut}

We illustrate Theorem~\ref{thm:cut-surjectivity} with the following pair of linkings (copied
from the end of Section~\ref{sec:extended-cut-rule}):
\twolinkings
First replace each cut by its encoding:
\begin{center}\v2\begin{math}
\renewcommand{\gap}{\hspace{3ex}}\psset{labelsep=4ex}
\newcommand{\arm}{5}
\Rnode{p}{P} fx\h{.5}\gap \ex x(\Rnode{pp}{\pp} fx\tensor \Rnode{p1}{P} fx)\h{1.5}\gap\ex z \Rnode{pp1}{\pp} z
\ncbar[angle=90,arm=\arm pt,nodesepA=1pt,nodesepB=2pt]{p}{pp}%
\ncbar[angle=90,arm=\arm pt,nodesepA=1pt,nodesepB=2pt]{p1}{pp1}
\h{15}
\Rnode{p}{P} fx\h{.5}\gap \ex y(\Rnode{pp}{\pp} y\tensor \Rnode{p1}{P} y)\h{1.5}\gap\ex z \Rnode{pp1}{\pp} z
\ncbar[angle=90,arm=\arm pt,nodesepA=1pt,nodesepB=2pt]{p}{pp}%
\ncbar[angle=90,arm=\arm pt,nodesepA=1pt,nodesepB=2pt]{p1}{pp1}
\end{math}\end{center}
Apply the cut-free surjectivity theorem:
\twoencodedseqs
Finally. replace each $\tensor$-rule-$\exists$-rule pair by an extended cut rule:
\twoseqs

\section{Unification nets resolve the exponential blow-ups of Girard nets}%
\label{sec:girard-net-complexity-issues}\label{sec:girard-blow-up}%
\label{sec:complexity-issues}\label{sec:complexity}\label{sec:girard-cut-elim-exponential}%

Redundant existential witnesses cause Girard's \fomll nets to suffer from two major complexity issues,
absent from MLL nets:
\begin{itemize}
\myitem 1
\defn{Exponential computation blow-up}:
\emph{Cut elimination is non-local and both exponential-time and exponential-space.}
\hspace{-2pt}Reducing\hspace{-.35pt} a \hspace{-.35pt}quantifier cut in \hspace{-.35pt}a\hspace{-.35pt} Girard net substitutes witnesses globally throughout the net: see Figure~\ref{fig:cut-elim-comparison} on page~\pageref{fig:cut-elim-comparison} (top half) for an illustration.
Chaining together a series of such substitutions, each duplicating a term, results in exponential growth of a Girard net during cut elimination; see Appendix~\ref{sec:girard-cut-elim-blowup}, especially Figure~\ref{fig:cut-elim-complexity-comparison}, for an example.

This is a severe regression from MLL nets, whose cut elimination is local and linear-time.
\myitem 2
\defn{Exponential size blow-up}:
\emph{Some sequents demand exponentially large cut-free Girard nets.}
The size of the smallest cut-free Girard net on a sequent grows exponentially with the size of the sequent.
In proof complexity terminology \cite{CR79}, cut-free Girard nets are not \emph{polynomially bounded}:
there is no
polynomial $p$ against which every provable sequent $\Gamma$ has a \emph{short} cut-free Girard net, \ie, a cut-free Girard net $G$ such that $|G|\le p(|\Gamma|)$, where $|X|$ is the size of $X$.
\begin{figure*}\begin{center}%
\newcommand\scalevalue{.89}
\subbox{A minimal cut-free proof}{\scalebox{\scalevalue}{\shortestBlowupProofFour}}
\v9
\subbox{The unique (hence minimal) cut-free Girard net}{\scalebox{\scalevalue}{\blowupGirardNet}}
\v9
\subbox{The unique unification net}{\scalebox{\scalevalue}{\(\blowupUnet\)}}
\v5
\end{center}\caption{\label{fig:size-blow-up}\label{fig:blow-up-eg}\label{fig:blow-up-comparison}Illustrating exponential size blow-up in cut-free \fomll proofs and cut-free Girard nets. The top sub-figure shows a minimal cut-free \fomll proof of $\;\protect\blowupSequentInline\;$, where $P$ is a $4$-ary predicate and $\protect\idot$ is an infix binary function symbol. The mid sub-figure shows the unique (hence minimal) cut-free Girard net. Due to explicit existential witnesses, both have an axiom rule/link which is exponentially larger than the sequent:
in the general case with $P$ an $n$-ary predicate (see Appendix~\ref{sec:size-blow-up}), the axiom rule/link contains
$2(2^n-1)$ occurrences of the constant $c$ (here $2(2^4-1) = 30$). In contrast, the cut-free unification net grows only linearly with $n$. Since this example has no multiplicative connective, it also shows that first-order additive proof nets with explicit witnesses suffer from the same exponential blow-up. Indeed, this example shows that quantifier-only sequent calculus suffers the blow-up.}\end{figure*}
An illustrative example of size blow-up is shown in Figure~\ref{fig:blow-up-eg}, and detailed in Appendix~\ref{sec:size-blow-up}.\footnote{Section~\ref{sec:exponential-compression} showed an alternative example, involving a par $\parr$. The example in Appendix~\ref{sec:size-blow-up} is more general since it is quantifier-only.}

This represents a major deficiency in Girard's \fomll nets because a polynomially-bounded variant of cut-free \fomll sequent calculus exists \cite{LS94}, placing \fomll in the complexity class NP.
It is also a severe regression from MLL nets, since cut-free MLL nets are \emph{linearly bounded}: there exists a multiplier $k$ against which every provable MLL sequent $\Gamma$ has a short cut-free MLL net, \ie, a cut-free MLL net $\theta$ such that $|\theta|\le k|\Gamma|$.
In fact, since every cut-free MLL net $\theta$ on $\Gamma$ is just an axiom linking on $\Gamma$, cut-free MLL nets are \emph{linearly sized}:
there exists a multiplier $k$ such that, for every provable sequent $\Gamma$, every cut-free MLL proof net on $\Gamma$ satisfies $|\theta|\le k|\Gamma|$.
\end{itemize}
The size blow-up example in Figure~\ref{fig:blow-up-eg} is particularly interesting because it has no multiplicative connective, thus also shows that Girard's first-order additive nets \cite{Gir96} with explicit existential witnesses suffer from the same size blow-up; indeed, it shows that even quantifier-only sequent calculus has the blow-up.

Unification nets resolve both of the complexity issues with Girard nets:
\begin{itemize}
\myitem{1}
\emph{Cut elimination
is local and linear-time}.
As in an MLL proof net, a cut reduction in a unification net is a purely local graph rewrite
and
the time complexity of eliminating all cuts from a unification net $\theta$ is linear
(Theorem~\ref{thm:cut-elim-linear-time}).
\myitem{2}
\emph{Cut-free unification nets are linearly-sized}.
Like a cut-free MLL proof net, a cut-free unification net is a linking on a sequent. Thus cut-free unification nets are linearly bounded, since every provable \fomll sequent $\Gamma$ has a cut-free unification net of size  $O(|\Gamma|)$. Furthermore, like cut-free MLL proof nets, cut-free unification nets are \emph{linearly sized}: \emph{every} cut-free unification net on $\Gamma$ has size $O(|\Gamma|)$.
\end{itemize}

\section{Factorizing the surjection through Girard nets and unification calculus}

In this section we factorize the surjection $\translate{-}$ from cut-free \fomll proofs onto cut-free unification nets defined in Section~\ref{sec:translation} in two different ways: the two outer paths of the commuting square described in Section~\ref{sec:beyond-sequentialization} and depicted in Figure~\ref{fig:beyond-sequentialization}.

In Section~\ref{sec:girard-factorization} we pass through cut-free Girard nets. The first leg thus eliminates redundant rule orderings, and the second leg eliminates redundant existential witnesses.
In Section~\ref{sec:ucalc-factorization} we do the opposite, first eliminating redundant witnesses, then eliminating redundant rule orderings, via an artificial abstraction of \fomll sequent calculus without explicit existential witnesses which we call \emph{unification calculus} (whose sole purpose is to obtain such a factorization).

\subsection{Factorization through cut-free Girard nets}\label{sec:girard-factorization}

Define the translation of a cut-free Girard net to a linking in the same manner as the translation of a cut-free proof: track the axiom links down onto the concluding formulas.
\begin{lemma}\label{lem:translate-girard}
  Every cut-free Girard net translates to a cut-free unification net.
\end{lemma}
\begin{proof}
  Let $G$ be a cut-free Girard net on $\Gamma$. Thus the
  concluding formulas of $G$ are in bijection with the formulas
  in $\Gamma$.
  Let the linking $\theta$ on $\Gamma$ be the translation of $G$.
  We must prove that
  $\theta$
  is correct. First, $\theta$ is unifiable: define the unifier $\sigma$
  by $\gets x t$ for each $\exists$-link $\frac{A\unifier{\gets x
      t}}{\ex x A}$ which is non-vacuous (\ie, such that $x$ occurs free in $A$).
  The assignment $\sigma$
  is a well-defined unifier since each axiom link of $G$ goes
  between atoms $P(t_1,\ldots, t_n)$ and $\pp (t_1,\ldots, t_n)$ (with
  identical term sequences $(t_1,\ldots, t_n)$, in contrast to
  unification nets for which we require only that the predicate
  symbols be dual).  Every switching of $\theta$ induces a
  corresponding switching of $G$: the choice of left/right into a
  $\parr$ in $\graphof\theta$ yields a choice for the corresponding formula
  in $G$, and a choice of edge $\exists x\toedge\forall y$ into
  the vertex $\forall y$ in $\graphof\theta$, from the precedence $\dep x y$
  associated with $x\mapsto t$ in $\sigma$ for some term $t$ containing
  $y$, determines a corresponding jump in $G$ from the hypothesis
  of the $\exists$-link
  $\frac{A\unifier{\gets x t}}{\ex x A}$ to the conclusion
  of the $\forall$-link $\frac{B}{\all y B}$.  Thus $\theta$ is
  correct, since any non-tree switching of $\theta$ induces a
  non-tree switching of $G$.
\end{proof}
\begin{theorem}
  The translation from cut-free Girard nets to cut-free unification nets is surjective.
\end{theorem}
\begin{proof}
  Let $\theta$ be a cut-free unification net on $\Gamma$ with mgu
  $\sigma$.  We unfold $\theta$ into a cut-free Girard net $G$ by
  working upwards from each root of $\Gamma$.

  We first unfold each
  formula $A$ in $\Gamma$ to a fragment of $G$ with conclusion
  $A$.  Define the unfolding $\unfold A$ of a formula $A$ as the
  following tree, alternating between Girard-links and formulas, whose root,
  called the \emph{conclusion} of $\unfold A$, is the formula
  $A$.
  If $A$ is an atom, then $\unfold A=A$.  If $A=B\tensor C$, define
  $\unfold A$ as $\frac{\unfold B\h3\unfold C}{B\mm1\tensor\mm1 C}$, the
  disjoint union of $\unfold B$ and $\unfold C$ and a $\tensor\mm2$-link
  taking the conclusions $B$ and $C$ of $\unfold B$ and $\unfold C$ as
  hypotheses and $A=B\tensor C$ as its conclusion.
  If $A=B\parr C$ define $\unfold A$ analogously, with $\parr$ in
  place of $\tensor\mm2$.
  If $A=\all x B$, define $\unfold A$ as $\frac{\unfold B}{\all x B}$,
  the tree $\unfold B$ with a $\forall$-link taking the conclusion $B$
  of $\unfold B$ as its hypothesis and $A=\all x B$
  as its conclusion.
  If $A=\ex x B$, define $\unfold A$ as
  $\frac{\unfold B[\gets x t]}{\ex x B}$, where $t$ is the the term
  assigned to $x$ by the mgu $\sigma$ and $\unfold B[\gets x t]$ is the
  result of substituting $t$ for $x$ in every formula in
  $\unfold B$.
  Thus $\unfold A$ is the tree $\unfold B[\gets x t]$ and a $\exists$-link whose hypothesis is the conclusion $B[\gets x t]$ of
  $\unfold B[\gets x t]$ and whose conclusion is $A=\ex x B$.
  Define the unfolding $\unfold\Gamma$ of $\Gamma$ as the disjoint
  union of the unfoldings of its formulas.
  By induction, the atoms in $\unfold\Gamma$ are in
  bijection with the leaves of $\Gamma$.

  Define $G$ from $\unfold\Gamma$ as follows: for each link
  $\{l,\ll\}$ in $\theta$ between a pair of leaves in $\Gamma$, add a
  axiom-link \;$\overline{\scriptstyle\:\unfold
    l\;\;\;\;\unfold{l}\m2'}$\; between the corresponding pair of
  atoms in $\unfold\Gamma$.

  We must show that $G$ is a cut-free Girard net.
  First we prove that the atoms either end of each axiom-link in $G$ are strictly dual: if
  one end is $P(s_1,\ldots, s_n)$ the other is $\pp (s_1,\ldots, s_n)$
  (identical term sequences).  Each axiom-link $\unfold L$ in
  $G$ was derived from a link $L$ in $\theta$ between
  leaves labelled $P(t_1,\ldots, t_n)$ and $\pp (t'_1,\ldots, t'_n)$.
  Since the mgu $\sigma$ equalizes corresponding term sequences, we
  have $t_i\sigma=t'_i \sigma$ where $t\sigma$ denotes the result of substituting
  existential variables in $t$ according to $\sigma$.  The same
  substitutions of existential variables applied during the
  construction of the unfolded formulas in $\unfold\Gamma$, hence the axiom-link $\unfold L$ in $G$ is between $P(t_1\sigma,\ldots, t_n\sigma)$ and
  $\pp (t_1\sigma,\ldots, t_n\sigma)$, which are strictly dual (by definition of unifiability of $\theta$ with mgu $\sigma$).

  We must show that $G$ has no switching cycle.  Without
  loss of generality every jump from a formula $A$ with an
  eigenvariable $x$ to the conclusion $\all x B$ of the corresponding
  $\forall x$-link can move to an edge from either (a) the hypothesis
  $B$ of the $\forall x$-link, or (b) the hypothesis $C$ of a
  $\exists$-link: if $A$ is above $\all x B$, choose (a), following
  the path between $A$ and $B$; otherwise $A$ must have a
  $\exists$-link below it which prevents the eigenvariable $x$ from
  being free in the conclusion,
  and we choose (b), following the path between $A$ and the hypothesis
  $C$ of the $\exists$-link.
  With switchings so transformed, there is a one-to-one correspondence
  between switchings of $G$ and switchings of $\theta$.

  By induction, since the translation from cut-free Girard nets to cut-free unification nets (Lemma~\ref{lem:translate-girard})
  uses the converse steps to those above (removing rather than adding witnesses), $G$ translates to $\theta$.
\end{proof}

The composite of the two surjections, from cut-free \fomll to cut-free Girard nets, then on to cut-free unification nets, is the translation $\translate{-}$ from cut-free \fomll proofs to unification nets defined in Section~\ref{sec:translation} (the diagonal of Figure~\ref{fig:beyond-sequentialization}): both surjections are defined by the same tracking of dual predicate symbols down from axioms.

\subsection{Factorization through cut-free unification calculus}\label{sec:ucalc-factorization}\label{sec:unification-calculus}

Let $\Pi$ be a proof of $\Gamma$. Define its \defn{skeleton} as $\Pi\id$ for $\id$ the identity on the non-vacuous existential variables of $\Pi$.
(Re-witnessing $\Pi\sigma$ was defined in Section~\ref{def:re-witnessing}.)
Define a \defn{unification proof} of $\Gamma$ as a skeleton of a proof of $\Gamma$, and define \defn{unification calculus} as the \fomll proof system comprising unification proofs.
In general, a skeleton will not be a well-formed sequent calculus proof since its axioms can be ill-formed, with non-dual predicates $Pt_1\ldots t_n$ and $\pp u_1\ldots u_1$ (just like the links of a unification net).

A unification proof $U$ can be verified in polynomial time: check the unifiability of the (ill-formed) axioms; if unifiable, let $\sigma$ be a most general unifier; verify that $U\sigma$ is a well-defined \fomll proof. Naively this is exponential time (since constructing the mgu is exponential time and space, in general), however, we can use the same technique as in the quadratic-time complexity proof (Theorem~\ref{thm:quadratic}) to build a sequential mgu, then lazily confirm that every rule of $U\sigma$ would be a well-formed rule, were we to actually carry out the substitution $\sigma$ at each rule (verifying that the predicates in every axiom link become dual, and the $\forall$ rule side condition on free variables holds).

That the surjection from cut-free \fomll to cut-free unification nets factorizes through cut-free unification calculus is self-evident: instead of extracting links directly from a proof $\Pi$, first take the skeleton $\Pi\id$ (dropping explicit witnesses) then extract the links from $\Pi\id$ (whose axiom rules are the same as $\Pi$, only with some terms substituted).

We can define cut elimination on unification calculus by mimicking sequent calculus cut elimination, without the explicit witnesses. Since witnesses are absent, cut elimination is polynomial-time.

\appendix
\appendixpage
\addappheadtotoc

\section{MLL sequent calculus cut elimination is non-local and at best quadratic}\label{sec:mll-seq-calc-cut-elim}

Let $\Pi_n$ be the following cut-free proof, with $n-1$ tensor rules:
\[\hspace{-15ex}\quadraticSeqCalcPiSubn\]
Let $\Pi_n^*$ be $\Pi_n$ followed by $n$ cut rules against axioms, as follows:
\quadraticSeqCalcCutElimEg
To reduce a cut we must first commute $n-1$ rules to raise the cut rule and ready the redex. Since redexes are blocked in this manner,
MLL sequent calculus is not local.
\begin{proposition}
  MLL sequent calculus cut elimination is at best quadratic time.
\end{proposition}
\begin{proof}
The number of commutations required to eliminate all cuts from $\Pi_n^*$ increases quadratically in $n$, while the size (number of rules) of $\Pi_n^*$ increases linearly in $n$.
\end{proof}

\section{Additional redundancy in Girard's 1996 variant}\label{sec:additional-redundancy}

The \cite{Gir96} variant of Girard's \fomll proof nets \cite{Gir87,Gir88,Gir91q} introduces additional redundancy not present in sequent calculus, since it
annotates every $\exists$-link with an existential witness a term $t$, even when the quantifier is vacuous (binding no variable).
For example,
{\figadditionalredundancy}Figure~\ref{fig:girard-additional-redundancy} (left) shows the unique cut-free proof of \:$\additionalproved$\:.
The centre of the figure shows the infinitely family of cut-free \cite{Gir96} nets $G_t^{96}$, one per term $t$.
The unique cut-free unification net is shown on the right.

\section{Exponential computation blow-up in Girard nets}\label{sec:girard-cut-elim}\label{sec:girard-cut-elim-blowup}\label{sec:computation-blow-up}

As in \fomll sequent calculus, cut elimination of Girard nets is non-local due to global
substitution of terms during cut elimination. The top half of Figure~\ref{fig:cut-elim-comparison}
showed an example.
Cut elimination is exponential time and space because
exponential growth can arise from iterating substitutions such as
$x\,\shortmapsto\, fxx$ which duplicate terms.
For example, let $G$ be the Girard net
\begin{center}\v2
  \begin{math}
    \GexponentialChunk{left1}{right1}{\pp}{P}{\pp}{P}{f}{x}{\gx}
  \end{math}
\v2\end{center}
where $f$ is a binary function symbol.
Let $G^n$ be the result of cutting $n$ copies of $G$ against one
another, using $n{}-{}1$ cuts, each between copies of
$\forall x \, \pp x$ and $\exists \gx\,P \gx$ (renaming variables as
necessary to preserve uniqueness of eigenvariables).
\figexp{}%
For example, the first Girard net shown in Figure~\ref{fig:cut-elim-complexity-comparison} is
$G^4$.
The cut-free normal form $|G^n|$ of $G^n$ contains a term with $2^n$
occurrences of $x$, illustrating how cut elimination blows up the size
of a Girard net exponentially.
The second Girard net shown in Figure~\ref{fig:cut-elim-complexity-comparison} illustrates $|G^4|$; observe the right-most term with $2^4\m3=\m3 16$ occurrences of $x$.
For comparison, the bottom half of Figure~\ref{fig:cut-elim-complexity-comparison} shows the corresponding cut elimination for unification nets, which is local and linear-time.

\section{Exponential size blow-up in Girard nets}\label{sec:size-blow-up}

The following example demonstrates the size blow-up of Girard nets.
In fact, since it has no multiplicative connective, it also shows that Girard's first-order additive nets \cite{Gir96} with explicit existential witnesses suffer from the same size blow-up; indeed, it shows that even quantifier-only sequent calculus has the blow-up.
Let $\idot$ be an infix binary function symbol, and define $\alpha_n$ and $\beta_n$ as the $n$-ary predicates
\[\newcommand\comsep{\h{.6},\h{1.2}}
  \begin{array}{rcr@{(\,}c@{\comsep}c@{\comsep}c@{\comsep}c@{\comsep\ldots\;\;)}}
    \alpha_n &=& \pp  &  x_1  & x_1 \idot x_1 &      x_3     & x_3 \idot x_3 \\
    \beta_n  &=&  P   &   c   &       x_2     & x_2\idot x_2 &      x_4
  \end{array}
\]
where $c$ is a constant (nullary function), each $x_i$ is a variable, and the ellipsis terminates at the $n\nth$ argument of the predicate.
Define the two-formula sequent
\[
\Gamma_n
\h4=\h4
\exists x_1\exists x_3\exists x_5\ldots \alpha_n\;\com\;\; \exists x_2\exists x_4\exists x_6\ldots \beta_n
\]
where neither quantifier series extends beyond $\exists x_n$.
The unique cut-free Girard net $G_n$ on $\Gamma_n$ has one axiom link and $n$ $\exists$-links,
and grows exponentially with $n$ because its axiom link contains $2(2^n-1)$ copies of the constant $c$.
For example, here are the axiom links of $G_1$, $G_2$, $G_3$ and $G_4$, respectively:
\begin{center}\vspace{1ex}\scalebox{.93}{\(
\hh{5}\begin{array}{c}
\blowupAxiomRuleOne   \\[3ex]
\blowupAxiomRuleTwo   \\[3ex]
\blowupAxiomRuleThree \\[3ex]
\blowupAxiomRuleFour  \\[2ex]
\end{array}\hh{5}\)}\end{center}
Figure~\ref{fig:size-blow-up} depicts $G_4$ in full.

\section{Intuition for cuts as existentially closed tensors}\label{sec:cutex-intuition}

This appendix provides some proof-theoretic intuiton behind treating
cut as an existentially closed tensor:
\begin{center}\begin{math}
  \h4\cut{\,A\,}{\:\AA\:} \h4 \approx \h4 \ex {\vec x} (A\tensor \AA)
\end{math}\v1\end{center}
Consider the proof below-left, with a conventional cut rule (without
the explicit cut $\cut{A}{\AA}$ after the rule).
\begin{center}\vv2\hh4\begin{math}
\forallrule
  {
    \cutrule
        {
          \axiomrule{\pp x\com P x}
        }
        {
          \existsrule
              {\axiomrule{\pp x\com  Px}}
              {\pp x\com \ex z Pz}
        }
        {\pp x\com \ex z Pz}
  }
  {\all x \pp x\com \ex z Pz}
\h5
\forallrule
  {
    \tensorrule
       {
         \axiomrule{\pp x\com P x}
       }
       {
         \existsrule
            {\axiomrule{\pp x\com P x}}
            {\pp x\com \ex zP z}
       }
       {\pp x\com  P x\tensor \pp x\com  \ex zP z}
  }
  {\all x \pp x\com P x\tensor \pp x\com \ex zP z}
\h5
\forallrule
  {
    \existsrule
       {
         \tensorrule
            {\axiomrule{\pp x\com P x}}
            {
               \existsrule
                  {\axiomrule{\pp x\com P x}}
                  {\pp x\com \ex zP z}
            }
            {\pp x\com P x\tensor \pp x\com \ex zP z}
       }
       {\pp x\com \exists y(P y\tensor \pp y)\com \ex zP z}
  }
  {\all x \pp x\com \exists y(P y\tensor \pp y)\com \ex zP z}
\end{math}\hh3\v1\end{center}
Were we to naively replace the cut rule by a $\tensor\mm2$-rule,
following the standard propositional recipe of encoding a cut as a
tensor, we would obtain the ill-formed proof above-center: the
$\forall$-rule fails the side condition on free variables because $x$
is free in $P x\tensor \pp x$.  The conventional cut rule (in the
proof above-left) hides the free $x$ in the cut formulas $P x$ and
$\pp x$, thereby enabling the subsequent $\forall$-rule.  By adding an
$\exists$-rule after the $\tensor\mm2$-rule, as above-right, we
achieve a similar hiding of $x$ to enable the $\forall$-rule.

\small
\bibliographystyle{alpha}
\bibliography{/Users/dominic/0/tex/bib/main}

\end{document}